\documentclass[12pt]{article}
\usepackage[margin =1in]{geometry}
\usepackage{amssymb,amsthm}
\usepackage{amsmath}
\usepackage{enumerate}
\usepackage{enumitem}
\usepackage{hyperref}
\usepackage{cite}
\usepackage{color}
\usepackage{subcaption}
\usepackage{mathtools}
\usepackage{bbm}

\numberwithin{equation}{section}

\newcommand{\localmu}{\mu}
\newcommand{\E}{{\bf E}}
\newcommand{\R}{{\bf R}}

\newcommand{\localconstant}{C}

\newcommand{\cU}{\mathcal{U}}

\newcommand{\norm}[1] {\left \| #1 \right \|}
\newcommand{\inclu}[0] {\ar@{^{(}->}}

\newcommand{\gph}{{\rm gph}\,}

\newcommand{\dist}{{\rm dist}}

\newcommand{\EE}{\mathbb{E}}

\newcommand{\lip}{\mathrm{lip}}

\newcommand{\RR}{\mathbb{R}}

\newcommand{\cX}{\mathcal{X}}
\newcommand{\cL}{\mathcal{L}}

\newcommand{\cP}{\mathcal{P}}
\newcommand{\cD}{\mathcal{D}}
\newcommand{\cM}{\mathcal{M}}

\newcommand{\epi}{\mathrm{epi}\,}

\newcommand{\range}{\mathrm{range}}

\newcommand{\abs}[1]{\left| #1 \right|}

\newcommand{\prox}{\mathrm{prox}}
\newcommand{\proj}{P}
\newcommand{\dom}{\mathrm{dom}\,}
\newcommand{\argmin}{\operatornamewithlimits{argmin}}

\newcommand{\NN}{\mathbb{N}}

\newcommand{\expect}[1]{\mathbb{E}\left[#1\right]}

\newcommand{\dotp}[1]{\left\langle #1\right\rangle}

\newtheorem{thm}{Theorem}[section]

\newtheorem{definition}[thm]{Definition}
\newtheorem{proposition}[thm]{Proposition}
\newtheorem{lem}[thm]{Lemma}
\newtheorem{cor}[thm]{Corollary}

\newtheorem{assumption}{Assumption}

\newcommand{\ub}{\texttt{ub}}

\newcommand{\cF}{\mathcal{F}}

\newtheorem{example}{Example}[section]

\theoremstyle{remark}
\newtheorem{claim}{Claim}

\newcommand{\discrete}{k}

\newcommand{\perturb}{\nu}

\usepackage{mathtools}
\usepackage[boxruled]{algorithm2e}

\newcommand{\tangent}[2]{T_{#1}({#2})}
\newcommand{\tangentM}[1]{\tangent{\cM}{#1}}

\newcommand{\zetak}{\zeta_k^{(k_0)}}
\numberwithin{equation}{section}

\newcommand{\cI}{\mathcal{I}}

\newcommand{\Nul}{{\rm Null}\,}

\newcommand{\M}{{\mathcal M}\,}
\newcommand{\cZ}{{\mathcal Z}\,}

\newcommand{\cG}{\mathcal{G}}

\newcommand{\xs}{x^{\star}}

\usepackage[boxruled]{algorithm2e}
\title{Asymptotic normality and optimality in \\nonsmooth stochastic approximation}

	\author{Damek Davis\thanks{School of ORIE, Cornell University,
Ithaca, NY 14850, USA;
\texttt{people.orie.cornell.edu/dsd95/}. Research of Davis supported by an Alfred P. Sloan research fellowship and NSF DMS award 2047637.}  \qquad Dmitriy Drusvyatskiy\thanks{Department of Mathematics, U. Washington,
Seattle, WA 98195; \texttt{www.math.washington.edu/{\raise.17ex\hbox{$\scriptstyle\sim$}}ddrusv}. Research of Drusvyatskiy was supported by NSF DMS-1651851 and CCF-2023166 awards.}\qquad Liwei Jiang\thanks{School of ORIE, Cornell University. Ithaca, NY 14850, USA;
	\texttt{orie.cornell.edu/research/grad-students/liwei-jiang}}}

\date{}
\begin{document}
\maketitle

\begin{abstract}
In their seminal work, Polyak and Juditsky showed that stochastic approximation algorithms for solving smooth equations enjoy a central limit theorem. Moreover, it has since been argued that the asymptotic covariance of the method is best possible among any estimation procedure in a local minimax sense of H\'{a}jek and Le Cam. A long-standing open question in this line of work is whether similar guarantees hold for important non-smooth problems, such as stochastic nonlinear programming or stochastic variational inequalities. In this work, we show that this is indeed the case. 
\end{abstract}

	

\section{Introduction}
Polyak and Juditsky~\cite{polyak1992acceleration} famously showed that the stochastic gradient method for minimizing smooth and strongly convex functions enjoys a central limit theorem: the error between the running average of the iterates and the minimizer, normalized by the square root of the iteration counter, converges to a normal random vector.
Moreover,  the covariance matrix of the limiting distribution is in a precise sense ``optimal" among any estimation procedure. 
A long standing open question is whether similar guarantees -- asymptotic normality and optimality -- exist for nonsmooth optimization and, more generally, for equilibrium problems. 
In this work, we obtain such guarantees under mild conditions that hold both in concrete circumstances (e.g. nonlinear programming) and under generic  linear perturbations.

The types of problems we will consider are best modeled as stochastic variational inequalities.  Setting the stage, consider the task of finding a solution $x^{\star}$ of the inclusion
\begin{equation}\label{eqn:stoch_var_ineq}
0 \in \mathop\EE_{z \sim \cP}[A(x, z)] + N_{\cX}(x).
\end{equation}
Here, $\cP$ is a probability distribution accessible only through sampling,  $A(\cdot, z)$ is a smooth map for almost every $z\sim \cP$, and $N_{\cX}(x)$ denotes the normal cone to a closed set $\cX$. 
Stochastic variational inequalities \eqref{eqn:stoch_var_ineq} are ubiquitous in contemporary optimization. For example, optimality conditions for constrained optimization problems $$\min_{x}~\mathop \EE_{z\sim \mathcal{P}}~ f(x,z)\qquad \textrm{subject to }x\in \cX,$$ fit into the framework~\eqref{eqn:stoch_var_ineq} by setting $A(x, z)=\nabla f(x,z)$ in \eqref{eqn:stoch_var_ineq}. More generally still,  Nash equilibria $x^{\star}=(x^{\star}_1,\ldots,x^{\star}_m )$ of stochastic games are solutions of the system
$$x_j^{\star}\in \argmin_{x_j\in\cX_j} \mathop\EE_{z \sim \cP}[f_j(x, z)]\qquad \textrm{for all }j=1,\ldots, m,$$
where $f_j$ and $\cX_j$, respectively, are the loss function and the strategy set of player $j$. First order optimality conditions for these $k$ coupled inclusions can be modeled as \eqref{eqn:stoch_var_ineq} by setting $[A(x, z)]_j:=\nabla_{x_j} f_j(x,z)$ and  $\cX:=\cX_1,\ldots, \cX_{m}$. 

There are two standard strategies for solving \eqref{eqn:stoch_var_ineq}: sample average approximation (SAA) and the stochastic forward-backward algorithm (SFB). The former proceeds by drawing a batch of samples $z_1,z_2,\ldots,z_k \overset{\rm iid}{\sim} \mathcal{P}$ and finding a solution $x_k$ to the empirical approximation \begin{equation}\label{eqn:saa}
0 \in \frac{1}{k}\sum_{i=1}^k[A(x, z_i)] + N_{\cX}(x).
\end{equation}
In contrast, the stochastic forward-backward (SFB) algorithm proceeds in an online manner, drawing a single sample $z_k\sim \mathcal{P}$ in each iteration $k$ and declaring the next iterate $x_{k+1}$ as 
\begin{equation}\label{eqn:sfb}
x_{k+1}\in P_{\cX}(x_k-\alpha_k \cdot A(x_k,z_k)).
\end{equation}
Here, $P_{\cX}(\cdot)$ denotes the nearest-point projection onto $\cX$. In the case of constrained optimization, $A(x,z)=\nabla f(x,z)$ is the gradient of some loss function $f(x,z)$, and the process \eqref{eqn:sfb} reduces to the stochastic projected gradient algorithm. Online algorithms like SFB are usually preferable to SAA since each iteration is inexpensive and can be performed online, whereas SAA requires solving the auxiliary optimization problem  \eqref{eqn:saa}. Although the asymptotic distribution of the SAA estimators  is by now well-understood \cite{king1993asymptotic,shapiro1989asymptotic,dupacova1988asymptotic}, our understanding of the asymptotic performance of the FSB iterates is limited in nonsmooth and constrained settings. The goal of this paper is to fill  this gap.
The main result of our work is the following.

\begin{quote}
Under reasonable assumptions, the running average of the SFB iterates exhibits the same asymptotic distribution as SAA. Moreover, both SAA and SFB are asymptotically optimal in a locally minimax sense of H\'{a}jek and Le Cam \cite{le2000asymptotics,van2000asymptotic}.
\end{quote}

We next describe our results, and their consequences, in some detail.
Namely, it is classically known (e.g. \cite{king1993asymptotic,shapiro1989asymptotic,dupacova1988asymptotic}) that the asymptotic performance of SAA \eqref{eqn:saa} is strongly influenced by the sensitivity of the solution $x^{\star}$ to perturbations of the left-hand-side of \eqref{eqn:stoch_var_ineq}.  In order to isolate this effect, let $S(v)$ consist of  solutions $x$ to the perturbed system
$$v\in \mathop\EE_{z \sim \cP}[A(x, z)] + N_{\cX}(x).$$
Throughout, we will assume that the solutions $S(v)$ vary smoothly near $x^{\star}$. More precisely, we will assume that the graph of $S$ locally around $(0,x^{\star})$ coincides with the graph of some smooth map $\sigma(\cdot)$. In the language of variational analysis \cite{dontchev2009implicit}, the map $\sigma(\cdot)$ is called a smooth localization of $S$ around $(0,x^{\star})$. It is known that this assumption holds in a variety of concrete circumstances and  under generic linear perturbations of semi-algebraic problems \cite{MR3461323}.

 Let us next provide the context and state our results.
It is known from \cite{king1993asymptotic,shapiro1989asymptotic} that under mild assumptions, the solutions $x_k$ of SAA \eqref{eqn:saa} are asymptotically normal:
\begin{equation}\label{eqn:ass_normsaa}
\sqrt{k}(x_k-x^{\star})\xrightarrow[]{D}\mathsf{N}\big(0, \nabla\sigma(0) \cdot {\rm Cov}(A(x^{\star},z)) \cdot \nabla\sigma(0)^{\top}\big).
\end{equation}
Thus the Jacobian of the solution map $\nabla\sigma(0)$ appears in the asymptotic covariance of the SAA estimator. In fact, we will argue that this is unavoidable. The first contributions of our work is that we prove that the asymptotic performance of SAA is locally minimax optimal---in the sense of H\'{a}jek and Le Cam \cite{le2000asymptotics,van2000asymptotic}---among all estimation procedures. Roughly speaking, this means that for any estimation procedure that outputs $\hat x_k$ based on $k$ samples, there exists a sequence of perturbations $\cP_k$ with $\frac{d\cP_k}{d\cP}=1+O(k^{-1/2})$, such that the performance of $\hat x_k$ on the perturbed sequence of problems is asymptotically no better than the performance of SAA on the target problem. 
We note that the analogous lower bound for stochastic nonlinear programming was obtained earlier in \cite{duchi2021asymptotic}, and our arguments are motivated by the techniques therein.
Aside from the lower bound, the main result of our work is to show that under reasonable assumptions, the running average of the SFB iterates enjoys the same  asymptotics as \eqref{eqn:ass_normsaa} and is thus  asymptotically optimal.

\begin{figure*}[t!]
    \centering
    \begin{subfigure}[t]{0.5\textwidth}
        \centering
        \includegraphics[scale=0.75]{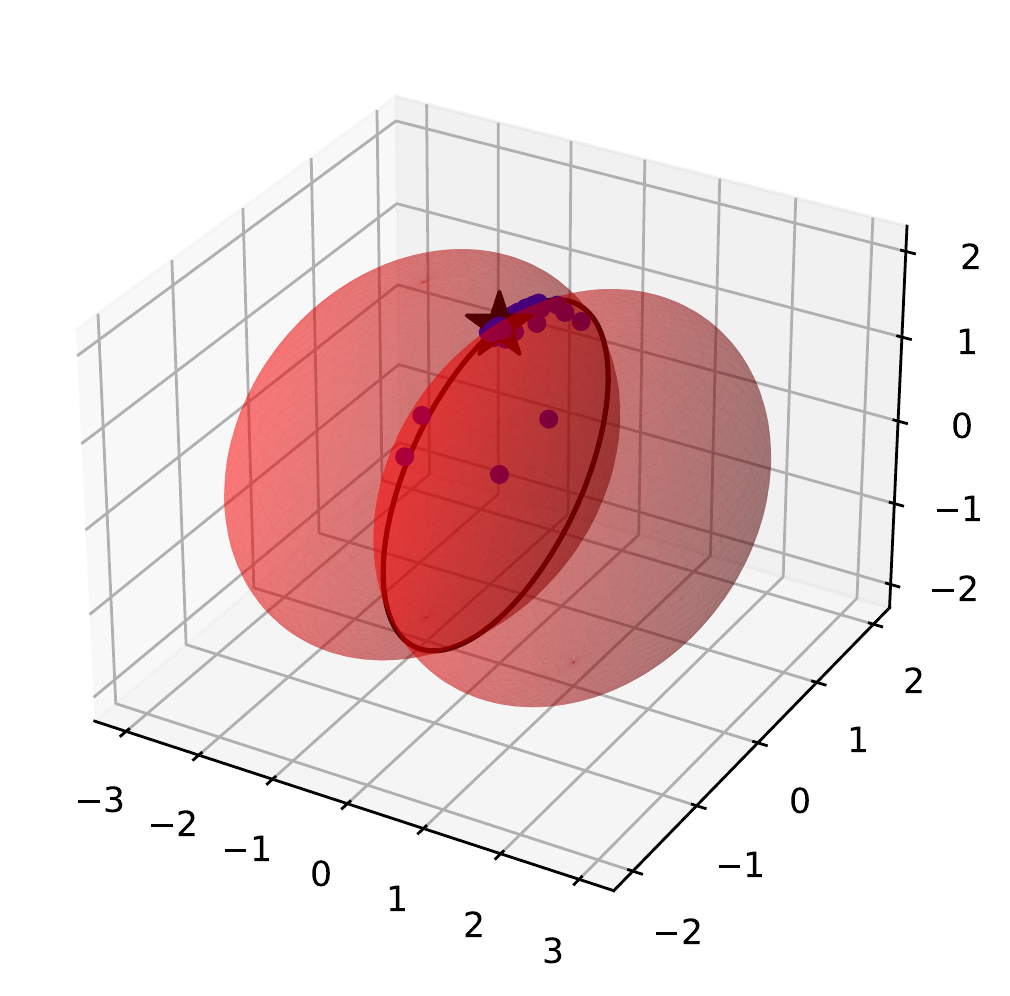}
        \caption{Feasible region and the iterates $x_k$}
    \end{subfigure}%
    ~ 
    \begin{subfigure}[t]{0.5\textwidth}
        \centering
        \includegraphics[scale=0.75]{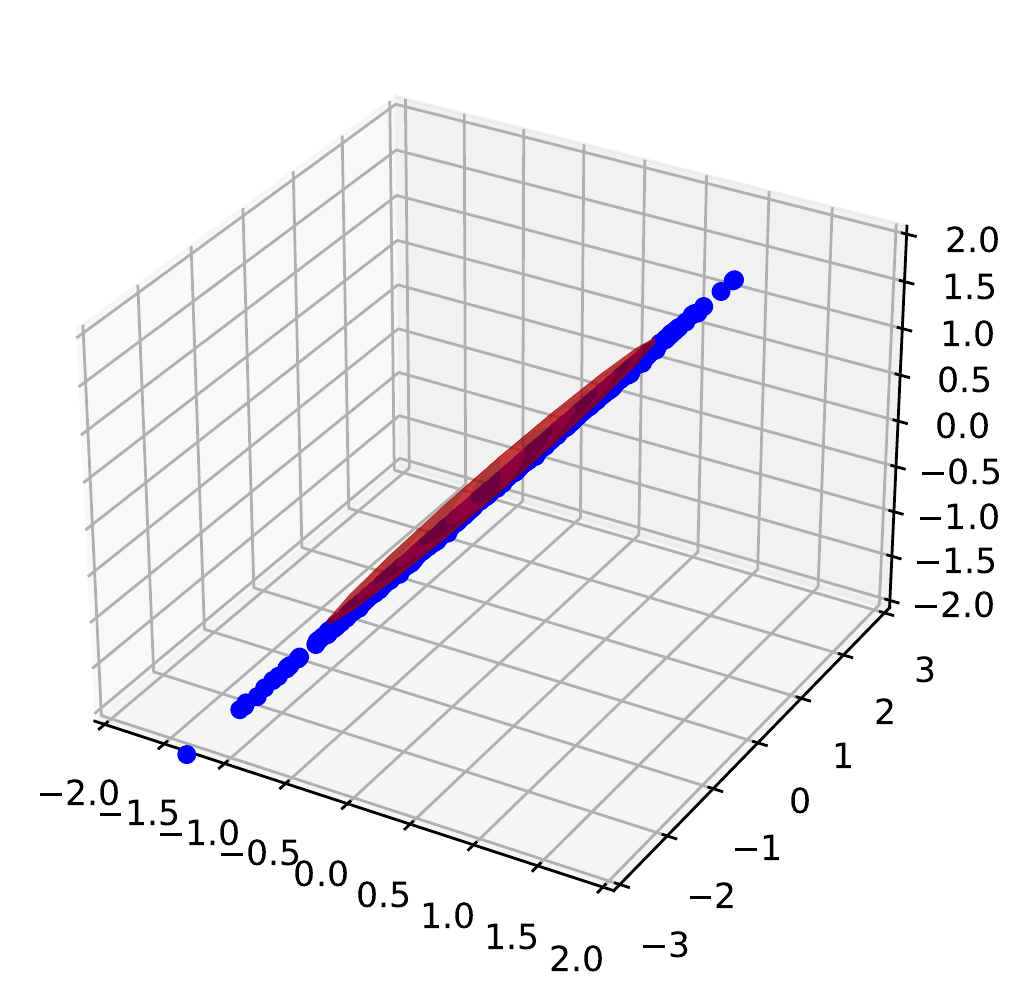}
        \caption{The deviations $\sqrt{k}(\bar x_k-x^\star)$ and the $95\%$ confidence region.}
    \end{subfigure}\\
    \begin{subfigure}[t]{0.5\textwidth}
        \centering
        \includegraphics[scale=0.5]{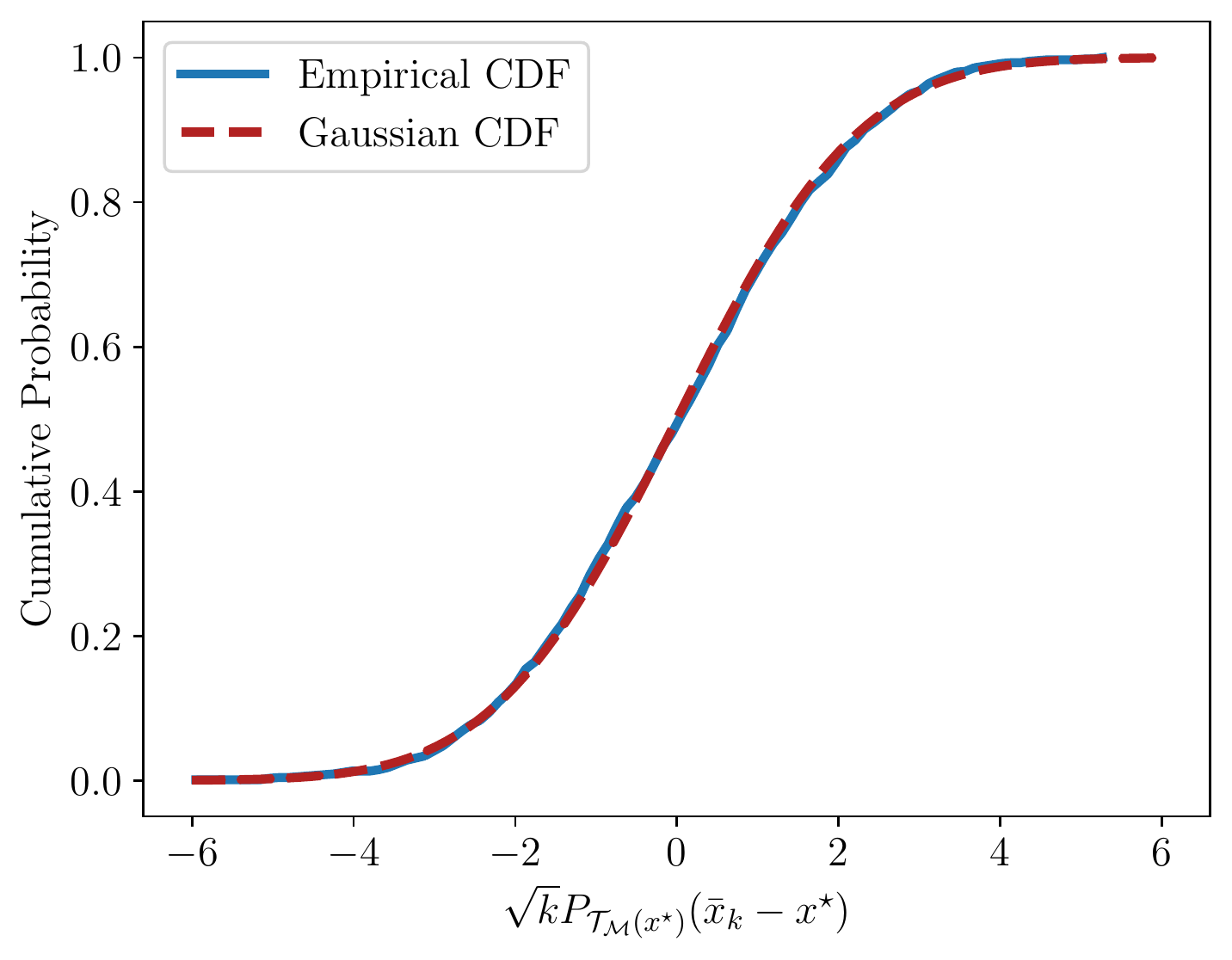}
        \caption{Empirical vs Gaussian CDF}
    \end{subfigure}%
    ~ 
    \begin{subfigure}[t]{0.5\textwidth}
        \centering
        \includegraphics[scale=0.5]{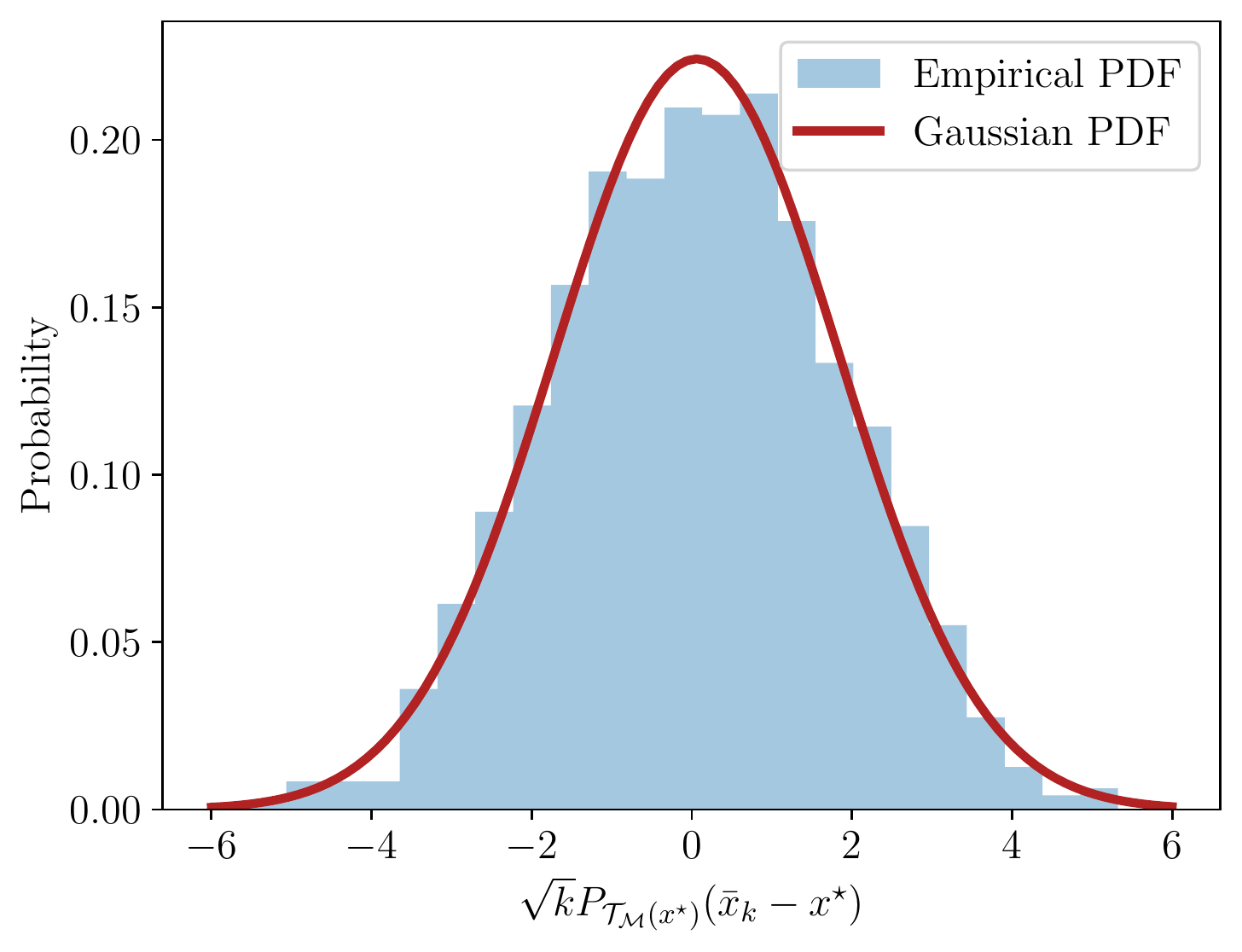}
        \caption{Histogram vs Gaussian density.}
    \end{subfigure}\\

    \caption{Performance of the stochastic projected gradient method for minimizing $\mathop\EE_g [-x_1+\langle g,x\rangle]$ over the intersection of two balls centered around $(-1,0,0)$ and $(1,0,0)$ of radius two. The expectation is taken over a Gaussian $g\sim N(0,I)$. The optimal solution $(0,0,\sqrt{3})$ (marked with a star)  lies on the active manifold $\cM$, which is a circle depicted in black. The figure on the top left depicts the iterates generated by a single run of the process initialized at the origin with stepsize $\eta_k=k^{-3/4}$ and executed for $1000$ iterations. The figure on the top right depicts the rescaled deviations $\sqrt{k}(\bar x_K-x^\star)$ taken over $100$ runs with $K=10^6$. The two figures clearly show that the iterates rapidly approach the active manifold and asymptotically the deviations $\sqrt{k}(\bar x_k-x^\star)$ are supported only along the tangent space to $\cM$ at $x^{\star}$. The two figures on the second row show the histogram and the empirical CDF, respectively, of the tangent components $\sqrt{k} P_{T_{\cM}(x^{\star})}(\bar x_k-x^{\star})$, overlaid with the analogous functions for a Gaussian.}\label{fig:illustration}
\end{figure*}


The guarantees we develop are already interesting for stochastic nonlinear programming:
\begin{align}\label{intro:constrainedstochasticsmooth}
\min_{x} ~\mathop f(x)=\mathop\EE_{z \sim \cP}[f(x, z)] \qquad \text{subject to}\qquad g_{i}(x) \leq 0 \qquad \forall i=1,\ldots, m.
\end{align}
Here each $g_i$ is a smooth function and the map $x \mapsto f(x, z)$ is smooth for a.e. $z\sim \mathcal{P}$. 
The optimality conditions for this problem can be modeled as the variational inequality  \eqref{eqn:stoch_var_ineq} under the identification $A(x,z)=\nabla f(x,z)$ and $\cX=\{x: g_i(x)\leq 0~\forall i=1,\ldots, m\}$. The stochastic forward-backward algorithm then becomes the stochastic projected gradient method. Our results imply that under the three standard conditions---linear independence of active gradients, strict complementarity, and strong second-order sufficiency---the running average of the SFB iterates $\bar x_k=\frac{1}{k}\sum_{i=1}^k x_i$ is asymptotically normal and optimal:
	$$\sqrt{k}(\bar x_k -  x^\star) \xrightarrow{D} N\left(0,  \nabla \sigma(0)\cdot\text{Cov}(\nabla f( x^\star, z))\cdot\nabla \sigma(0)\right).$$
Moreover,  as is classically known, the Jacobian $\nabla \sigma(0)$ admits an explicit description as 
\begin{equation}\label{eqn:invert_sigma_nlp}
\nabla \sigma(0)=(P_{\mathcal{T}}\nabla^2_{xx} \mathcal{L}(x^{\star},y^{\star})P_{\mathcal{T}})^{\dagger},
\end{equation}
where $\nabla^2_{xx}\mathcal{L}(x^{\star},y^{\star})$ is the Hessian of the Lagrangian function, the symbol $\dag$ denotes the  Moore-Penrose pseudoinverse, and 
$P_{\mathcal{T}}$ is the projection onto the linear subspace $\{\nabla g_i(x^{\star})\}^{\perp}_{i\in \mathcal{I}}$ and $\mathcal{I}=\{i: g_i(x^{\star})=0\}$ is the set of active indices. An illustrative example of the announced result is depicted in Figure~\ref{fig:illustration}, which plots the performance of the projected stochastic gradient method for minimizing a linear function over the intersection of two balls.
This result may be surprising in light of the existing literature. Namely, Duchi and Ruan~\cite{duchi2021asymptotic} uncover a striking gap between the estimation quality of SAA and at least one standard online method, called dual averaging~\cite{nesterovdualaveraging,xiao2009dual}, for stochastic nonlinear optimization. 
Indeed, even for the problem of minimizing the expectation of a linear function over a ball, the dual averaging method exhibits a suboptimal asymptotic covariance~\cite[Section 5.2]{duchi2021asymptotic}.\footnote{In contrast, in the special case that $\cX$ is polyhedral and convex, the dual averaging method is optimal~\cite{duchi2021asymptotic}.} In contrast, we see that the stochastic projected gradient method is asymptotically optimal.

Let us now return to the general problem~\eqref{eqn:stoch_var_ineq} and the stochastic forward-backward algorithm~\eqref{eqn:sfb}. In order to derive the claimed asymptotic guarantees for SFB, we will impose a few extra assumptions. First, in addition to assuming that $\sigma(\cdot)$ is smooth near the origin, we will assume that there exists a neighborhood $U$ of the origin such that $\sigma(U)$ is a smooth manifold. This assumption is mild, since it holds automatically for example if the matrix  $\nabla\sigma(\cdot)$ has constant rank on a neighborhood of the origin. In the language of \cite{drusvyatskiy2014optimality}, the set  $\cM=\sigma(U)$ is called an {\em active manifold} around $\bar x$. Returning to the case of stochastic nonlinear programming, the active manifold is simply the zero-set of the active inequalities 
$$\mathcal{M}=\{x: g_i(x)=0~~ \forall i\in \mathcal{I}\}.$$
See Figure~\ref{fig:illustration} for an illustration.
Variants of active manifolds have been extensively studied in nonlinear programming, under the names of identifiable surfaces \cite{wright1993identifiable}, partly smooth sets \cite{lewis2002active}, $\mathcal{UV}$-structures \cite{lemarecha2000,mifflin2005algorithm}, $g\circ F$ decomposable functions \cite{shapiroreducible}, and minimal identifiable sets \cite{drusvyatskiy2014optimality}.

The main idea of our argument is to relate the nonsmooth dynamics of SFB to a smooth stochastic approximation algorithm on $\cM$. More precisely, we will show that under mild conditions,  
 the shadow sequence $y_k := P_{\cM}(x_k)$ along the manifold $\cM$ behaves smoothly up to a small error
\begin{equation}\label{eqn:shadow_intro}
y_{k+1} = y_k - \alpha_k P_{\tangentM{y_k}}(A(y_k, z_k)+\nu_k) + o(\alpha_k),
\end{equation}
where $\tangentM{y_k}$ denotes the tangent space of $\cM$ at $y_k$ and $\nu_k$ is a zero-mean noise. Consequently, we may build on the techniques of Polyak and Juditsky \cite{polyak1992acceleration}  to obtain the asymptotics of the shadow sequence $y_k$, and then infer information about the original iterates $x_k$. We note that in the constrained optimization setting, the iteration \eqref{eqn:shadow_intro} becomes an inexact  Riemannian gradient method on the restriction of $f$ to $\cM$.

The validity of \eqref{eqn:shadow_intro} relies on two extra conditions, introduced  in \cite{davis2021subgradient}, which relate the ``first-order'' behavior of $\cX$ to that of $\cM$.
Namely, for $x\in \cX$ and $y\in \cM$ near $x^\star$, we assume: 
\begin{itemize}
\item $\langle N_{\cX}(x)\cap \mathbb{S}^{d-1},x-y\rangle= o(\|x-y\|)$\hfill [$(b)$-regularity]
\item $N_{\cX}(x)\cap \mathbb{S}^{d-1}\subset N_{\cM}(y)+O(\|x-y\|)\mathbb{B}$\hfill [strong $(a)$-regularity]
\end{itemize}
The $(b)$-regularity condition simply asserts that the secant line joining $x\in \cX$ and $y\in \cM$ becomes tangent to $\mathcal{X}$ as $x$ and $y$ tend to the same point near $x^{\star}$. The strong-$(a)$ regularity in contrast asserts that the normal cone $N_{\cX}(x)$ is contained in $N_{\cM}(y)$ up to a linear error $O(\|x-y\|)$---a kind of Lipschitz condition. The two regularity conditions are introduced and thoroughly developed in \cite{davis2021subgradient}, with numerous examples and calculus rules presented. In particular, both conditions hold automatically for  stochastic nonlinear programming.

\paragraph{Outline.} The outline of the rest of the paper is as follows. Section~\ref{sec:notation} presents the basic notation and constructions that will be used in the paper. Existence of smooth localizations $\sigma(\cdot)$ is a central assumption of our work. Section~\ref{sec:smooth_invertible_maps} develops asymptotic convergence guarantees and a local minimax lower bound for SAA. Sections~\ref{assumptions:for_algos} presents the classes of algorithms, which we will consider. Section~\ref{sec:ass_norm_main} states the main result on asymptotic normality of iterative methods. Section~\ref{sec:twopillarsmainresults} discuses the two pillars underpinning the main result of our work.

\section{Notation and basic constructions}\label{sec:notation}
This section records basic notation that we will use throughout the paper.
To this end, the symbol $\R^d$ will denote a Euclidean space with 
inner product $\langle\cdot,\cdot \rangle$ and the induced norm $\|x\|=\sqrt{\langle x,x\rangle}$. The symbol ${\bf B}$ will stand for the closed unit ball in $\R^d$, while $B_r(x)$ will denote the closed ball of radius $r$ around a point $x$. 
For any function $f\colon\R^d\to\R\cup\{+\infty\}$,  the {\em domain}, {\em graph}, and {\em epigraph} are defined as
\begin{align*}
\dom\, f&:=\{x\in \R^d: f(x)<\infty\},\\
\gph f&:=\{(x,f(x))\in \R^d\times \R: x\in\dom\, f\},\\
\epi\, f&:=\{(x,r)\in \R^d\times \R: r\geq f(x)\},
\end{align*}
respectively. We say that $f$ is {\em closed} if $\epi f$ is a closed set, or equivalently if $f$ is lower-semicontinuous. The {\em proximal map of $f$ with parameter $\alpha>0$} is given by 
$$\prox_{\alpha f}(x):=\argmin_y \left\{f(y)+\frac{1}{2\alpha}\|y-x\|^2\right\}.$$
 The {\em distance} and the {\em projection} of a point $x\in\R^d$ onto a set $Q\subset\R^d$ are 
\begin{align*}
d(x,Q):=\inf_{y\in Q}\|y-x\|\qquad\textrm{and}\qquad P_Q(x):=\argmin_{y\in Q}\|y-x\|,
\end{align*}
respectively. The indicator function of  $Q$, denoted by $\delta_Q(\cdot)$, is defined to be zero on $Q$ and $+\infty$ off it.  The symbol $o(h)$ stands for any function $o(\cdot)$ satisfying $o(h)/h\to 0$ as $h\searrow 0$.

\subsection{Smooth manifolds}\label{sec:manifolds}
Next, we recall a few definitions from smooth manifolds; we refer the reader to \cite{lee2013smooth,boumal2020introduction} for details.
Throughout the paper, all smooth manifolds $\mathcal{M}$ are assumed to be embedded in $\R^d$ and we consider the tangent and normal spaces to $\mathcal{M}$ as subspaces of $\R^d$. Thus, a set $\M\subset\R^d$ is a {\em $C^p$ manifold} (with $p\geq 1$)  if around any point $x\in \M$ there exists an open neighborhood $U\subset\R^d$ and a $C^p$-smooth map $F$ from $U$ to some Euclidean space $\R^n$ such that the Jacobian $\nabla F(x)$ is surjective and equality $\M\cap U=F^{-1}(0)$ holds. Then $F=0$ are called the {\em local defining equations} for $\cM$, and  the tangent and normal spaces to $\M$ at $x$ are defined by $T_{\M}(x):=\Nul(\nabla F(x))$ and $N_{\M}(x):=(T_{\M}(x))^{\perp}$, respectively.  Note that for $C^p$ manifolds $\M$ with $p\geq 1$, the projection $P_{\M}$ is $C^{p-1}$-smooth on a neighborhood of each point $x$ in $\mathcal{M}$ and is $C^p$-smooth on a neighborhood of the origin in the tangent space $T_{\M}(x)$ \cite{miller2005newton}. 
Moreover, the inclusion $\range(\nabla P_{\cM}(x)) \subseteq \tangentM{P_{\cM}(x)}$ holds for all $x$ near $\cM$ and the equality $\nabla P_{\cM}(x) = P_{\tangentM{x}}$ holds for all $x \in \cM$.

Let $\M\subset\R^d$ be a $C^p$-manifold for some $p\geq 1$. Then a function $f\colon \M\to\R$ is called $C^p$-smooth around a point $x\in \M$ if there exists a $C^p$ function $\hat f\colon U\to \R$ defined on an open neighborhood $U\subset\R^d$ of $x$ and that agrees with $f$ on $U\cap \M$. Then the {\em covariant gradient of $f$ at $x$} is the vector 
$\nabla_{\M} f(x):=P_{T_{\M}(x)}(\nabla \hat f(x)).$
When $f$ and $\M$ are $C^2$-smooth, the {\em covariant Hessian of $f$ at $x$} is  defined to be the unique self-adjoint 
bilinear form $\nabla_{\M}^2 f(x)\colon T_{\M}(x)\times T_{\M}(x)\to \R$ satisfying
$$\frac{d^2}{dt^2}f(P_{\M}(x+tu))\mid_{t=0}~=~\langle \nabla_{\M}^2 f(x)u, u\rangle\qquad \textrm{for all }u\in T_{\M}(x).$$
If $\cM$ is $C^3$-smooth, then we can write $\nabla_{\M}^2 f(x)$ simply as $$\nabla_{\M}^2 f(x)=P_{T_{\cM}( x)}\nabla^2 \hat f(x)P_{T_{\cM}( x)},$$ while regarding the right-side as a linear operator on $T_{\cM}( x)$.

A map $F\colon\cM\to\R^m$ is called $C^p$ smooth near a point $ x$ if there exists a map $\hat F\colon U\to\R^d$ defined on some neighborhood $U\subset\R^d$ of $ x$ that agrees with $F$ on $\cM$ near $ x$. In this case, we define the {\em covariant Jacobian} $\nabla F( x)\colon T_{\cM}( x)\to\R^m$ by the expression $\nabla_{\cM} F( x)u=\nabla \hat F( x) u$ for all $u\in T_{\cM}( x)$. An easy computation shows that in the particular case when $F(x)=\nabla_{\cM} f(x)$ for a $C^3$-smooth function $f\colon\cM\to\R$, 
the quadratic form defined by  $\nabla_{\cM} F(x)$ on $T_{\cM}(x)$ coincides with $\nabla^2_{\cM} f(x)$. More precisely, equality holds:
$$P_{T_{\cM}( x)}\cdot \nabla_{\cM} F(x)\cdot P_{T_{\cM}( x)}=P_{T_{\cM}( x)}\cdot\nabla^2_{\cM} f(x) \cdot P_{T_{\cM}( x)}.$$

\subsection{Normal cones and subdifferentials}\label{sec:normalcones}
Next, we require a few basic constructions of nonsmooth and variational analysis. Our discussion will follow mostly closely Rockafellar-Wets \cite{rockafellar2009variational}. Other influential treatments of the subject include \cite{mordukhovich2006variational,penot2012calculus,clarke2008nonsmooth,borwein2010convex}. 
 The {\em Fr\'{e}chet normal cone} to a set  $Q\subset\R^d$ at a point $x\in \R^d$, denoted $\hat N_Q(x)$, consists of all vectors $v\in \R^d$ satisfying 
\begin{equation}\label{eqn:frechet}
\langle v,y-x\rangle\leq o(\|y-x\|)\quad\textrm{as}\quad y\to x\textrm{ in }Q.
\end{equation}
Thus $v$ lies in $\hat N_Q(x)$ if up to first-order $v$ makes an obtuse angle with all directions pointing from $x$ into $Q$. Generally, Fr\'{e}chet normals are highly discontinuous  with respect to perturbations of the basepoint $x$. Consequently, we enlarge the Fr\'{e}chet normal cone as follows.
The {\em limiting normal cone} to $Q$ at $ x\in Q$, denoted by $N_Q(x)$, consists of all vectors $v\in\R^d$ for which there exist sequences $x_i\in Q$ and $v_i\in \hat N_Q(x_i)$ satisfying $(x_i,v_i)\to (x,v)$.

The analogous of normal cones for functions are subdifferentials, or sets of generalized derivatives.
Namely, consider a function $f\colon\R^d\to\R\cup\{\infty\}$ and a point $x\in \dom\, f$. The  {\em Fr\'{e}chet subdifferential  of $f$ at $x$}, denoted $\hat \partial f(x)$, consists of all vectors $v\in \R^d$ satisfying 
the approximation property:
$$f(y)\geq f(x)+\langle v,y-x\rangle+o(\|y-x\|)\quad \textrm{as}\quad y\to x.$$
The {\em limiting subdifferential  of $f$ at $x$}, denoted $\partial f(x)$, consists of all vectors $v\in \R^d$
such that there exist sequences $x_i\in \R^d$ and Fr\'{e}chet subgradients $v_i\in \hat \partial f(x_i)$ satisfying $(x_i,f(x_i),v_i)\to (x,f(x),v)$ as $i\to\infty$. 
A point $ x$ satisfying $0\in\partial f(x)$ is called {\em critical} for $f$. The primary goal of algorithms for nonsmooth optimization is the search for critical points.

\subsection{Active manifolds}\label{sec:activemanifolds}
Critical points of typical nonsmooth functions lie on a certain manifold that captures the activity of the problem in the sense that critical points of slight linear perturbation of the function do not leave the manifold. Such active manifolds have been modeled in a variety of ways, including identifiable surfaces \cite{wright1993identifiable}, partial smoothness \cite{lewis2002active}, $\mathcal{UV}$-structures \cite{lemarecha2000,mifflin2005algorithm}, $g\circ F$ decomposable functions \cite{shapiroreducible}, and minimal identifiable sets \cite{drusvyatskiy2014optimality}.
In this work, we adopt the following formal model of activity,  explicitly used in~\cite{drusvyatskiy2014optimality}.

\begin{definition}[Active manifold]\label{defn:ident_man}{\rm 
Consider a function $f\colon\R^d\to\R\cup\{\infty\}$ 
	and 
fix a set $\mathcal{M} \subseteq \dom f$ 
	containing 
	a point $\bar x$ satisfying $0\in \partial f(\bar x)$. 
Then $\mathcal{M}$ 
	is called an
{\em  active} $C^p${\em-manifold around} $\bar x$  
	if 
		there exists 
	a constant $\epsilon>0$ satisfying the following.
\begin{itemize}
\item {\bf (smoothness)}  
The set $\mathcal{M}$ 
	is 
a $C^p$ manifold near $\bar x$ and  the restriction 
	of 	$f$ 
	to $\mathcal{M}$ 
	is $C^p$-smooth near $\bar x$.
\item {\bf (sharpness)} 
The lower bound holds:
$$\inf \{\|v\|: v\in \partial f(x),~x\in U\setminus \cM\}>0,$$
where we set $U=\{x\in B_{\epsilon}(\bar x):|f(x)-f(\bar x)|<\epsilon\}$.
\end{itemize}
More generally, we say that {\em $\cM$ is an active manifold for $f$ at $\bar x$ for $\bar v \in \partial f(\bar x)$} if $\cM$ is an active manifold for the tilted function $f_v(x)=f(x)-\langle v,x\rangle$ at $\bar x$.
}
\end{definition}

The sharpness condition simply means that the subgradients of $f$ must be uniformly bounded away from zero at points off the manifold that are sufficiently close to $\bar x$ in distance and in function value. The localization in function value can be omitted for example if $f$ is weakly convex or if $f$ is continuous on its domain; see \cite{drusvyatskiy2014optimality} for details.

Intuitively, the active manifold has the distinctive feature that the function varies smoothly along the manifold and grows linearly normal to it; see Figure~\ref{fig:test1} for an illustration. This is summarized by the following theorem from  \cite[Theorem D.2]{davis2019stochastic_geo}. 
\begin{proposition}[Identification implies sharpness]\label{prop: sharpness}
	Suppose that a closed function $f\colon\E\to\R\cup\{\infty\}$ admits an active manifold $\cM$ at  a point $\bar x$ satisfying $0\in \hat \partial f(\bar x)$. 
	Then there exist constants $c, \epsilon > 0 $ such that 
	\begin{equation}\label{eqn:growth_cond}
		f(x) - f(\proj_{\cM}(x)) \ge  c\cdot \dist(x, \cM), \qquad \forall x \in B_{\epsilon}(\bar x).
	\end{equation}
\end{proposition}

\begin{figure}[h]
	\centering
		\centering
		\includegraphics[width=.5\linewidth]{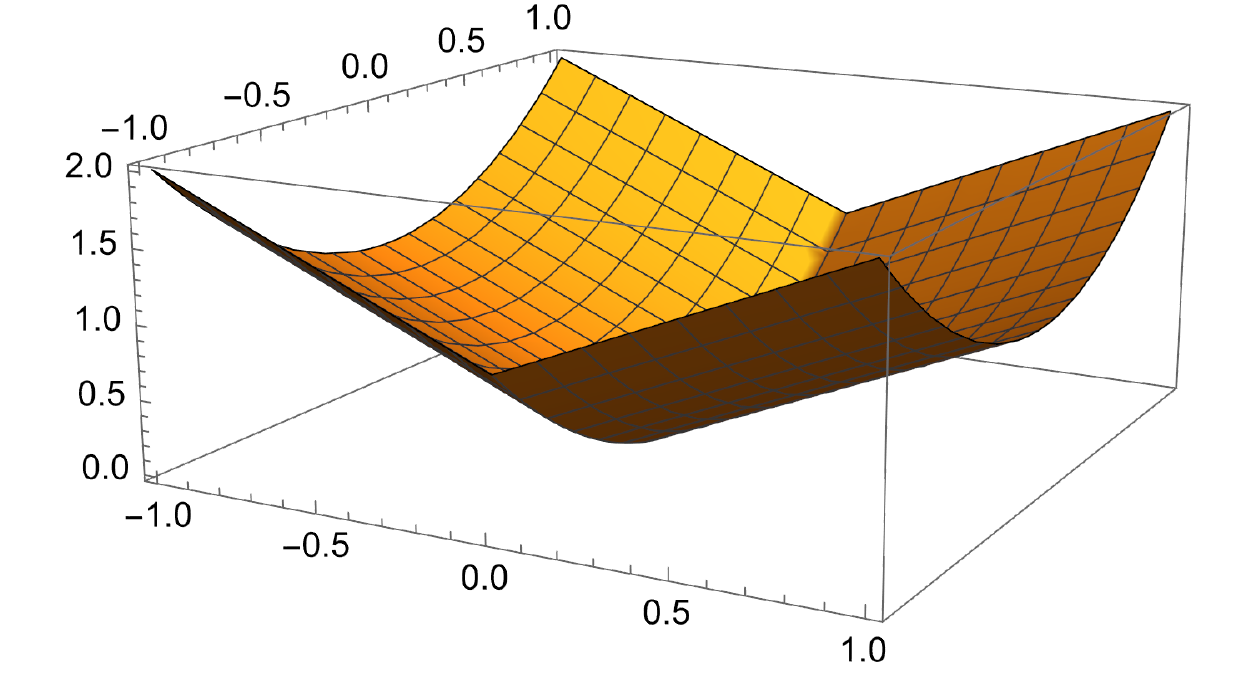}
		\captionof{figure}{$f(x_1,x_2)=|x_1|+x^2_2$}
		\label{fig:test1}
\end{figure}

Notice that there is a nontrivial assumption $0\in \hat \partial f(\bar x)$ at play in Proposition~\ref{prop: sharpness}. Indeed, under the weaker inclusion $0\in  \partial f(\bar x)$ the growth condition \eqref{eqn:growth_cond} may easily fail, as the univariate example $f(x)=-|x|$ shows. It is worthwhile to note that under the assumption $0\in \hat \partial f(\bar x)$, the active manifold is locally unique around $\bar x$ \cite[Proposition~8.2]{drusvyatskiy2014optimality}.

In order to make progress, we will require two extra conditions to hold along the active manifold that tightly couple the subgradients of $f$ on and off the manifold. Although these two conditions, introduced  in \cite{davis2021subgradient}, may look formidable, they are very mild indeed.

The motivation for the first regularity condition stems from a weakening of Taylor's theorem that is appropriate for nonsmooth functions. Namely, recall that any $C^1$-smooth function $f$ satisfies the first-order approximation property
\begin{equation}\label{eqn:taylor_dumb}
f(y)=f(x)+\langle \nabla f(x),y-x\rangle+o(\|y-x\|)\qquad \textrm{as }x,y\to \bar x.
\end{equation}
This estimate is fundamental to optimization theory and practice since it quantifies the approximation qualify of the linear model of $f$ furnished by the gradient. When $f$ is nonsmooth, the analogue of \eqref{eqn:taylor_dumb} with subgradients replacing the gradient can not possibly hold uniformly over all points $x$ and $y$ near $\bar x$, since it would imply differentiability. Instead, a reasonable requirement is to only require \eqref{eqn:taylor_dumb} to hold with points $y$  lying on the active manifold $\cM$. In fact, for our purposes, it suffices to replace the equality with an inequality.

\begin{definition}[$(b_{\leq})$-regularity]{\rm 
Consider a function $f\colon\R^d\to\R\cup\{\infty\}$ that is locally Lipschitz continuous on its domain. Fix a set $\cM\subset\dom f$ that is a $C^1$ manifold around $\bar x$ and such that the restriction of $f$ to $\cM$ is $C^1$-smooth near $\bar x$. We say that $f$ is {\em $(b_{\leq})$-regular along $\cM$ at $\bar x$} if there exists $\epsilon>0$ such that the estimates 
		\begin{align}
			f(y)&\geq f(x)+\langle v, y-x\rangle+(1+\|v\|)\cdot o(\|y-x\|),\label{eqn:bd1}
		\end{align}
		hold for all $x\in \dom f\cap B_{\epsilon}(\bar x)$, $y\in \cM\cap B_{\epsilon}(\bar x)$, and $v\in \partial f(x)$.}
\end{definition}

Importantly, condition $(b_{\leq})$ along an active manifold $\cM$  directly implies that all negative subgradients of $f$, taken at points $x$ near $\bar x$, point towards the active manifold $\cM$. Indeed, this is a direct consequence of Proposition~\ref{prop: sharpness}, and is summarized in the following corollary. Consequently, subgradient type methods always move in the direction of the active manifold---clearly a desirable property. 
\begin{cor}[Proximal aiming]\label{cor:prox-aiming_gen}
	Consider a closed function $f\colon\R^d\to\R\cup\{\infty\}$ that admits an active $C^1$-manifold $\cM$ at a point $\bar x$ satisfying $0\in \hat\partial f(\bar x)$. Suppose that $f$ is locally Lipschitz continuous on its domain and that $f$ is $(b_{\leq})$-regular along $\cM$ near $\bar x$. Then, there exists a constant $\mu>0$ such that, the estimate
	\begin{equation}\label{eqn:prox_aim_super_here}
	\langle v,x-P_{\cM}(x)\rangle\geq \mu\cdot\dist(x,\cM)+(1+\|v\|)\cdot o(\dist(x,\cM)),
	\end{equation}
	holds for all $x \in \dom f$ near $\bar x$ and for all $v\in \partial f(x)$. 
	\end{cor}

The second regularity condition has a different flavor, stemming from a weakening of Lipschitz continuity of the gradient. Namely, nonsmoothness by its nature implies that the deviation  $\partial f(x)-\partial f(y)$ is not controlled well by the distance in the arguments $\|x-y\|$. On the other hand, it turns out that if we take $y\in \cM$ and only look at the subgradient deviations in tangent directions, the error  $P_{T_{\cM}(y)}[\partial f(x)-\partial f(y)]$ is typically linearly bounded by  $\|x-y\|$. Moreover, in typical circumstances $P_{T_{\cM}(y)}[\partial f(y)]$ consists only of a single vector, the covariant gradient $\nabla_{\cM} f(y)$.

\begin{definition}[Strong $(a)$-regularity]
{\rm
Consider a function $f\colon\R^d\to\R\cup\{\infty\}$ that is locally Lipschitz continuous on its domain. Fix a set $\cM\subset\dom f$ that is a $C^1$ manifold around $\bar x$ and such that the restriction of $f$ to $\cM$ is $C^1$-smooth near $\bar x$. We say that $f$ is {\em strongly $(a)$-regular along $\cM$ near $\bar x$} if there exist constants $C,\epsilon>0$ satisfying: 
		\begin{align}
		\|P_{T_{\cM}(y)}(v-\nabla_{\cM} f(y))\|&\leq C(1+\|v\|) \|x-y\|,\label{eqn:str_a1}
		\end{align}
		 for all $x\in \dom f\cap B_{\epsilon}(\bar x)$, $y\in \cM\cap B_{\epsilon}(\bar x)$, and $v\in \partial f(x)$.

}
\end{definition}

The two regularity conditions easy extend to sets through their indicator functions. Namely, we say that a set $Q\subset\R^d$ is $(b_{\leq})$-regular (respectively strongly $(a)$-regular) along a $C^1$ manifold $\cM\subset Q$ at $\bar x\in \cM$ if the indicator function $\delta_Q$ is $(b_{\leq})$-regular (respectively strongly $(a)$-regular) along $\cM$ at $\bar x$.

The paper \cite{davis2021subgradient} presents a wide array of functions that admit active manifolds along which both conditions $(b_{\leq})$  and strong $(a)$  hold. Here, we discuss in detail a single example of nonlinear programming, and refer the reader to \cite{davis2021subgradient} for many more examples.

\begin{example}[Nonlinear programming]\label{ex:nlp}
{\rm
Consider the problem of nonlinear programming
\begin{equation}\label{eqn:nlp}
\begin{aligned}
\min_{x}~ &f(x)\\
{\rm s.t.}~\,&g_i(x)\leq 0\qquad \textrm{for }i=1,\ldots,m\\
&g_i(x)= 0\qquad \textrm{for }i=m+1,\ldots,n,
\end{aligned}
\end{equation}
where $f$ and $g_i$ are $C^p$-smooth functions on $\R^d$. Let $\cX$ denotes the set of all feasible points to the problem.
Consider now a point $\bar x\in \cX$ that is critical for the function $f+\delta_{\cX}$ and define the active index set 
$$\mathcal{I}=\{i: g_i(\bar x)=0\}.$$
Suppose the following is true:
\begin{itemize}
\item {\bf (LICQ)} the gradients $\{\nabla g_i(\bar x)\}_{i\in \mathcal{I}}$ are linearly independent.
\end{itemize}
Then the set 
$$\cM=\{x: g_i(x)=0~\forall i\in \cI\}$$
is a $C^p$ smooth manifold locally around $\bar x$. Moreover, all three functions  $f$, $\delta_{\cX}$, and $f+\delta_{\cX}$ are $(b_{\leq})$-regular and strongly $(a)$-regular along $\cM$ near $\bar x$. In order to ensure that $\cM$ is an active manifold of $f+\delta_{\cX}$, an extra condition is required. Define the Lagrangian function 
$$\mathcal{L}(x,y):=f(x)+\sum_{i=1}^{n+m} y_i g_i(x).$$
Criticality of $\bar x$ and LICQ ensures that there exists a (unique) Lagrange multiplier vector $\bar y\in \R^{m}_+\times \R^n$ satisfying $\nabla_x \mathcal{L}(\bar x,\bar y)=0$ and $\bar y_i=0$ for all $i\notin \cI$.
Suppose the following standard assumption is true:
\begin{itemize}
\item {\bf (Strict complementarity)} $\bar y_i>0$ for all $i\in \mathcal{I}\cap \{1,\ldots, m\}$.
\end{itemize}
Then $\cM$ is indeed an active $C^p$ manifold for $f+\delta_{\cX}$ at $\bar x$.}
\end{example}

\subsection{Smoothly invertible maps and active manifolds}
Performance of statistical estimation procedures strongly depends on the sensitivity of the problem to perturbation. A variety of estimation problems can in turn be modeled as the task of solving an inclusion $0\in F(x)$ for some set-valued map $F$, whose values we can only approximate by sampling. We next review basic perturbation theory based on the inverse/implicit function theorem paradigm, while closely following the monograph \cite{dontchev2009implicit}.

A {\em set-valued map} $F\colon\R^d\rightrightarrows\R^m$ is an assignment that maps a point $x\in \R^d$ to a set $F(x)\subset \R^m$. Set-valued maps always admit a set-valued inverse:
$$F^{-1}(y)=\{x: y\in F(x)\}.$$
The domain and graph of $F$ are defined, respectively, as 
$$\dom F:=\{x: F(x)\neq \emptyset\}\qquad \textrm{and}\qquad \gph F:=\{(x,y): y\in F(x)\}.$$
We will be interested in the sensitivity of the solutions to the system $v\in F(x)$ with respect to perturbations of the left-hand-side $v$, or equivalently, the variational behavior of the map $v\mapsto F^{-1}(v)$. In particular, we will be interested in settings when the graph of $F^{-1}$ coincides locally around a base point $(v,x)$ with a graph of a smooth map. This is the content of the following definition. 

\begin{definition}[Smooth invertibility]{\rm 
Consider a set-valued map $F\colon\R^d\rightrightarrows\R^m$ and  a pair $(\bar x,\bar v)\in \gph F$. We say that $F$ is {\em $C^p$ invertible around $(\bar x,\bar v)$ with inverse $\sigma(\cdot)$} if there exists a single-valued $C^p$-smooth map $\sigma(\cdot)$ and a neighborhood $U$ of $(\bar v,\bar x)$ satisfying
$$U\cap \gph F^{-1}=U\cap \gph \sigma.$$ 
}
\end{definition}

The definition might seem odd at first: there is nothing ``smooth'' about $F$, and yet we require the graph of $F^{-1}$ to coincide with a graph of a smooth function near $(\bar v,\bar x)$. On the contrary, we will see that in a variety of settings this assumption is indeed valid. In particular, smooth invertibility is typical in nonlinear programing.

\begin{example}[Nonlinear programming]\label{exa: nonlin2}
{\rm
Returning to Example \eqref{ex:nlp} with $p\geq 2$, define the set-valued map
$$F(x)=\nabla f(x)+N_{\cX}(x).$$
Then it is classically known that $F$ is $C^{p-1}$ invertible at $(\bar x,0)$ if and only if the matrix 
$$\Sigma:=P_{T_{\cM}(\bar x)}\nabla^2_{xx}\mathcal{L}(\bar x,\bar y)P_{T_{\cM}(\bar x)}$$
is nonsingular on $T_{\cM}(\bar x)$. In this case, the Jacobian of the inverse map is 
$\nabla \sigma(0)=\Sigma^{\dagger}$, where $\dagger$ denotes the Moore-Penrose pseudoinverse.
It is worthwhile to note that $\Sigma$ can be equivalently written as
$P_{T_{\cM}(\bar x)}\nabla^2_{\cM}f(\bar x)P_{T_{\cM}(\bar x)}.$}
\end{example}

Smooth invertibility is closely tied to active manifolds, and Example~\ref{exa: nonlin2} is a simple consequence. Indeed the following much more general statement is true. This result follows from a standard argument combining active manifolds and the implicit function theorem. The proof appears in  Appendix~\ref{sec:appsmooth}. We will require a mild assumption on the considered functions. Namely, following \cite[Definition 2.1]{poliquin1996prox} a function $f$ is called {\em subdifferentially continuous} at a point $\bar x$ if for any sequences $(x_i,v_i)\in \gph \partial f$ converging to some pair $(\bar x,\bar v)\in \gph\partial f$, the function values $f(x_i)$ converge to $f(\bar x)$. In particular, functions that are continuous on their domains and closed convex functions are subdifferentially continuous.

\begin{thm}[Smooth Invertibility and Active Manifolds]\label{thm:smooth_invert_main} Consider the set-valued map 
$$F(x):=A(x)+\partial f(x),$$
where $A\colon\R^d\to\R^d$ is $C^p$-smooth and $f\colon\R^d\to\R\cup\{\infty\}$ is subdifferentially continuous near a point $\bar x$. Suppose that $f$ admits a $C^{p+1}$ active manifold $\cM$ at some point $\bar x$ for $-A(\bar x)\in \hat\partial f(\bar x)$. 
Let $G(x)=0$ be any $C^{p+1}$-smooth local defining equations for $\cM$ near $\bar x$ and let $\hat f$ be any $C^{p+1}$-smooth function that agrees with $f$ on a neighborhood of $\bar x$ in $\cM$. Define the map 
$$\mathcal{H}(x,y):=A(x)+\nabla \hat f(x)+\nabla G(x)^{\top}y.$$
Then there exists a unique multiplier vector $\bar y$ satisfying the  condition $0=\mathcal{H}(\bar x,\bar y)$. Moreover, $F$ is $C^p$-invertible around $(0, \bar x)$ with inverse $\sigma(\cdot)$ if and only if the matrix 
$$\Sigma:=P_{T_{\cM}( \bar x)}\nabla_{x} \mathcal{H}(\bar x, \bar y)P_{T_{\cM}(\bar x)}$$ is nonsingular on $T_{\cM}(\bar x)$, and in this case  equality $\nabla \sigma(0)=\Sigma^{\dagger}$ holds.
\end{thm}

\section{Asymptotic normality and optimality of SAA.}\label{sec:smooth_invertible_maps}
Before analyzing the asymptotic performance of stochastic approximation algorithms, it is instructive to first recall guarantees for sample average approximation (SAA), where the assumptions, conclusions, and arguments are much simpler to state. This is the content of this section: we derive the asymptotic distribution of the SAA estimator for nonsmooth problems and show that it is asymptotically locally minimax optimal in the sense of H\'{a}jek and Le Cam \cite{le2000asymptotics,van2000asymptotic}.
Throughout the section we focus on the problem of finding a point $x^{\star}$ satisfying the variational inclusion:
 \begin{equation}\label{eqn:VIII}
0\in A(x)+H(x) \qquad \textrm{where }\qquad A(x)=\mathop\EE_{z\sim \cP} A(x,z).
\end{equation}
Here $H\colon\R^d\rightrightarrows \R^d$ is an arbitrary set-valued map, $\cP$ is a fixed probability distribution on some measure space $(\cZ, \mathcal{F})$, and $A\colon\R^d\times \mathcal{\cZ}\to \R^d$ is a measurable map. We will impose the following assumption throughout the rest of the section.

\begin{assumption}\label{assump:smooth_inverse}
{\rm The map $F:=A+H$ is $C^1$-smoothly invertible at $(0,\bar x)$ with inverse $\sigma(\cdot)$.
}
\end{assumption}

The SAA approach to solving \eqref{eqn:VIII} proceeds as follows. Let $S=\{z_1,\ldots, z_k\}$ be i.i.d samples drawn from $\cP$ and let $x_k$ be a solution of the problem
 \begin{equation}\label{eqn:sample_average}
0\in A_{S}(x)+H(x)\qquad \textrm{where}\qquad A_S(x):=\frac{1}{k}\sum_{i=1}^k A(x,z_i),
\end{equation}
assuming one exists. We will first show that the solutions of sample average approximations are asymptotically normal with covariance $\nabla\sigma(0) \cdot {\rm Cov}(A(\bar x,z)) \cdot \nabla\sigma(0)^{\top}$. Though variants of this result are well-known \cite{king1993asymptotic,shapiro1989asymptotic,dupacova1988asymptotic}, we provide a short proof in Appendix~\ref{sec:appsample_average_approx} highlighting the use of the solution map $\sigma(\cdot)$.  To this end, we impose the following standard assumption.

\begin{assumption}[Integrability and smoothness]\label{base_assumpt} {\rm Suppose that there exists a neighborhood $U$ around $\bar x$ satisfying the following.
\begin{enumerate}
\item\label{ass:lower1} For almost every $z$, the map $A(\cdot,z)$ is differentiable at every $x\in U$.
\item\label{ass:lower2} The second moment bounds hold:
$$\sup_{x\in U}\mathop\mathbb{E}_{z\sim \cP}\|A(x,z)\|^2<\infty\qquad \textrm{and}\qquad \sup_{x\in U}\mathop\EE_{z\sim \cP}\left[ \|\nabla A(x,z)\|^2_{\rm op}\right]<\infty.$$
\end{enumerate}
}
\end{assumption}

The following theorem shows that as long as $x_k$ eventually stay in a sufficiently small neighborhood of $\bar x$, the error $\sqrt{k}(x_k-\bar x)$ is asymptotically normal with covariance $\nabla\sigma(0) \cdot {\rm Cov}(M(\bar x,z)) \cdot \nabla\sigma(0)^{\top}$. Verifying that the problem \eqref{eqn:sample_average} admits solutions $x_k$ that are sufficiently close to $\bar x$ is a separate and well-studied topic and we do not discuss it here.

\begin{thm}[Sample average approximation]\label{thm:asympt_covar}
Suppose that Assumptions~\ref{assump:smooth_inverse} and \ref{base_assumpt} hold. In particular, there exist $\epsilon_1,\epsilon_2>0$ and a $C^1$-smooth map $\sigma\colon B_{\epsilon_1}(0)\to  B_{\epsilon_2}(\bar x)$ satisfying 
$$\gph \sigma=(B_{\epsilon_1}(0)\times B_{\epsilon_2}(\bar x))\cap \gph F^{-1}.$$ 
Suppose moreover that there exists a square integrable function $L(z)$ satisfying
\begin{equation}\label{eqn:lip_assume}
\|\nabla A(x_1,z)-\nabla A(x_2,z)\|\leq L(z)\|x_1-x_2\|\qquad \forall x_1,x_2\in B_{\epsilon_2}(\bar x).
\end{equation}
Shrinking $\epsilon_2$, if necessary, let us ensure that
$\epsilon_2\leq \min\left\{\frac{\lip(\sigma)^{-1}}{2\EE L}, \sqrt{\frac{\epsilon_1}{2\EE L}}\right\}$.
Let $S=\{z_1,\ldots, z_k\}$ be i.i.d samples drawn from $\cP$ and let $x_k$ be a measurable selection of the solutions \eqref{eqn:sample_average} such that ${\rm Pr}[x_k\in B_{\epsilon_2}(\bar x)]\to 1$ as $k$ tends to infinity. Then asymptotic normality holds:
$$\sqrt{k}(x_k-\bar x)\xrightarrow[]{D}\mathsf{N}\big(0, \nabla\sigma(0) \cdot {\rm Cov}(A(\bar x,z)) \cdot \nabla\sigma(0)^{\top}\big).$$
\end{thm}

We will next show that the asymptotic covariance in \eqref{thm:asympt_covar} is the best possible among all estimators of $\bar x$, in a local minimax sense of H\'ajek and Le Cam \cite{le2000asymptotics,van2000asymptotic}. Namely, we will lower-bound the performance of {\em any} estimation procedure for finding a solution to \eqref{eqn:VIII} induced by an adversarially-chosen sequence of small perturbations $\mathcal{P}'$ to $\mathcal{P}$. We will measure the size of the perturbation with the $\phi$-divergence
$$
 \Delta_{\phi}\big(\cP'\parallel \cP\big)=\int_{\cZ} \phi\bigg(\frac{d\cP'}{d\cP}\bigg) \,d\cP,
$$
induced by any $C^3$-smooth convex function $\phi\colon (0,\infty) \rightarrow \RR$ satisfying $\phi(1)=0$.
Define an {\em admissible neighborhood} $B_{\varepsilon}$ of $\cP$ to consist of all probability distributions $\cP'$ such that $\Delta_{\phi}(\cP'\parallel \cP) \leq \varepsilon$ and such that there exists a solution $\bar x_{\cP'}\in U$ to the perturbed variational equation $0\in \mathop\EE_{z\sim \cP'}[A(x,z)]+H(x).$  The following is the main result of this section.
In the theorem, we let $\EE_{P'_k}$ denote the expectation with respect to $k$ i.i.d. observations $z_i\sim \cP'$.

\begin{thm}[Asymptotic optimality]\label{thm:main_opt}
Suppose that Assumption~\ref{assump:smooth_inverse} and \ref{base_assumpt} hold. Let $\mathcal{L}\colon \RR^{d} \rightarrow [0,\infty)$ be any symmetric, quasiconvex, and lower semicontinuous function, and let $\widehat{x}_k\colon \cZ^{k}\rightarrow \RR^{d}$ be a sequence of estimators.  Then the inequality
  \begin{equation*}
  \lim_{c\to \infty}\liminf_{k \rightarrow \infty} \sup_{\cP'\in B_{c/k}} {\EE}_{P_{k}'}\big[\mathcal{L}\big(\sqrt{k} ( {\widehat x}_{k} - \bar x_{\cP'})\big)\big] \geq {\EE}\big[\mathcal{L}(Z)\big]
  \end{equation*}
  holds, where $Z\sim \mathsf{N}\big(0, \nabla\sigma(0) \cdot {\rm Cov}(A(\bar x,z)) \cdot \nabla\sigma(0)^{\top}\big)$.\end{thm}

The proof of Theorem~\ref{thm:main_opt} appears in Appendix \ref{sec:local_minimax}, and is inspired by the analogous theorem of Duchi-Ruan for stochastic nonlinear programming \cite{duchi2021asymptotic}.

\section{Stochastic approximation: assumptions \& examples}\label{assumptions:for_algos}
We next move to stochastic approximation algorithms, and in this section set forth the algorithms we will consider and the relevant assumptions. The concrete examples we will present will all be geared to solving variational inclusions, but the specifics of this problem class is somewhat distracting. Therefore we will instead only isolate the essential ingredients that are needed for our results to take hold. Setting the stage, our goal is to find a point $x$ satisfying the inclusion 
\begin{equation}\label{eqn:set-valued}
0\in F(x),
\end{equation}
where $F\colon\RR^d\rightrightarrows\RR^d$ is a set-valued map. Throughout, we fix one such solution $\bar x$ of \eqref{eqn:set-valued}. We will assume that in a certain sense, the problem \eqref{eqn:set-valued} is ``variationally smooth''.  That is, there exists a distinguished manifold $\mathcal{M}$---the active manifold in concrete examples---containing $\bar x$ and such that the map $x\mapsto P_{T_{\cM}(x)}F(x)$ is single-valued and $C^1$-smooth on $\cM$ near $\bar x$. The following assumption makes this precise.

\begin{assumption}[Smooth reduction]\label{ass:basic_assumpt}
{\rm 
Suppose that there exists a $C^p$ manifold $\mathcal{M}\subset\R^d$ such that the following properties are true.
\begin{enumerate}[label=$\mathrm{(C\arabic*)}$]
\item\label{it:single_val_FM} 
The map
$F_{\cM}\colon \mathcal{M}\to \R^d$ defined by 
$$F_{\cM}(x):=P_{T_{\cM}(x)}F(x)$$
is single-valued  on some neighborhood of $\bar x$ in $\cM$.
\item\label{it:smooth_reduction} 
There exists a neighborhood $U$ of $(\bar x,0)$ such that 
$$U\cap \gph F=U\cap \gph (F_{\cM}+N_{\cM}).$$
\end{enumerate}

}
\end{assumption}

We note that smooth invertibility of $F$ can be easily characterized in terms of the covariant  Jacobian $\nabla_{\cM} F_{\cM}(\bar x)$. This is the content of the following lemma.

\begin{lem}[Jacobian of the solution map]
The map $F$ is $C^p$-smoothly invertible around $(\bar x,0)$ with localization $\sigma(\cdot)$  if and only if the linear map 
$P_{T_{\cM}(\bar x)} \nabla F_{\cM}(\bar x)P_{T_{\cM}(\bar x)}$ is nonsingular on $T_{\cM}(\bar x)$. In this case, the Jacobian of the localization is given by $$\nabla \sigma(0)=(P_{T_{\cM}(\bar x)} \nabla_{\cM} F_{\cM}(\bar x)P_{T_{\cM}(\bar x)})^{\dagger}.$$
\end{lem}
\begin{proof}
Let $\Phi$ be a smooth extension of $F$ to a neighborhood $V\subset\R^d$ of $\bar x$. In light of Assumption~\ref{it:smooth_reduction}, the graphs of $F$ and $\Phi+N_{\cM}$ coincide near $(\bar x,0)$, and therefore we can focus on existence of smooth localizations of $(\Phi+N_{\cM})^{-1}$.  Applying Lemma~\ref{lem:basic_jac_soln} with $\bar y=0$, we see that $\Phi+N_{\cM}$ is $C^p$-smoothly invertible around $(\bar x,0)$ if and only if the linear map 
$P_{T_{\cM}(\bar x)}\nabla \Phi(\bar x)P_{T_{\cM}(\bar x)}$ is nonsingular on $T_{\cM}(\bar x)$. In this case, the Jacobian of the localization is given by $\nabla \sigma(0)=(P_{T_{\cM}(\bar x)}\nabla \Phi(\bar x)P_{T_{\cM}(\bar x)})^{\dagger}$.
Noting the equality $\nabla F_{\cM}(\bar x)P_{T_{\cM}(\bar x)}=\nabla_{\cM} \Phi(\bar x)P_{T_{\cM}(\bar x)}$ completes the proof.
\end{proof}

The stochastic approximation algorithms we consider assume access to  a \emph{generalized gradient mapping}:
$$
G \colon \RR_{++} \times \RR^d\times \RR^d \rightarrow \RR^d.
$$ 
Given $x_0 \in \RR^d$, the algorithm  iterates the update 
\begin{align}\label{alg:perturbedGiteration}
x_{k+1} = x_k - \alpha_kG_{\alpha_k}(x_k, \perturb_k), 
\end{align}
where $\alpha_k>0$ is a control sequence and $\perturb_k$ is stochastic noise. We will place relevant assumptions on the noise $\perturb_k$ later in Section~\ref{sec:twopillarsmainresults}. 

We make two assumptions on $G$. The first (Assumption~\ref{assumption:localbound}) is similar to classical Lipschitz assumptions and ensures the steplength can only scale linearly in $\|\perturb\|$. 

\begin{assumption}[Steplength] \label{assumption:localbound}
{\rm
We suppose that there exists a constant $\localconstant>0$ and a neighborhood $\cU$ of $\bar x$ such that the estimate
$$\sup_{x \in \cU_F} \|G_\alpha(x, \perturb)\| \leq \localconstant(1+\|\perturb\|),$$
holds for all $\nu \in \RR^d$ and $\alpha > 0$, where  we set $\cU_F := \cU \cap \dom F$.}\end{assumption}

The second assumption makes precise the relationship between the mapping $G$ and $F_{\mathcal{M}}$.

\begin{assumption}[Strong (a) and aiming]\label{assumption:Aproposed}
{\rm We suppose that there exist constants $\localconstant, \localmu > 0$ and a neighborhood $\cU$ of $\bar x$ such that the following hold for all $\nu \in \RR^d$ and $\alpha > 0$, where we set $\cU_F := \cU \cap \dom F$.
\begin{enumerate}[label=$\mathrm{(E\arabic*)}$]
	\item\label{assumption:smoothcompatibility} {\bf (Tangent comparison)}
For all $x \in \cU_F$, we have
\begin{align*}
\|P_{\tangentM{P_{\cM}(x)}}(G_\alpha(x, \perturb) - F(P_{\cM}(x)) - \nu)\| \leq C (1 + \|\perturb\|)^2(\dist(x, \cM) +\alpha).
\end{align*}
\item \label{assumption:aiming} {\bf (Proximal Aiming)} For $x \in \cU_F$, we have 
\begin{align*}
\dotp{G_\alpha(x, \perturb) - \nu, x - P_{\cM}(x)} &\geq \mu \cdot \dist(x, \cM) - (1+\|\perturb\|)^2(o(\dist(x, \cM)) + C\alpha).
\end{align*}
\end{enumerate}}
\end{assumption}
Some comments are in order. 
Assumption~\ref{assumption:smoothcompatibility}  ensures that the direction of motion $G_{\alpha_k}(x_k, \perturb_k)$ approximates well $F_{\cM}(P_{\cM}(x))$ in tangent directions to the manifold $\cM$. Assumption~\ref{assumption:aiming} ensures that after subtracting the noise from $G_{\alpha_k}(x_k, \perturb_k)$, the update direction $x_{k} - x_{k+1}$ locally points towards the manifold $\cM$. We will later show that this ensures the iterates $x_k$ approach the manifold $\cM$ at a controlled rate.

\subsection{Examples of stochastic approximation for variation inclusions}\label{sec:algos_and_assumptions}

The rest of the section is devoted to examples of algorithms satisfying Assumptions~\ref{assumption:localbound} and \ref{assumption:Aproposed}. In all cases, we will consider the task of solving the variational inclusion 
\begin{equation}\label{eqn:basic_VI}
0\in A(x)+\partial g(x)+\partial f(x).
\end{equation}
Here $A\colon\R^d\to\R^d$ is any single-valued continuous map, $f\colon\R^d\to\overline\R$ is a closed function, and $g\colon\R^d\to\overline\R$ is a closed function that is bounded from below.\footnote{In particular, $\prox_{\alpha f}(x)$ is nonempty for all $x\in\R^d$ and all $\alpha>0$.}
As explained in the introduction, variational inclusions encompass a variety of problems, most-notably first-order optimality conditions for nonlinear programming and Nash equilibria of games.
In order to identify \eqref{eqn:basic_VI} with \eqref{eqn:set-valued}, we define the set-valued map $F$ to be 
$$F(x):=A(x)+\partial g(x)+\partial f(x).$$
Throughout, we fix a point $\xs$ satisfying the inclusion~\eqref{eqn:basic_VI}.

A classical algorithm for problem~\eqref{eqn:basic_VI} is the stochastic forward-backward iteration, which proceeds by taking ``forward-steps'' on $A+\partial g$ and proximal steps on $f$. Specifically, given a current iterate $x_t$, the algorithm performs the update
\begin{equation}\label{eqn:algorithm}
\left\{\begin{aligned}
&{\rm Choose }~w_t\in \partial g(x_t)\\
&{\rm Choose }~ x_{t+1}\in \prox_{\alpha_t f}(x_t-\alpha_t(A(x_t)+w_t+\nu_t))
\end{aligned}\right\},
\end{equation}
where $\nu_t$ is a noise sequence.
The operator $G_{\alpha}(x,\nu)$ corresponding to this algorithm is simply
$$G_{\alpha}(x, \perturb) := \frac{x - s_f(x - \alpha( A(x)+s_g(x) + \perturb))}{\alpha},$$
where $s_g(x)$ is any selection of the subdifferential $\partial g(x)$ and $s_f(x)$ is any selection of the proximal map $\prox_{\alpha f}(x)$. The goal of this section is to verify Assumption~\ref{assumption:Aproposed} for this operator under a number of reasonable assumptions on $A$, $g$, and $f$. 

In particular, the local boundedness condition \ref{assumption:localbound} for $G$ is widely used in the literature, with a variety of sufficient conditions known. The following lemma describes a number of such conditions, which we will use in what follows. The proof appear in Appendix~\ref{sec:lemma_bounded_app}.

\begin{lem}[Local boundedness]\label{lem:basic_level_bound}
Suppose that $A(\cdot)$ and $s_g(\cdot)$ are locally bounded around $\bar x$. Then Assumption~\ref{assumption:localbound} is guaranteed to hold in any of the following settings.
\begin{enumerate}
\item\label{lb:2} $f$ is the indicator function of a closed set $\cX$.
\item\label{lb:3} $f$ is convex and the function $x\mapsto\dist(0,\partial f(x))$ is bounded on $\dom f$ near $\bar x$.
\item\label{lb:4} $f$ is Lipschitz continuous on $\dom g\cap \dom f$.
\end{enumerate}
\end{lem}

We next verify Assumption~\ref{assumption:Aproposed} in a number of reasonable settings; all proofs appear in the appendix. In particular, it will be useful to note the following expression for $F_{\cM}$. We will use this lemma throughout the section, without explicit reference.

\begin{lem}[Local tangent reduction]\label{lem:loc_tan_redux_app}
Suppose that $f$ and $g$ are Lipschitz continuous on their domains, $A$ is $C^p$-smooth, $f+g$ admits an active $C^{p+1}$ manifold at $\xs$ for $-A(\xs)$, and $f$ and $g$  are both $C^{p+1}$-smooth and strongly $(a)$ regular along $\cM$ near $\xs$. Then Assumption~\ref{ass:basic_assumpt} holds and $F_{\cM}$ admits the simple form
\begin{equation}\label{eqn:simple_F}
F_{\cM}(x)=P_{T_{\cM}(x)}(A(x))+\nabla_{\cM}g(x)+\nabla_{\cM}f(x),
\end{equation}
for all $x\in \cM$ near $\xs$.
\end{lem}

\subsubsection{Stochastic forward algorithm $(f=0)$}\label{sec:examplesoperators}
We begin with the simplest case of \eqref{eqn:basic_VI} where $f=0$. In this case, the iteration \eqref{alg:perturbedGiteration} reduces to a pure stochastic forward algorithm and the map $G$ takes the simple form 
$$G_{\alpha}(x, \perturb) := A(x)+s_g(x) + \perturb,$$
which is independent of $\alpha$.
Let us introduce the following assumption on the problem data.
\begin{assumption}[Assumptions for the forward algorithm]\label{assumption:subgradient}
{\rm Suppose that $f=0$ and that both $g(\cdot)$ and $A(\cdot)$ are Lipschitz continuous around $\bar x$. 
Suppose that $\cM \subseteq \cX$ is a $C^{p}$-smooth manifold for $g$ at $\bar x$.
\begin{enumerate}[label=$\mathrm{(F\arabic*)}$]
	\item \label{assumption:projectedgradient:stronga0} {\bf (Strong (a))}
The function $g$ is strongly $(a)$-regular along $\cM$ at $\bar x$. 
\item \label{assumption:projectedgradient:proximalaiming0}{\bf (Proximal aiming)} There exists $\mu > 0$ such that the inequality holds: 
\begin{align}\label{prop:projectedgradient:eq:aiming0}
\dotp{A(\bar x)+v, x - P_{\cM}(x)} \geq \mu\cdot  \dist(x, \cM) \qquad \text{for all $x$ near $\bar x$ and $v \in  \partial g(x)$.}
\end{align}
\end{enumerate}}
\end{assumption}

Note that Corollary~\ref{cor:prox-aiming_gen} shows that the aiming condition~\ref{assumption:projectedgradient:proximalaiming0} holds as long as the inclusion $-A(\bar x)\in \hat\partial g(\bar x)$ holds, $\cM$ is an active manifold for $g$ at $\bar x$ for $v=- A(\bar x)$, and $g$ is $(b_{\leq})$-regular along $\cM$ at $\bar x$.
The following proposition shows that Assumption~\ref{assumption:subgradient} suffices to ensure Assumption~\ref{assumption:Aproposed}. The proof appears in the appendix.

\begin{proposition}[Forward method]\label{prop:subgradient}
Assumption~\ref{assumption:subgradient} implies  Assumption~\ref{assumption:Aproposed}. 
\end{proposition}

The following is now immediate.
\begin{cor}[Active manifolds]
Suppose $f=0$ and that both $g(\cdot)$ and $A(\cdot)$ are Lipschitz continuous around $\bar x$. 
Suppose moreover that the inclusion $-A(\bar x)\in \hat \partial g(\bar x)$ holds, that $g$ admits a $C^2$ active manifold around $\bar x$ for $\bar v=-A(\bar x)$, and that $g$ is both $(b)_{\leq}$-regular and strongly $(a)$-regular along $\cM$ at $\bar x$. Then Assumption~\ref{assumption:Aproposed} holds.
\end{cor}

\subsubsection{Stochastic projected forward algorithm $(f=\delta_{\cX})$}
Next, we focus on the particular instance of \eqref{eqn:basic_VI} where  $f$ is an indicator function of a closed set $\cX$. In this case, the iteration \eqref{alg:perturbedGiteration} reduces to a stochastic projected forward algorithm and the map $G$ takes the form 
$$G_{\alpha}(x, \perturb) := \frac{x - s_\cX(x - \alpha( A(x)+s_g(x) + \perturb))}{\alpha},$$
where $s_\cX(x)$ is a selection of the projection map $P_{\cX}(x)$.
In order to ensure Assumption~\ref{assumption:Aproposed} for the stochastic projected forward method, we introduce the following assumption.
\begin{assumption}[Assumptions for the projected gradient mapping]\label{assumption:projectedgradient}
{\rm Suppose that $f$ is the indicator function of a closed set $\cX$ and both $g (\cdot)$ and $A(\cdot)$ are Lipschitz continuous around $\bar x$. Let $\cM \subseteq \cX$ be a $C^2$ manifold containing $\bar x$ and suppose that $f$ is $C^2$ on $\cM$ near $\bar x$.
\begin{enumerate}[label=$\mathrm{(G\arabic*)}$]
\item \label{assumption:projectedgradient:stronga} {\bf (Strong (a))}
The function $g$ and set $\cX$ are strongly $(a)$-regular along $\cM$ at $\bar x$. 
\item \label{assumption:projectedgradient:proximalaiming}{\bf (Proximal aiming)} There exists $\mu > 0$ such that the inequality holds 
\begin{align}\label{prop:projectedgradient:eq:aiming}
\dotp{A(\bar x)+v, x - P_{\cM}(x)} \geq \mu\cdot  \dist(x, \cM) \qquad \forall\text{~$x \in \cX$ near $\bar x$ and $v \in \partial g(x)$.}
\end{align}
\item\label{assumption:projectedgradient:bproxregularity} {\bf (Condition (b))} The set $\cX$ is  $(b_{\leq})$-regular along $\cM$ at $\bar x$.\end{enumerate}}
\end{assumption}

Note that a similar argument as Corollary~\ref{cor:prox-aiming_gen} shows that the aiming condition~\ref{assumption:projectedgradient:proximalaiming} holds as long as the inclusion $-A(\bar x)\in \hat\partial (g+f)(\bar x)$ holds, $\cM$ is an active manifold of $g+f$ at $\bar x$ for $v=-A(\bar x)$, and $g$ is $(b_{\leq})$-regular along $\cM$ at $\bar x$.

The following proposition shows that Assumption~\ref{assumption:projectedgradient} is sufficient to ensure Assumption~\ref{assumption:Aproposed}.
\begin{proposition}[Projected forward method]\label{prop:projectedgradient}
If Assumptions~\ref{assumption:localbound}  and \ref{assumption:projectedgradient} hold, then so does Assumption~\ref{assumption:Aproposed}.
\end{proposition}

The following is now immediate.
\begin{cor}[Active manifolds]
Suppose that $f$ is the indicator function of a closed set $\cX$ and both $g (\cdot)$ and $A(\cdot)$ are Lipschitz continuous around $\bar x$. Suppose moreover the inclusion $-A(\bar x)\in \hat \partial (g+f)(\bar x)$ holds, $g+f$ admits a $C^2$ active manifold around $\bar x$ for the vector $\bar v=-A(\bar x)$, and both $g$ and $f$ are $(b_{\leq})$-regular and  strongly $(a)$-regular along $\cM$ at $\bar x$. Then Assumption~\ref{assumption:Aproposed} holds.
\end{cor}

\subsubsection{Stochastic forward-backward method $(g=0)$}
Finally, we focus on the particular instance of \eqref{eqn:basic_VI} where  $g=0$. In this case, the iteration \eqref{alg:perturbedGiteration} reduces to a stochastic forward-backward algorithm and the map $G$ becomes
$$G_{\alpha}(x, \perturb) := \frac{x - s_f(x - \alpha( A(x) + \perturb))}{\alpha},$$

In order to ensure Assumption~\ref{assumption:Aproposed} for the stochastic proximal gradient method, we introduce the following assumptions.
\begin{assumption}[Assumptions for the forward-backward method]\label{assumption:proximalgradient}
{\rm Suppose  $g=0$  and $f(\cdot)$ and $A(\cdot)$ are Lipschitz continuous on $\dom f$ near $\bar x$. Suppose moreover that there exists a $C^2$ manifold  $\cM \subset \cX$ containing $\bar x$ and such that $f$ is $C^2$-smooth on $\cM$ near $\bar x$.
\begin{enumerate}[label=$\mathrm{(H\arabic*)}$]
	\item \label{assumption:proximalgradient:stronga}{\bf (Strong (a))}
The function $f$ is strongly $(a)$-regular along $\cM$ at $\bar x$. 
\item \label{assumption:proximalgradient:proximalaimingprox} {\bf (Proximal Aiming)} There exists $\mu > 0$ such that the inequality 
\begin{align}\label{prop:proximalgradient:eq:aiming}
\dotp{A(\bar x)+v, x - P_{\cM}(x)} \geq \mu\cdot  \dist(x, \cM) - (1+\|v\|)o(\dist(x, \cM))
\end{align}
holds for all $x \in \dom f$ near $\bar x$ and $v \in \hat\partial f(x)$.
\end{enumerate}}
\end{assumption}

Note that Corollary~\ref{cor:prox-aiming_gen} shows that the aiming condition~\ref{assumption:proximalgradient:proximalaimingprox} holds as long as the inclusion $-A(\bar x)\in \hat\partial f(\bar x)$ holds, $\cM$ is an active manifold for $f$ at $\bar x$ for $v=-A(\bar x)$, and $f$ is $(b_{\leq})$-regular along $\cM$ at $\bar x$.

The following proposition shows that Assumption~\ref{assumption:proximalgradient} is sufficient to ensure Assumption~\ref{assumption:Aproposed}.

\begin{proposition}[Forward-backward method]\label{prop:proximalgradient}
If Assumptions~\ref{assumption:localbound} and~\ref{assumption:proximalgradient} hold, then so does Assumption~\ref{assumption:Aproposed}.
\end{proposition}

The following is now immediate.
\begin{cor}[Active manifolds]
Suppose  $g=0$  and both $f$ and $A(\cdot)$ are Lipschitz continuous on $\dom f$ near $\bar x$.  Suppose moreover the inclusion $-A(\bar x)\in \hat \partial f(\bar x)$ holds. Suppose that $f$ admits a $C^2$ active manifold around $\bar x$ for $\bar v=-A(\bar x)$ and $f$ is both $(b)_{\leq}$-regular and strongly $(a)$-regular along $\cM$ at $\bar x$. Then Assumption~\ref{assumption:Aproposed} holds.
\end{cor}

\section{Asymptotic normality}\label{sec:ass_norm_main}

Next, we impose two assumptions on the step-size $\alpha_k$ and the noise sequence  $\perturb_k$. The first is standard, and is summarized next.

\begin{assumption}[Standing assumptions]\label{assumption:zero}
{\rm~Assume the following.
\begin{enumerate}[label=$\mathrm{(J\arabic*)}$]
\item The map $G$ is measurable.
\item There exist constants $c_1, c_2 > 0$ and $\gamma \in (1/2, 1]$ such that 
$$
\frac{c_1}{k^\gamma} \leq \alpha_k \leq \frac{c_2}{k^\gamma}.
$$
\item $\{\perturb_k\}$ is a martingale difference sequence w.r.t.\ to the increasing sequence of $\sigma$-fields 
$$
\cF_k = \sigma(x_j \colon j \leq k \text{ and } \perturb_j  \colon j<k),
$$
and there exists a function $q \colon \RR^d \rightarrow \RR_+$ that is bounded on bounded sets with 
$$
\EE[\perturb_k \mid \cF_k] = 0 \qquad \text{ and } \qquad  \EE[\|\perturb_k\|^4\mid \cF_k]  < q(x_k).
$$
We let $\EE_k[\cdot ] = \EE[ \cdot \mid \cF_k]$ denote the conditional expectation.
\item The inclusion $x_k \in \dom F$ holds for all $k \geq 1$.
\end{enumerate}}
\end{assumption}
All items in Assumption~\ref{assumption:zero} are standard in the literature on stochastic approximation methods and mirror for example those found in~\cite[Assumption C]{davis2020stochastic}. The only exception is the fourth moment bound on $\|\perturb_k\|$, which stipulates that $\nu_k$ has slightly lighter tails. This bound appears to be necessary for the setting we consider.

To prove our asymptotic normality results, we impose a further  assumption on the noise sequence $\nu_k$, which also appears in~\cite[Assumption D]{duchi2021asymptotic}. Before stating it, as motivation, consider   the  stochastic variational inequality \eqref{eqn:basic_VI} given by:
$$0\in A(x)+\partial f(x)+\partial g(x)\qquad \textrm{where}\qquad A(x)=\mathop\EE_{z\sim \cP} A(x,z).$$
Then the noise $\nu_k$ in the algorithm \eqref{eqn:algorithm} takes the form
$$\nu_k=A(x_k; z_k)-A(x_k).$$
Equivalently, we may decompose the right-hand-side as
$$
\nu_k= \underbrace{A(\bar x; z_k)-A(\bar x)}_{=:\nu^{(1)}_k} + \underbrace{(A(x_k; z_k) - A(\bar x; z_k)) + (A(\bar x) - A(x_k))}_{=:\nu^{(2)}_k(x_k)},
$$
The two components $\nu^{(1)}_k$ and $\nu^{(2)}_k(x_k)$ are qualitatively different in the following sense.
On one hand, the sum $\frac{1}{\sqrt{k}}\sum_{i=1}^k \nu^{(1)}_i$ clearly converges to a zero-mean normal vector as long as the covariance ${\rm Cov}(A(\bar x, z))$ exists. On the other hand, $\nu^{(2)}_k(x_k)$ is small in the sense that
$\EE_k\|\nu^{(2)}_k(x_k)\|^2\leq 2\cdot \EE_z[L(z)^2]\cdot\|x_k-\bar x\|^2,$
where $L(z)$ is a Lipschitz constant of $A(\cdot,z)$.
With this example in mind, we introduce the following assumption on the noise sequence.

\begin{assumption}\label{assumption:martinagle}
{\rm Fix a point $\bar x \in \dom F$ at which Assumption~\ref{ass:basic_assumpt} holds and let $U$ be a matrix whose column vectors form an orthogonal basis of $T_\cM(\bar x)$. We suppose the noise sequence has decomposable structure $\nu_k = \nu_k^{(1)} + \nu_k^{(2)}(x_k)$, where $\nu_k^{(2)} \colon \dom F \rightarrow \RR^d$ is a random function satisfying 
$$
\EE_k[\|U^\top\nu_k^{(2)}(x)\|^2] \leq C\|x - \bar x\|^2 \qquad \text{for all $x\in \dom F$ near $\bar x$},
$$
and some $C > 0$. In addition, we suppose that for all $x \in \dom F$, we have $\EE_k[\perturb_k^{(1)}] = \EE_k[\perturb_k^{(2)}(x)] = 0$ and the following limit holds:
$$
\frac{1}{\sqrt{k}} \sum_{i=1}^k U^\top\perturb_i^{(1)}  \xrightarrow{D} N(0, \Sigma).
$$
for some symmetric positive semidefinite matrix $\Sigma$.}
\end{assumption}
We are now ready to state the main result of this work---asymptotic normality for stochastic approximation algorithms.

\begin{thm}[Asymptotic Normality]\label{thm: asymptotic normality proposed}
Suppose that Assumption~\ref{ass:basic_assumpt}, \ref{assumption:localbound}, \ref{assumption:Aproposed}, \ref{assumption:zero}, and~\ref{assumption:martinagle} hold. Suppose that $\gamma \in (\frac{1}{2},1)$ and that the sequence $x_k$ generated by the process \eqref{alg:perturbedGiteration} converges to $\bar x$ with probability one. Suppose that there exists a constant $\mu>0$ satisfying 
\begin{equation}\label{eqn:strong_growth}
\langle\nabla_{\cM} F_{\cM}(\bar x)v,v\rangle \geq \mu \|v\|^2\qquad \textrm{for all}\qquad v\in T_{\cM}(\bar x).
\end{equation}
Then $F$ is $C^p$-smoothly invertible around $(\bar x,0)$ with inverse $\sigma(\cdot)$ and 
the average iterate $\bar x_k = \frac{1}{k} \sum_{i=1}^k x_i$ satisfies
	$$\sqrt{k}(\bar x_k - \bar x) \xrightarrow{D} N\left(0,  \nabla \sigma(0) \cdot \Sigma \cdot\nabla \sigma(0)^{\top}\right).$$
	Moreover, $\nabla \sigma(0)$ can be equivalently written as $\nabla \sigma(0)=(P_{T_{\cM}(\bar x)} \nabla_{\cM} F_{\cM}(\bar x)P_{T_{\cM}(\bar x)})^{\dagger}$.
\end{thm}

The conclusion of this theorem is surprising: although the sequence $x_k$ never reaches the manifold, the limiting distribution of $\sqrt{k}(\bar x_k - \bar x)$ is supported on the tangent space $\tangentM{\bar x}$. Thus asymptotically, the ``directions of nonsmoothness," which are normal to $\cM$, are quickly ``averaged out." When $\|G_{\alpha_k}(x_k, \perturb_k)\|$ is bounded away from $0$ for all $k$, this means that $x_k$ must  oscillate across the manifold, instead of approaching it from one direction.

\subsection{Asymptotic normality in nonlinear programming}
As a simple illustration of Theorem~\ref{thm: asymptotic normality proposed}, we now spell out the consequence for the stochastic projected gradient method for stochastic nonlinear programming, already discussed in Example~\ref{ex:nlp}. Namely, consider the problem \eqref{eqn:nlp} and let $\bar x$ be a local minimizer. Suppose that $g_i$ are $C^3$-smooth near $\bar x$ and $f$ takes the form $f(x)=\EE_{z\sim \cP} f(x,z)$ for some probability distribution $\cP$ and each function $f(\cdot,z)$ is $C^3$-smooth near $\cX$. 
Consider the following stochastic projected gradient method: 
\begin{align}\label{eq:NLPformulationSGD}
\text{Sample: } &z_k \sim P \notag \\
\text{Update: } & x_{k+1} \in P_{\cX}(x_k - \alpha_k \nabla f(x_k; z_k)).
\end{align}
In order to understand the asymptotics of the algorithm, as in Example~\ref{ex:nlp}, let $\bar y$ be the Lagrange multiplier vector and suppose that LICQ and strict complementarity holds. Suppose moreover the second-order sufficient conditions: 
there exists $\mu > 0$ such that  
$$
w^\top\left[\nabla^2_{xx} \mathcal{L}(\bar x,\bar y)\right] w \geq \mu \|w\|^2 \qquad \text{for all $w \in \tangentM{\bar x}$.}
$$
Note that, as explained in Example~\ref{exa: nonlin2}, this condition is simply the requirement that the covariant Hessian of $f := f_0 + \delta_{\cX}$
$$
\nabla_{\cM}^2 f(\bar x) = P_{\tangentM{\bar x}}\nabla^2_{xx}\mathcal{L}(x^{\star},y^{\star}) P_{\tangentM{\bar x}}
$$ 
is positive definite on $\tangentM{\bar x}$. Finally, to ensure our noise sequence
\begin{align*}
	\perturb_k &= \nabla f(x_k; z_k) - \nabla f(x_k) \\
	&= \underbrace{\nabla f(\bar x; z_k) - \nabla f_0(\bar x)}_{=:\perturb_k^{(1)}}  + \underbrace{(\nabla f(x_k; z_k) - \nabla f(\bar x; z_k) + \nabla f(\bar x) - \nabla f(x_k))}_{=: \perturb_k^{(2)}(x_k)}, 
\end{align*}
satisfies Assumptions~\ref{assumption:zero} and~\ref{assumption:martinagle}, we assume the stochasticity is sufficiently well-behaved:
\begin{enumerate}[label=$\mathrm{(G\arabic*)}$]
\setcounter{enumi}{4}
\item \label{assumption:stochasticgradientNLP}{\bf (Stochastic Gradients)} As a function of $x$, the fourth moment 
$$
x \in \cX \mapsto \EE_{z \sim \cP}[\|\nabla f(x; z) - \nabla f(x)\|^4]
$$ 
is bounded on bounded sets. Moreover, there exists $C > 0$ such that
\begin{align*}
\EE_{z \sim \cP}[\|\nabla f(x; z) - \nabla f(\bar x; z)\|^2 ] \leq C\|x - \bar x\|^2 \qquad \text{for all $x \in \cX$.}
\end{align*}
Finally, the gradients $P_{\tangentM{\bar x}} \nabla f(\bar x; z)$ have finite covariance $\Sigma = \text{Cov}(P_{\tangentM{\bar x}}\nabla f(\bar x; z))$.
\end{enumerate}

With these assumptions in hand, we have the following asymptotic normality result for nonlinear programming---a direct corollary of Theorem~\ref{thm: asymptotic normality proposed}.
\begin{cor}[Asymptotic normality in nonlinear programming]\label{cor:asymptoticnormalityNLP}
Suppose that $\gamma \in (\frac{1}{2},1)$ and consider the iterates $x_k$ generated by the stochastic projected gradient method~\eqref{eq:NLPformulationSGD}. Then if $x_k$ converges to $\bar x$ with probability 1, the average iterate $\bar x_k = \frac{1}{k} \sum_{i=1}^k x_i$ satisfies
	$$\sqrt{k}(\bar x_k - \bar x) \xrightarrow{d} N\left(0,  \nabla \sigma(0) \cdot\mathrm{Cov}(\nabla f(\bar x; z))  \cdot \nabla \sigma(0)^{\top}\right),$$
	where $\nabla \sigma(0)=(P_{T_{\cM}(\bar x)} \nabla^2_{xx} \mathcal{L}(\bar x,\bar y)P_{T_{\cM}(\bar x)})^{\dagger}$.
\end{cor}

As stated in the introduction, this appears to be the first asymptotic normality guarantee for the standard stochastic projected gradient method in general nonlinear programming problems with $C^3$ data, even in the convex setting. Finally we note that even for simple optimization problems, dual averaging procedures can achieve suboptimal convergence \cite{duchi2021asymptotic}. This is surprising since such methods identify the active manifold~\cite{JMLR:v13:lee12a} (also see~\cite[Section 4.1]{duchi2021asymptotic}), while projected stochastic gradient methods do not.

\section{The two pillars of the proof of Theorem~\ref{thm: asymptotic normality proposed}}\label{sec:twopillarsmainresults}
The proof of our main result, Theorem~\ref{thm: asymptotic normality proposed}, appears in the appendix. In this section, we outline the main ingredients of the proof. Namely, Assumption~\ref{assumption:Aproposed} at a point $\bar x$ guarantees two useful behaviors, provided the iterates $\{x_k\}$ of algorithm~\eqref{alg:perturbedGiteration} remain in a small ball around $\bar x$. First $x_k$ must approach the manifold $\cM$ containing $\bar x$ at a controlled rate, a consequence of the proximal aiming condition. Second the shadow $y_k = P_{\cM}(x_k)$ of the iterates along the manifold form an approximate Riemannian stochastic gradient sequence with an implicit retraction. Moreover, the approximation error of the sequence decays with $\dist(x_k, \cM)$ and $\alpha_k$, quantities that quickly tend to zero.

The formal statements summarizing these two modes of behaviors require local arguments. Consequently, we will frequently refer to the following stopping time:
given an index $k \geq 1$ and a constant $\delta > 0$, define
\begin{align}\label{def:stoppingtime}
\tau_{\discrete, \delta} := \inf\{j \geq k \colon x_j  \notin B_{\delta}(\bar x)\}.
\end{align}
Note that the stopping time implicitly depends on $\bar x$, a point at which Assumption~\ref{assumption:Aproposed} is satisfied. The following proposition proved in \cite[Proposition 5.1]{davis2021subgradient} shows that sequence $x_k$ rapidly approaches the manifold. We note that \cite[Proposition 5.1]{davis2021subgradient} was stated specifically for optimization problems rather than for finding zeros of set-valued maps; the argument in this more general setting is identical however.

\begin{proposition}[Pillar I: aiming towards the manifold]\label{prop:gettingclosertothemanifold}
Suppose that Assumptions~\ref{ass:basic_assumpt}, \ref{assumption:localbound}, \ref{assumption:Aproposed},\ref{assumption:zero} hold. Let $\gamma \in (1/2, 1]$ and assume $c_1 \geq 32/\mu$ if $\gamma = 1$. Then for all $k_0 \geq 1$ and sufficiently small $\delta > 0$, there exists a constant $C$, such that the following hold with stopping time $\tau_{k_0, \delta}$ defined in~\eqref{def:stoppingtime}: 
\begin{enumerate}
\item \label{eq:prop:gettingclosertothemanifoldbound1}  There exists a random variable $V_{k_0, \delta}$ such that 
\begin{enumerate}
 \item \label{eq:prop:gettingclosertothemanifoldbound1:as} The limit holds: $$\frac{k^{2\gamma - 1}}{\log(k+1)^2}\dist^2(x_{k}, \cM)1_{\tau_{k_0, \delta} > k} \xrightarrow{\text{a.s.}}  V_{k_0, \delta}.$$
\item \label{eq:prop:gettingclosertothemanifoldbound1:sum} The sum is almost surely finite: $$\sum_{k=1}^\infty \frac{k^{\gamma - 1}}{\log(k+1)^2}\dist(x_{k}, \cM)1_{\tau_{k_0, \delta} > k} < +\infty.$$ 
\end{enumerate} 
\item \label{eq:prop:gettingclosertothemanifoldbound1p5}  We have
\begin{enumerate}
\item \label{eq:prop:gettingclosertothemanifoldbound1p5:expectedsquareddistance}The expected squared distance satisfies:
 $$
\EE[\dist^2(x_k, \cM)1_{\tau_{k_0, \delta} > k}]  \leq C\alpha_k \qquad \text{for all $k \geq 1$}.
$$
\item \label{eq:prop:gettingclosertothemanifoldbound1p5:saddle} The tail sum  is bounded:
$$
\EE\left[  \sum_{i=k}^\infty  \alpha_i \dist(x_i , \cM)1_{\tau_{k_0, \delta} > i}\right] \leq C\sum_{i=k}^\infty \alpha_i^2 \qquad \text{for all $k \geq 1$.}
$$
\end{enumerate}
\end{enumerate}
\end{proposition}

Next we study the evolution of  the shadow $y_k = P_{\cM}(x_k)$  along the manifold, showing that $y_k$ is locally an inexact Riemannian stochastic gradient sequence with error that asymptotically decays as $x_k$ approaches the manifold. Consequently, we may control the error using Proposition~\ref{prop:gettingclosertothemanifold}. The following proposition was proved in \cite[Proposition 5.2]{davis2021subgradient}. We note that \cite[Proposition 5.2]{davis2021subgradient} was stated specifically for optimization problems rather than for finding zeros of set-valued maps; the argument in this more general setting is identical however.

\begin{proposition}[Pillar II: the shadow iteration]\label{prop:shadow}
Suppose that Assumptions~\ref{ass:basic_assumpt}, \ref{assumption:localbound}, \ref{assumption:Aproposed},\ref{assumption:zero} hold. Then for all $k_0 \geq 1$ and sufficiently small $\delta > 0$, there exists a constant $C$, such that the following hold with stopping time $\tau_{k_0, \delta}$ defined in~\eqref{def:stoppingtime}: there exists a sequence of $\cF_{k+1}$-measurable random vectors $E_k \in \RR^d$ such that
\begin{enumerate}
\item \label{prop:shadow:part:sequence} The shadow sequence 
\begin{align*}
y_k = \begin{cases}
P_{\cM}(x_k) & \text{if $x_k \in B_{2\delta}(\bar x)$} \\
\bar x &  \text{otherwise.}
\end{cases}
\end{align*}
satisfies $y_k \in B_{4\delta}(\bar x) \cap \cM$ for all $k$ and the recursion holds:
\begin{align}\label{eqn:shadow:eq:iteration}
\boxed{y_{k+1} = y_k - \alpha_k F_{\mathcal{M}}(y_k) - \alpha_k P_{\tangentM{y_k}}(\perturb_k) + \alpha_k E_k \qquad \text{for all $k \geq 1.$}} 
\end{align}
Moreover, for such $k$, we have $\EE_k[P_{\tangentM{y_k}}(\perturb_k)] = 0$.
\item \label{prop:shadow:part:error}  Let $\gamma \in (1/2, 1]$ and assume that $c_1 \geq 32/\mu$ if $\gamma = 1$.  
\begin{enumerate}
\item\label{prop:shadow:part:error:part:upperbound} We have the following bounds for $k_0 \leq k \leq \tau_{k_0, \delta} -1$:
\begin{enumerate}
\item \label{prop:shadow:part:error:part:upperbound:2}$\max\{\EE_k [\|E_k\|1_{\tau_{k_0, \delta} > k}, \EE_k [\|E_k\|^21_{\tau_{k_0, \delta} > k}]\} \leq C$.
\item \label{prop:shadow:part:error:part:upperbound:3}$\EE[\|E_k\|^2 1_{\tau_{k_0, \delta} > k}] \leq C\alpha_k$.
\item\label{prop:shadow_missing} The sum is finite:
$$\sum_{k=1}^{\infty} \frac{k^{\gamma-1}}{\log(k+1)^2}\max\{\|E_k\|1_{\tau_{k_0},\delta>k},\EE_k[\|E_k\|]1_{\tau_{k_0},\delta>k}\}<+\infty.$$
\end{enumerate}
\item \label{prop:shadow:part:error:part:random:saddle} The tail sum is bounded
$$
\EE\left[1_{\tau_{k_0, \delta} = \infty}  \sum_{i=k}^\infty  \alpha_i\|E_k\|\right] \leq C\sum_{i=k}^\infty \alpha_i^2 \qquad \text{for all $k \geq 1$}.
$$
\end{enumerate}
\end{enumerate}
\end{proposition}

With the two pillars we separate our study of the sequence $x_k$ into two orthogonal components: In the tangent/smooth directions, we study the sequence $y_k$, which arises from an inexact gradient method with rapidly decaying errors and is amenable to the techniques of smooth optimization. In the normal/nonsmooth directions, we steadily approach the manifold, allowing us to infer strong properties of $x_k$ from corresponding properties for $y_k$.

		\bibliographystyle{plain}
	\bibliography{reference}
	\appendix
	
	\section{Proofs from Section~\ref{sec:notation}}

\subsection{Proof of Theorem~\ref{thm:smooth_invert_main} }\label{sec:appsmooth}
The proof of Theorem~\ref{thm:smooth_invert_main} follows from two  lemmas. The first allows one to reduce the sensitivity analysis of the inclusion $v\in A(x)+\partial f(x)$ to an entirely smooth setting. More precisely, the following basic result, proved in \cite[Proposition 10.12]{drusvyatskiy2014optimalityarxiv}, shows that as soon as $f$ admits an active manifold, the graph of the subdifferential $\partial f$ admits a smooth description.

\begin{lem}[Smooth reduction]\label{lem:smooth_reduct}
Let $f$ be a subdifferentially continuous function that admits a $C^2$ active manifold $\cM$ at a point $\bar x$ for a vector $\bar w\in \hat\partial f(\bar x)$. Then equality holds:
$$\gph \partial f=\gph \partial (f+\delta_{\cM})\qquad \textrm{locally around}\qquad (\bar x,\bar w).$$
\end{lem}

\noindent Note that letting $\hat f$ be a $C^2$-smooth function that agrees with $f$ on $\cM$ near $\bar x$, we may write
$$ \partial (f+\delta_{\cM})(x)=\partial (\hat f+\delta_{\cM})(x)=\nabla \hat f(x)+N_{\cM}(x).$$
Thus, under the same assumptions as in Lemma~\ref{lem:smooth_reduct}, equality :
$$\gph \partial f=\gph (\nabla \hat f+N_{\cM})\qquad \textrm{holds locally around}\qquad (\bar x,\bar w).$$
It follows from the lemma, that we may now focus on variational inclusions of the form $v\in \Phi(x)+N_{\cM}(x)$, where $\Phi$ and $\cM$ are smooth. Perturbation theory of such inclusions is entirely classical and is summarized in the following lemma.

\begin{lem}[Smooth variational inequality]\label{lem:basic_jac_soln}
Consider a set-valued map 
\begin{equation}\label{eqn:local_smooth_reduction}
F(x)=\Phi( x)+N_{\cM}(x)
\end{equation}
and the let $\bar x$ be a point satisfying $0\in F(\bar x)$. Suppose that  $\Phi\colon\R^d\to\R^d$ is a $C^{p}$-smooth map and $\cM\subset\R^d$ is a $C^{p+1}$-smooth manifold around $\bar x$. Let $G(x)=0$ be any $C^{p+1}$-smooth local defining equations for $\cM$ and define the map 
$$\mathcal{H}(x,y)=\Phi(x)+\nabla G(x)^{\top}y.$$
Then there exists a unique vector $\bar y$ satisfying the  condition $0=\mathcal{H}(\bar x,\bar y)$. Moreover, $F$ is $C^p$-invertible around $(0, \bar x)$ with inverse $\sigma(\cdot)$ if and only if 
$P_{T_{\cM}(\bar x)}\nabla_{x} \mathcal{H}( \bar x,\bar y)P_{T_{\cM}( \bar x)}$ is nonsingular on $T_{\cM}(\bar x)$. In this case,  equality holds:
 $$\nabla \sigma(0)=(P_{T_{\cM}(\bar x)}\nabla_{x} \mathcal{H}(\bar x, \bar y)P_{T_{\cM}( \bar x)})^{\dagger}.$$
\end{lem}
\begin{proof}
The existence of $\bar y$ follows from the expression $N_{\cM}(\bar x)=\range(\nabla G(\bar x)^\top)$, while uniqueness follows from surjectivity of $\nabla G(\bar x)$.

We first prove the backward implication and derive the claimed expression for the Jacobian. Suppose that $P_{T_{\cM}(\bar x)}\nabla_{x} \mathcal{H}(\bar x, \bar y)P_{T_{\cM}( \bar x)}$ is indeed nonsingular on $T_{\cM}( \bar x)$.
Then there exists $\epsilon>0$ such that for any $v\in \epsilon \mathbb{B}$ and $x\in B_{\epsilon}( \bar x)$, the inclusion $v\in \Phi(x)+N_{\cM}(x)$ holds if and only if there exists $y$ satisfying
\begin{equation}
\left\{\begin{aligned}
v&= \Phi(x) +  \nabla G(x)^{\top}y\\
0&= G(x)
\end{aligned} \right\}.
\end{equation}
Treating the right-hand-side as a mapping of $(x,y)$, its Jacobian at $( \bar x, \bar y)$ is given by 
\begin{equation}\label{eqn:invertible_sys}
\begin{bmatrix}
\nabla_{x} \mathcal{H}(\bar x,\bar y) & \nabla G( \bar x)^{\top}\\
\nabla G( \bar x) & 0
\end{bmatrix}.
\end{equation}
A quick computation shows that this matrix is invertible since
$P_{T_{\cM}( \bar x)}\nabla_{x} \mathcal{H}( \bar x, \bar y)P_{T_{\cM}( \bar x)}$ is nonsingular on $T_{\cM}( \bar x)$. Therefore, the inverse function theorem ensures that for all small $v$, the system \eqref{eqn:invertible_sys} admits a unique solution $\sigma(v)$ near $ \bar x$, and which varies $C^p$ smoothly in $v$. Inverting  \eqref{eqn:invertible_sys} yields the expression for the Jacobian
$\nabla \sigma(0)=(P_{T_{\cM}( \bar x)}\nabla_{x} \mathcal{H}( \bar x, \bar y)P_{T_{\cM}( \bar x)})^{\dagger}$.

Conversely, suppose that $F$ is $C^p$-smoothly invertible around $(0, \bar x)$ with inverse $\sigma(\cdot)$. Fix a vector $v\in \R^d$. Then for all sufficiently small $t>0$ there exists a unique vector $y(t)\in \R^d$ satisfying 
$$tv= \Phi(\sigma(tv))+\nabla G(\sigma(tv))^{\top}y(t).$$
Subtracting the equation $0=\Phi( \bar x)+\nabla G( \bar x)^{\top} \bar y$ and projecting both sides to $P_{T_{\cM}( \bar x)}$ yields
\begin{align*}
P_{T_{\cM}( \bar x)}v=P_{T_{\cM}( \bar x)}\left[\frac{\Phi(\sigma(tv))-\Phi( \bar x)}{t}\right]+P_{T_{\cM}( \bar x)}\left[\frac{\nabla G(\sigma(tv))^{\top}}{t}y(t)\right].
\end{align*}
It is straightforward to see that $y(t)$ is continuous, and therefore the right-side tends to
$P_{T_{\cM}( \bar x)}\nabla_x \mathcal{H}( \bar x, \bar y)\nabla \sigma(0)v$ as $t$ tends to zero. Summarizing, since $v$ is arbitrary, we have shown the matrix identity
$$P_{T_{\cM}( \bar x)}=P_{T_{\cM}( \bar x)}\nabla_x \mathcal{H}( \bar x, \bar y)\nabla \sigma(0).$$
Taking into account that the range of $\nabla \sigma(0)$ is contained in $T_{\cM}( \bar x)$, it follows that the range of $P_{T_{\cM}( \bar x)}\nabla_x \mathcal{H}( \bar x,\bar y)P_{T_{\cM}(\bar x)}$ must be equal to $T_{\cM}( \bar x)$. Therefore $P_{T_{\cM}( \bar x)}\nabla_x \mathcal{H}( \bar x, \bar y)P_{T_{\cM}( \bar x)}$ must be nonsingular on $T_{\cM}( \bar x)$, as claimed.
\end{proof}

We are now ready to complete the proof of Theorem~\ref{thm:smooth_invert_main}. Namely, Lemma~\ref{lem:smooth_reduct} together with continuity of $A(\cdot)$ directly imply that locally around $(\bar x,0)$, the graph of $F$ coincides with the graph of the map 
$$x\mapsto \Phi(x)+N_{\cM}(x),$$
where we set $\Phi(x)=A(x)+\nabla \hat f(x)$.  Lemma~\ref{lem:basic_jac_soln} directly implies that $F$ is $C^p$-invertible around $(0, \bar x)$ with inverse $\sigma(\cdot)$ if and only if 
$\Sigma=P_{T_{\cM}(\bar x)}\nabla_{x} \mathcal{H}( \bar x, \bar y)P_{T_{\cM}( \bar x)}$ is nonsingular on $T_{\cM}(\bar x)$. In this case,  equality $\nabla \sigma(0)=\Sigma^{\dagger}$ holds.

\section{Proofs from Section~\ref{sec:smooth_invertible_maps}}
\subsection{Proof of Theorem~\ref{thm:asympt_covar}}\label{sec:appsample_average_approx}
Define the random vectors $w_k:=A(x_k)-A_S(x_k)$ and define the events 
$$\mathcal{E}_k:=\{x_k\in B_{\epsilon_2}(\bar x)\}\qquad \textrm{and}\qquad \mathcal{Z}_k=\{w_k\in B_{\epsilon_1}(0)\}.$$ 
Note that $1_{\mathcal{E}_k}\xrightarrow[]{p}1$ by our assumptions.
The very definition of $x_k$ implies the inclusion $w_k \in (A+H)(x_k)$. Therefore, the equality holds:
\begin{equation}\label{eqn:basic_science}
\sqrt{k}(x_k-\bar x)1_{\mathcal{E}_k\cap \mathcal{Z}_k}=\sqrt{k}[\sigma(w_k1_{\mathcal{E}_k\cap \mathcal{Z}_k})-\sigma(0)].
\end{equation}
A application of the delta method (Lemma~\ref{lem:basic_facts_conv}(\ref{basic_facts_conv5})) 
shows that as long as $\sqrt{k}w_k1_{\mathcal{E}_k\cap \mathcal{Z}_k}\xrightarrow[]{D}\mathsf{N}\big(0,\Sigma)$, for some matrix $\Sigma$, then the right-hand-side of \eqref{eqn:basic_science} converges in distribution to $\mathsf{N}\big(0,\nabla \sigma(0)\Sigma\nabla \sigma(0)^{\top})$. 
Our task therefore reduces to computing the limit of $\sqrt{k}w_k1_{\mathcal{E}_k\cap \mathcal{Z}_k}$.

Let us first show $1_{\mathcal{Z}_k}\xrightarrow[]{p}1$.
A first-order expansion of  $A(\cdot)$ and $A(\cdot,z_i)$ around $\bar x$ yields
\begin{equation}\label{eqn:basic_expans}
\begin{aligned}
\| w_k&-\underbrace{(A(\bar x)-A_S(\bar x))}_{O_P(1/\sqrt{k})}+\underbrace{(\nabla A(\bar x)-\nabla A_S(\bar x))(x_k-\bar x)}_{O_P(1/\sqrt{k})}\|1_{\mathcal{E}_k}\\
&\quad\leq\tfrac{1}{2}\Big(\EE L+\underbrace{\frac{1}{k}\sum_{i=1}^k L(z_i)}_{\EE L+O_P(1/\sqrt{k})}\Big) \|x_k-\bar x\|^21_{\mathcal{E}_k},
\end{aligned}
\end{equation}
where the statements in brackets follow from the central limit theorem. Rearranging, yields
$$\| w_k1_{\mathcal{E}_k}\|\leq  \epsilon_2^2\EE L+O_P(1/\sqrt{k})\leq \frac{\epsilon_1}{2}+O_P(1/\sqrt{k}).$$
We conclude $1_{\mathcal{Z}_k\cap \mathcal{E}_k}\xrightarrow[]{p}1$. Taking into account $1_{\mathcal{E}_k}\xrightarrow[]{p}1$, we  deduce $1_{\mathcal{Z}_k}\xrightarrow[]{p}1$, as claimed.

Next, we show $\|x_k-\bar x\|=O_P(1/\sqrt{k})$.
Observe that  \eqref{eqn:basic_science} directly implies
$\|x_k-\bar x\|1_{\mathcal{E}_k\cap \mathcal{Z}_k}\leq \lip(\sigma) \|w_k1_{\mathcal{E}_k\cap \mathcal{Z}_k}\|.$
Combining this with \eqref{eqn:basic_expans}, after multiplying through by $1_{\mathcal{Z}_k}$, we deduce
$$\lip(\sigma)^{-1}\|x_k-\bar x\|1_{\mathcal{E}_k\cap \mathcal{Z}_k}= O_P(1/\sqrt{k})+(\EE L+O_P(1/\sqrt{k}))\|x_k-\bar x\|^21_{\mathcal{E}_k\cap \mathcal{Z}_k}.$$
Rearranging, yields
$$(\lip(\sigma)^{-1}-\EE L\|x_k-\bar x\|)\cdot\|x_k-\bar x\|1_{\mathcal{E}_k\cap \mathcal{Z}_k}=O_P(1/\sqrt{k}).$$
Noting that the coefficient on the left hand side is bounded below by $1/2$ in the event $\mathcal{E}_k$, we  deduce $\|x_k-\bar x\|1_{\mathcal{E}_k\cap \mathcal{Z}_k}=O_P(1/\sqrt{k})$. Taking into account $1_{\mathcal{E}_k\cap \mathcal{Z}_k}\xrightarrow[]{p}1$, we deduce $\|x_k-\bar x\|=O_P(1/\sqrt{k})$ as claimed.

Finally, multiplying \eqref{eqn:basic_expans} through by $\sqrt{k}$ yields
\begin{align*}
\sqrt{k} w_k1_{\mathcal{E}_k}&=\underbrace{\sqrt{k}(A(\bar x)-A_S(\bar x))1_{\mathcal{E}_k}}_{\xrightarrow[]{D} \mathsf{N}\left(0,{\rm Cov}(A(\bar x,z)\right)}+\underbrace{(\nabla A(\bar x)-\nabla A_S(\bar x))\sqrt{k}(x_k-\bar x)1_{\mathcal{E}_k}}_{\xrightarrow[]{p}0}\\
&\quad+\underbrace{\frac{1}{2}\left(\EE L(z)+\frac{1}{k}\sum_{i=1}^k L(z_i)\right) O(\sqrt{k}\|x_k-\bar x\|^2)1_{\mathcal{E}_k}}_{\xrightarrow[]{p}0},
\end{align*}
where the claimed limits follow from the central limit theorem and the fact that $\sqrt{k}\|x_k-\bar x\|$ is bounded in probability. Multiplying through by $1_{\mathcal{Z}_k}$, we deduce $\sqrt{k}w_k1_{\mathcal{E}_k\cap \mathcal{Z}_k}\xrightarrow[]{D} \mathsf{N}\big(0,{\rm Cov}(A(\bar x,z))).$ Thus returning to \eqref{eqn:basic_science} and the application of the Delta method, we see that $\sqrt{k}(x_k-\bar x)1_{\mathcal{E}_k\cap \mathcal{Z}_k}\xrightarrow[]{D} \mathsf{N}\big(0,\nabla \sigma(0){\rm Cov}(A(\bar x,z))\nabla \sigma(0)^{\top})$.
It remains to note that $\sqrt{k}(x_k-\bar x)1_{\mathcal{E}_k\cap \mathcal{Z}_k}$ and  $\sqrt{k}(x_k-\bar x)$ have the same limiting distribution since $1_{\mathcal{E}_k\cap \mathcal{Z}_k}\xrightarrow[]{p}1$.

\subsection{Proof of Theorem~\ref{thm:main_opt}}\label{sec:local_minimax}
Throughout the section, we suppose that Assumption~\ref{assump:smooth_inverse} and \ref{base_assumpt} hold.
We now show that the asymptotic covariance in \eqref{thm:asympt_covar} is the best possible among all estimators of $\bar x$. Namely, we will lower-bound the performance of {\em any} estimation procedure for finding an equilibrium point on an adversarially-chosen sequence of small perturbations of the target problem. 
In order to define this sequence, define the set  
$$\mathcal{G}:=\{g\colon\mathcal{Z}\to\R^d: \mathop\EE_{z\sim\cP}[g(z)]=0,~\mathop\mathbb{E}_{z\sim \cP}\|g(z)\|^2<\infty\}.$$
Fix now a function $g\in \cG$ and an arbitrary $C^{3}$-smooth function $h \colon \RR \rightarrow [-1, 1]$ such that its first three derivatives are bounded and $h(t) = t$ for all $t \in [-1/2, 1/2]$. Now for each $u \in \RR^{d}$, define a new probability distribution $\cD^{u}$ whose density is given by
\begin{equation}
  \label{eq:perturbation-of-the-force}
d\cP_u(z) := \frac{1+h\big(u^{\top} g(z)\big)}{C(u)}\,d\cP(z),
\end{equation}
where  $C(u)$ is the normalizing constant  $C(u) := 1 + \int h\big(u^{\top} g(z)\big)\,d\cP(z)$. Thus each vector $u\in\R^d$ induces the perturbed problem 
\begin{equation}\label{eqn:VIII_perturb}
0\in L(x,u)+H(x) \qquad \textrm{where }\qquad L(x,u)=\mathop\EE_{z\sim \cP_u} A(x,z).
\end{equation}

Reassuringly, the following lemma shows that map $(x,u)\mapsto L(x,u)$ is $C^1$ near $(\bar x,0)$.

\begin{lem}\label{lem:deriv_formula}
The map $(x,u)\mapsto L(x,u)$ is $C^1$ near $(\bar x,0)$ with partial derivatives 
$$\nabla_x L(\bar x,0)=\nabla A(\bar x)\qquad \textrm{and}\qquad \nabla_u L(x,0)=\mathop\mathbb{E}_{z\sim \cP}A(\bar x,z)g(z)^{\top}.$$
\end{lem}
\begin{proof}
Let us first consider the normalizing constant $C(u)=1 + \int h\big(u^{\top} g(z)\big)\,d\cP(z)$. The dominated convergence theorem\footnote{using that $h'$ and $h''$ are bounded and $\EE_{z\sim \cP}\|g(z)\|^2<\infty$} directly implies that $C(\cdot)$ is twice differentiable with $\nabla C(u)=\int h'\big(u^{\top} g(z) \big)g(z)^{\top}\,d\cP(z)$ and $\nabla^2 C(u)=\int h''\big(u^{\top} g(z) \big)g(z)g(z)^{\top}\,d\cP(z)$. Moreover, the dominated convergence theorem\footnote{using that there is a neighborhood $U$ of $\bar x$ such that $\sup_{x\in U}\|\nabla A(x,z)\|^2$ is integrable.} implies that $A(x)$ is $C^1$-smooth with $\nabla A(x)=\EE_{z\sim \cP} \nabla A(x,z)$. 
Thus it now suffices to argue that $\hat L(x,u):=\int h(u^{\top} g(z))A(x,z) \,d\cP(z)$ is $C^1$-smooth.
An application of the dominated convergence theorem in $u$ directly implies that $\hat L(x,u)$ is differentiable in $u$ with 
$\nabla_u \hat L(x,u)=\int h'(u^{\top} g(z))A(x,z)g(z)^{\top} \,d\cP(z)$ and moreover $\nabla_u \hat L(x,u)$ is continuous in $(x,u)$. Similarly, the dominated convergence theorem\footnote{using that $h$ is bounded and there is a neighborhood $U$ of $\bar x$ such that $\sup_{x\in U}\|\nabla A(x,z)\|^2$ is integrable.} implies that $\hat L(x,u)$ is differentiable in $x$ with $\nabla_x \hat L(x,u)=\int h(u^{\top} g(z))\nabla A(x,z) \,d\cP(z)$ and $\nabla_x \hat L(x,u)$ is continuous in $(x,u)$. Thus $L(\cdot,\cdot)$ is $C^1$-smooth near $(\bar x,u)$. Observe that the expression $\nabla_x L(\bar x,0)=\nabla A(\bar x)$ follows trivially since $L(x,0)\equiv A(x)$ for all $x$. To see the expression for $\nabla_u L(x,0)$, observe that $\nabla C(0)=0$ and $\nabla^2 C(0)=0$ and therefore $C(u)=1+o(\|u\|^2)$. It follows immediately that $\nabla_u L(x,0)=\nabla_u \hat L(x,0)$ for all $x$, thereby completing the proof.
\end{proof}

The family of problems~\ref{eqn:VIII_perturb} would not be particularly useful if their solution would vary wildly in $u$. On the contrary, the following lemma shows that all small $u$, each problem \ref{eqn:VIII_perturb} admits a unique solution in $U$, which moreover varies smoothly in $u$. We will use the following standard notation. A map $\sigma(\cdot)$ is called a {\em localization} of a set-valued map $F(\cdot)$ around a pair $(\bar u,\bar v)\in \gph F$ if the two sets, $\gph \sigma$ and $\gph F$, coincide locally around $(\bar u,\bar v)$.

\begin{lem}[Derivative of the solution map]\label{lem:der_soln}
The solution map 
$$S(u)=\{x: 0\in L(x,u)+H(x)\}.$$
admits a single-valued localization $s(\cdot)$ around around $(0,\bar x)$ that is differentiable at $0$ with Jacobian
$$\nabla s(0)=- \nabla\sigma(0)\cdot \mathop\mathbb{E}_{z\sim \cP}[A(\bar x,z)g(z)^{\top}].$$
\end{lem}
\begin{proof}
Assumption \ref{assump:smooth_inverse} ensures that the map $A+H$ is $C^1$ invertible around $(\bar x,0)$ with some inverse $\sigma(\cdot)$. Define now the linearization $\Psi(x):=A(\bar x)+\nabla A(\bar x)(x-\bar x)$ of $A$ at $\bar x$. Invoking \cite[Theorem 2B.10]{dontchev2009implicit}, we deduce that the map $\Psi+H$ is also $C^1$ invertible around $(0,\bar x)$ with inverse $\hat \sigma$  and which satisfies $\nabla\hat \sigma(0)=\nabla\sigma(0)$. Note that in light of Lemma~\ref{lem:deriv_formula}, we may equivalently write $\Psi$ as $\Psi(x):=L(\bar x,0)+\nabla_x L(\bar x,0)(x-\bar x)$. Applying \cite[Theorem 2D.6]{dontchev2009implicit}, we deduce that the map $S(u)$ admits a single-valued localization $s(\cdot)$ around $(0,\bar x)$ that is differentiable at $0$ and satisfies $\nabla s(0)=-\nabla \hat \sigma(0)\circ \nabla_u L(\bar x,0)$. An application of Lemma~\ref{lem:deriv_formula} completes the proof.
\end{proof}

In light of Lemma~\ref{lem:der_soln}, for all small $u$, we define the solution $\bar x_u:=s(u)$.  
The rest of the section is devoted to proving the following theorem. We let $\EE_{P_{u}^k}$ denote the expectation with respect to $k$ i.i.d. observations $z_i\sim \cP_u$.

\begin{thm}[Local minimax]\label{thm:optimality}
Let $\cL\colon \RR^{d} \rightarrow [0,\infty)$ be symmetric, quasiconvex, and lower semicontinuous, let $\widehat x_{k}\colon \cZ^{k}\rightarrow U$ be a sequence of estimators, and set $g(x,z):=A(\bar x,z)-A(x)$. Then the inequality
  \begin{equation}
    \label{eq:4}
    \lim_{c\to\infty}\liminf_{k \rightarrow \infty} \sup_{\|u\|\leq \sqrt{c/k}} {\EE}_{P_{k,u}}\big[\cL\big(\sqrt{k} (\widehat{x}_{k} - \bar x_{u})\big)\big] \geq {\EE}\big[\cL(Z)\big]
  \end{equation}
  holds, where $Z\sim \mathsf{N}\big(0, \nabla\sigma(0) \cdot {\rm Cov}(A(\bar x,z)) \cdot \nabla\sigma(0)^{\top}\big)$.
 \end{thm}

Theorem~\ref{thm:optimality} directly implies Theorem~\ref{thm:main_opt}. Indeed, as pointed out in \cite[Section 3.2]{duchi2021asymptotic},  for any given $C>0$, there exists $C>0$ such that the condition $\|u\|\leq \sqrt{C/k}$ implies  $\cP_{u} \in B_{c/k}$ for all sufficiently large $k$.

\subsubsection{Proof of Theorem~\ref{thm:optimality}}
The proof of Theorem~\ref{thm:optimality} will be based on the local minimax theorem of H\'ajek and Le Cam, which appears for example in \cite[Theorem 6.6.2]{le2000asymptotics}. The theorem relies on two concepts: (1) local asymptotic normality of a sequence of distributions and (2) regularity of a sequence of mappings. The exact definition of the former will not be important for us due to reasons that will be clear shortly. The latter property, however, will be crucially used.

\begin{definition}[Regular mapping sequence]
  Let $W \subset \RR^{n}$ be a neighborhood of zero. A sequence of mappings $\{\Gamma_{k}\colon W \rightarrow \RR^{d}\}$ is \emph{regular} with derivative $D \in \RR^{d \times n}$ if
  $$\lim_{k\rightarrow \infty}\sqrt{k} \big( \Gamma_{k}(u) - \Gamma_{k}(0)\big) = D u \qquad \text{for all } u \in W$$
\end{definition}

We can now state the local minimax theorem.

\begin{thm}\label{thm:localminimax}
  Let $\{\Omega_{k}, \cF_{k}, Q_{{k,u}}\}_{u \in \RR^d}$ be a locally asymptotically normal family with precision $Q \succeq 0$ and let $\{\Gamma_{k}\colon \RR^d\rightarrow \RR^{d}\}$ be regular sequence with derivative $D$. Let $\cL\colon \RR^{d} \rightarrow [0,\infty)$ be symmetric, quasiconvex, and lower semicontinuous. Then, for any sequence of estimators $T_{k}\colon \Omega_{k} \rightarrow \RR^{d}$, we have
  \begin{equation}\label{eq:lecam}
  \sup_{U_{0} \subset \RR^d, |U_{0}| < \infty}\liminf_{k \rightarrow \infty} \max_{u \in U_{0}}\, {\EE}_{Q_{k, u}}  \big[\cL \big(\sqrt{k}\big(T_{k} - \Gamma_{k}(u)\big)\big)\big] \geq {\EE}\big[\cL(Z)\big],
  \end{equation}
  where $Z \sim \mathsf{N}\big(0, D Q^{{-1}}D^{\top}\big)$ when $Q$ is invertible; if $Q$ is singular, then \eqref{eq:lecam} holds with $Z \sim \mathsf{N}\big(0, D (Q + \lambda I)^{{-1}}D^{\top}\big)$ for all $\lambda>0$.
  \end{thm}

We aim to apply Theorem~\ref{thm:localminimax} as follows.
For each $u$, we take $(\Omega_{k}, \cF_{k}, Q_{{k,u}})$ to be the $k$-fold product of the probability spaces $(\Omega, \mathcal{F},\cP_{u/\sqrt{k}})$ and set $\Gamma_k(u)=s(u/\sqrt{k})$. It was shown in \cite[Lemma 8.3]{duchi2021asymptotic} that the the sequence $\{\Omega_{k}, \cF_{k}, Q_{{k,u}}\}_{u \in \R^d}$ is locally asymptotically normal with precision  $D=\underset{z \sim \cP}{\mathop{\EE}}\big[g(z)g(z)^{\top}\big].$ Moreover, the following lemma establishes regularity of the sequence $\Gamma_k$.

\begin{lem} \label{lem:smooth-kappa}
  The sequence $\Gamma_k\colon \RR^d\to \RR^d$ is regular with derivative $- \nabla\sigma(0)\cdot \mathop\mathbb{E}_{z\sim \cP}[g(z)A(\bar x,z)^{\top}]$.
\end{lem}
\begin{proof}
Using Lemma~\ref{lem:der_soln}, a first-order expansion of $s(\cdot)$ around $\bar x$ yields 
$$\sqrt{k}(\Gamma_k(u)-\Gamma_k(0))=\sqrt{k}(s(u/\sqrt{k})-\bar x)=- \nabla\sigma(0)\cdot \mathop\mathbb{E}_{z\sim \cP}[g(z)A(\bar x,z)^{\top}] u+\frac{o(k^{-1/2})}{k^{-1/2}}.$$
Letting $k$ tends to infinity completes the proof.
\end{proof}

We now apply Theorem~\ref{thm:localminimax}. Let $\cL\colon \RR^{d} \rightarrow [0,\infty)$ be symmetric, quasiconvex, and lower semicontinuous, $\widehat x_{k}\colon \cZ^{k}\rightarrow \RR^{d}$ be a sequence of estimators, and $c>0$. Set 
\begin{align*}
g(z)&:=A(\bar x,z)-A(\bar x),\\
\Sigma&:=\underset{z \sim \cP}{\mathop{\EE}}\big[g(z)g(z)^{\top}\big],\\
W&:=\nabla\sigma(0).
\end{align*}
To connect the inequality \eqref{eq:lecam} to \eqref{eq:4}, observe that for any finite subset $U_0\subset \RR^d$, we have
  \begin{equation*}
\liminf_{k \rightarrow \infty} \sup_{\|u\|\leq \sqrt{c/k}} {\EE}_{P_{k,u}}\big[\cL\big(\sqrt{k} (\widehat{x}_{k} - \bar{x}_{u})\big)\big] \geq \liminf_{k \rightarrow \infty}\, \max_{u\in U_0}\, {\EE}_{Q_{k,u}}\big[\cL\big(\sqrt{k} \big(\widehat{x}_{k} - \Gamma_k(u)\big)\big)\big].
  \end{equation*}
  where $c=\max_{u\in U_0}\|u\|^2$.
Taking the supremum over all finite $U_0\subset \RR^d$ and applying Theorem~\ref{thm:localminimax}  yields
\begin{equation}\label{eq:4-almost}
	\lim_{c\to \infty}\liminf_{k \rightarrow \infty} \sup_{\|u\|\leq \sqrt{c/k}} {\EE}_{P_{k,u}}\big[\cL\big(\sqrt{k} (\widehat{x}_{k} - \bar x_{u})\big)\big] \geq {\EE}\big[\cL(Z_{\lambda})\big]
\end{equation}
where $Z_{\lambda} \sim \mathsf{N}\big(0,W\Sigma(\Sigma + \lambda I)^{{-1}} \Sigma^{\top} W^{\top}\big)$ for any $\lambda > 0$. Basic linear algebra shows 
$$\lim_{\lambda\downarrow0} \Sigma\big(\Sigma + \lambda I\big)^{{-1}}\Sigma =\Sigma.$$
A straightforward argument based on the monotone convergence theorem (see e.g. \cite[Section 5.1.2]{cutler2022stochastic}) therefore implies that the right side of 
\eqref{eq:4-almost} tends to ${\EE}\big[\cL(Z)\big]$ as $\lambda\downarrow 0$, where $Z\sim \mathsf{N}(0,W\Sigma W^{\top})$.
The proof is complete.

\section{Proofs from Section~\ref{assumptions:for_algos}}

\subsection{Proof of Lemma~\ref{lem:basic_level_bound}}\label{sec:lemma_bounded_app}
Throughout the proof, let $\epsilon>0$ and $L$ be such that 
$$\max\{\|s_g(x)\|, \|A(x)\|\}\leq L\qquad \forall x\in B_{\epsilon}(\bar x).$$
To see  Claim~\eqref{lb:2}, for all $x$ sufficiently close to $\bar x$, we compute 
\begin{align*}
\alpha\|G_{\alpha}(x, \perturb)\|&=\|x-s_f(x-\alpha(A(x)+s_g(x)+\nu))\|\\
&\leq \dist_{\cX}(x-\alpha(A(x)+s_g(x)+ \perturb))+\alpha\|A(x)+s_g(x)+ \perturb\|\\
&\leq 2\alpha\|A(x)+s_g(x)+ \perturb\|\\
&=  2\alpha (2L +\|\perturb\|).
\end{align*}

Throughout the rest of the proof, we set $x_+:=x-\alpha G_{\alpha}(x,\nu)$.
We now verify~Claim~\ref{lb:3}. To this end, suppose that $f$ is convex and by increasing $L$ we may ensure $\dist(0,\partial f(x))\}\leq L$ for all $x\in B_{\epsilon}(\bar x)\cap \dom f.$
Choose a vector $z\in \partial f(x)$ of minimal length. Algebraic manipulations show 
$x=\prox_{\alpha f}(x+\alpha z)$. Since $\prox_{\alpha f}$ is nonexpansive, for all $x\in B_{\epsilon}(\bar x)$, we deduce 
\begin{align*}
\alpha\|G_{\alpha}(x, \perturb)\|&=\|x-\prox_{\alpha f}(x-\alpha(A(x)+s_g(x)+\nu))\|\\
&\leq \alpha \|A(x)+s_g(x)+z+\nu\|\\
&\leq \alpha(3L+\|\nu\|),
\end{align*}
as claimed.

We next verify Claim~\ref{lb:4}. To this end, suppose that $f$ is $L$-Lipschitz continuous on $\dom g\cap \dom f$.  Set $w := A(x) + s_g(x)+\perturb$ and observe that the very definition of $x_+$ ensures
\begin{align*}
\frac{1}{2\alpha}\|x_+ - x\|^2  &\leq f(x) - f(x_+) - \dotp{w, x_+ - x} \\
&\leq L\|x_+ - x\| + \|w\|\|x_+ -x\|.
\end{align*}
Consequently, for all $x\in B_{\epsilon}(\bar x)$ we have 
$
\alpha^{-1}\|x_+ - x\| \leq 2(L+\|w\|) \leq 2(3\localconstant + \|\perturb\|),
$
as desired.

\subsection{Proof of Lemma~\ref{lem:loc_tan_redux_app}}\label{sec:loc_tan_app}
Strong $(a)$-regularity implies the equalities, $P_{T_{\cM}(x)}\partial f(x)=\{\nabla_{\cM} f(x)\}$  and $P_{T_{\cM}(x)}\partial g(x)=\{\nabla_{\cM} g(x)\}$, for all $x\in \cM$ near $\xs$. The claimed expression~\eqref{eqn:simple_F} follows immediately. 
Next, Lemma~\ref{lem:smooth_reduct} implies
that $\gph \partial (f + g)$ coincides with $\gph [\nabla_{\cM} (f+g)+N_{\cM}]$ around $(\xs,-A(\xs))$. Thus locally around $(\xs,0)$ equalities hold:
\begin{align*}
\gph [A+\partial (f + g)]&=\gph [A+\nabla_{\cM} f+\nabla_{\cM} g+N_{\cM}]=\gph (F_{\cM}+N_{\cM}),
\end{align*}
as claimed.

\subsection{Proof of Proposition~\ref{prop:subgradient}}
\begin{proof}
Fix $x\in \cX$, $\alpha>0$, and $\nu\in \R^d$.  Assumption~\ref{assumption:localbound} holds trivially, and follows for example from Lemma~\ref{lem:basic_level_bound}(\ref{lb:2}). In order to verify Assumption~\ref{assumption:smoothcompatibility}, we compute
\begin{align*}
\|P_{\tangentM{P_{\cM}(x)}}(G_\alpha(x, \perturb) - F(P_{\cM}(x)) - \nu)\| &=  \|P_{\tangentM{P_{\cM}(x)}}(A(x)+s_g(x) - F(P_{\cM}(x)))\|.
\end{align*}
Therefore, for $x$ sufficiently close to $\bar x$ we may upper bound the right-hand-side as  
\begin{align*}
 &\|P_{\tangentM{P_{\cM}(x)}}[s_g(x)-\nabla_{\cM} g(P_{\cM}(x))]+P_{\tangentM{P_{\cM}(x)}}[A(x)-A(P_{\cM}(x))]\|\\
 &\leq C\cdot \dist(x,\cM),
\end{align*}
where the inequality follows from~\ref{assumption:projectedgradient:stronga0} and local Lipschitz continuity of $A(\cdot)$. Thus Assumption~\ref{assumption:smoothcompatibility} holds. Finally, to see Assumption~\ref{assumption:aiming}, we compute
\begin{align*}
\dotp{G_\alpha(x, \perturb) - \nu, x - P_{\cM}(x)}&=\dotp{A(\bar x)+s_g(x), x - P_{\cM}(x)}+\dotp{A(x)-A(\bar x), x - P_{\cM}(x)}\\
&\geq \mu\cdot \dist(x,\cM)-\|A(x)-A(\bar x)\|\cdot\dist(x,\cM)\\
&\geq \tfrac{\mu}{2}\cdot \dist(x,\cM),
\end{align*}
for all $x$ sufficiently close to $\bar x$. The proof is complete.
\end{proof}

\subsection{Proof of Proposition~\ref{prop:projectedgradient}}
Choose $\epsilon > 0$ small enough that the following hold for all $x \in B_{\epsilon}(\bar x) \cap \cX$. First \eqref{prop:projectedgradient:eq:aiming} holds. In particular, since $A(\cdot)$ is locally Lipschitz near $\bar x$, we may be sure that 
\begin{equation}\label{prop:projectedgradient:eq:aiming3}
\dotp{A(x)+v, x - P_{\cM}(x)} \geq \frac{\mu}{2}\cdot  \dist(x, \cM),
\end{equation}
for all $v \in \partial g(x)$.
 Second we require that for some $L > 0$, we have
\begin{align}
\|P_{\tangentM{P_{\cM}(x)}}(s_g(x) - \nabla_{\cM} g(P_{\cM}(x))\| &\leq L\cdot \dist(x, \cM),\label{eqn:dumbo0}\\
\|P_{\tangentM{z}}(u)\| &\leq L \|x-z\|, \label{eqn:dumbo}
\end{align}
for all $u \in N_{\cX}(x)$ of unit norm and all $z\in B_{\epsilon}(\bar x)\cap\cM$, a consequence of \ref{assumption:projectedgradient:stronga}. Third, we may choose $\epsilon>0$ so small so that
\begin{equation}\label{eqn:blah_prox}
\dotp{z, x - x'} \geq o(\|x-x'\|) 
\end{equation}
for all $z \in N_{\cX}(x)$ of unit norm, and $x'\in \cM\cap B_{\epsilon}(\bar x)$---a consequence of \ref{assumption:projectedgradient:bproxregularity}. 
We will fix $x\in B_{\epsilon/2}(\bar x)\cap \cX$ and arbitrary $\alpha > 0$ and $\nu \in \RR^d$, and choose an arbitrary $y\in\proj_{\cM}(x)$. Define 
\begin{align*}
w = G_\alpha(x, \perturb) - \perturb- A(x)-s_g(x)  \qquad \text{ and } 
\qquad x_+ = s_\cX(x - \alpha(A(x)+s_g(x)+ \perturb)). 
\end{align*}
Note the inclusion $w\in N_{\cX}(x^+)$. Moreover, shrinking $\epsilon>0$ Assumption~\ref{assumption:localbound} directly implies  
$$
\max\{\|w\|,\|G_{\alpha}(x, \perturb)\|\} \leq C(1 + \|\perturb\|)  \qquad \text{and} \qquad \|x_+ - x\| \leq C(1 + \|\perturb\|)\alpha,
$$
for some constant $C>0$.
We will use these estimates often in the proof. Finally, we let $C$ be a constant independent of $x, \alpha$ and $\nu$, which changes from line to line.

\noindent\underline{Assumption~\ref{assumption:smoothcompatibility}:}
Suppose first that $x_+\in B_{\epsilon}(\bar x)$. Using \eqref{eqn:dumbo}, we compute
\begin{align}\label{eq:projgradientverdierbound1}
\|P_{\tangentM{y}}w\|&\leq L\|w\| \|x_+ - y\|\notag \\
&\leq L\|w\|(\|x_+ - x\| + \dist(x, \cM)) \notag \\
&\leq \localconstant(1+ \|\perturb\|)^2\alpha + \localconstant(1 + \|\perturb\|)\dist(x, \cM).
\end{align}
On the other hand, if $x_+\notin B_{\epsilon}(\bar x)$, then we compute \begin{align}\label{eq:projgradientverdierbound3}
\|P_{\tangentM{y}}w\| \leq \|w\| \leq \frac{2}{\epsilon}\|w\|\|x_+ - x\| \leq \localconstant(1 + \|\perturb\|)^2\alpha.
\end{align}
In either case, Assumption~\ref{assumption:smoothcompatibility} now follows since from \eqref{eqn:dumbo0} we have
\begin{align*}
\|P_{\tangentM{y}}(A(x)+s_g(x) - F(y)\|&\leq
\|P_{\tangentM{y}}( \nabla_\cM g(y) -s_g(x))\|\\
&~~~+\|P_{\tangentM{y}}[A(x)-A(y)]\| \\
&\le \localconstant \dist(x, \cM),
\end{align*}
as we had to show.

\noindent\underline{Assumption~\ref{assumption:aiming}:} We write the decomposition
\begin{equation}\label{eqn:decomp}
\langle G_{\alpha}(x,\nu)-\nu, x-y\rangle=\underbrace{\langle A(x)+s_g(x),x-y\rangle}_{R_1}+\underbrace{\langle w,x_+-y\rangle}_{R_2}+\underbrace{\langle w,x-x_+\rangle}_{R_3}.
\end{equation}
The estimate \eqref{prop:projectedgradient:eq:aiming3} gives 
\begin{equation}\label{eqn:dumb3}
\begin{aligned}
R_1& \geq \tfrac{\mu}{2}\cdot\dist(x,\cM).
\end{aligned}
\end{equation}
We next look at two cases. Suppose first $x_+\in B_{\epsilon}(\bar x)$.
Using the inclusion $w\in N_{\cX}(x_+)$ and \eqref{eqn:blah_prox}, we compute 
\begin{align}
R_2\geq \|w\|\cdot o(\|x_+-y\|)&\geq \|w\|\cdot (o(\|y-x\|)-\|x-x_+\|)\notag\\
&\geq -C(1+\|\nu\|)^2(o(\dist(x,\cM))+\alpha)\label{eqn:dumb1}.
\end{align}
Next, the Cauchy–Schwarz inequality implies
\begin{equation}\label{eqn:dumb2}
|R_3|\leq\|w\| \|x-x_+\|\leq C(\alpha(1+\|\nu\|)^2).
\end{equation}
 Combining \eqref{eqn:decomp}-\eqref{eqn:dumb2} yields the claimed bound \ref{assumption:aiming}.

Suppose now on the contrary that $x_+\notin B_{\epsilon}(\bar x)$ and therefore $\|x-y\|\leq \|x-x_+\|$. We thus deduce $R_2+R_3=\langle w,x-y\rangle\geq -\|w\|\|x-y\|\geq -C\alpha(1+\|\nu\|)^2$ holds. Combining this estimate with \eqref{eqn:decomp} and \eqref{eqn:dumb3} verifies the claim \ref{assumption:aiming}.

\subsection{Proof of Proposition~\ref{prop:proximalgradient}}
Let $\epsilon \in (0,1)$ be small enough such that the following hold for all $x \in B_{\epsilon}(\bar x) \cap \dom f$. First \eqref{prop:proximalgradient:eq:aiming} holds and therefore taking into account Lipschitz continuity of $A(\cdot)$, we may equivalently write 
\begin{equation}\label{eqn:angle}
\dotp{A(x)+v, x - P_{\cM}(x)} \geq \frac{\mu}{2}\cdot  \dist(x, \cM) - (1+\|v\|)o(\dist(x, \cM)),
\end{equation}
for all $v \in \hat\partial f(x)$.
Second we require that for some $L > 0$, we have
\begin{equation}\label{eqn:basic_est_needed}
\|P_{\tangentM{P_{\cM}(x)}}(u - \nabla_{\cM} f(P_{\cM}(x))\| \leq L\sqrt{1+ \|u\|^2}\cdot\dist(x, \cM)
\end{equation}
for all $u \in \partial f(x)$, a consequence of strong (a) regularity. Third, we assume that $\nabla_{\cM} f$ is $L$-Lipschitz on $B_{\epsilon}(\bar x) \cap \cM$. Fourth, we assume that $A(\cdot)$ is $L$-Lipschitz. Shrinking $\epsilon$ we may moreover assume $\epsilon\leq \frac{\mu}{8L}$.
Finally, we may also assume that the maps $x\mapsto P_{\cM}(x)$ and $x\mapsto P_{T_{\cM}(P_{\cM}(x))}$ are Lipschitz continuous on $B_{\epsilon}(\bar x)$.

Fix $x \in B_{\epsilon/2}(\bar x)\cap \dom f$ and $\nu \in \RR^d$ and set $y:=P_{\cM}(x)$. We define the vectors 
\begin{align*}
w = G_\alpha(x, \perturb) - A(x) - \perturb  \qquad \text{ and } 
\qquad x_+ = s_f(x - \alpha(A(x)+ \perturb)). 
\end{align*}
Simple algebraic manipulations show the inclusion $w \in \hat \partial f(x_+)$. Moreover, Assumption~\ref{assumption:localbound} directly implies  
$$
\max\{\|w\|,\|G_{\alpha}(x, \perturb)\|\} \leq C(1 + \|\perturb\|)  \qquad \text{and} \qquad \|x_+ - x\| \leq C(1 + \|\perturb\|)\alpha.
$$
We will use these estimates often in the proof. Finally, we let $C$ be a constant independent of $x, \alpha$ and $\nu$, which changes from line to line.

\noindent\underline{Assumption~\ref{assumption:smoothcompatibility}:} 
First suppose $x_+\in B_{\epsilon}(\bar x)$. Using the triangle inequality, we write 
\begin{align*}
&\|P_{\tangentM{P_{\cM}(x)}}(G_\alpha(x, \perturb) - F(P_{\cM}(x)) - \nu)\|\\
&=\|P_{\tangentM{P_{\cM}(x)}}(w+A(x)- A(P_{\cM}(x))-\nabla_{\cM} f(P_{\cM}(x)))\|\\
&\leq \underbrace{\|P_{\tangentM{P_{\cM}(x)}}(w-\nabla_{\cM} f(P_{\cM}(x_+)))\|}_{R_1}+\underbrace{\|A(x)- A(P_{\cM}(x))\|}_{R_2}+\underbrace{\|\nabla_{\cM} f(P_{\cM}(x)))-\nabla_{\cM} f(P_{\cM}(x_+)))\|}_{R_3}.
\end{align*}
Using the triangle inequality and the estimate \eqref{eqn:basic_est_needed} we deduce
\begin{align*}
R_1&\leq \|P_{\tangentM{P_{\cM}(x_+)}}(w-\nabla_{\cM} f(P_{\cM}(x_+)))\|+\|P_{\tangentM{P_{\cM}(x_+)}}-P_{\tangentM{P_{\cM}(x)}}\|_{\rm op}\cdot \|w-\nabla_{\cM} f(P_{\cM}(x_+))\|\\
&\leq C(1+\|w\|)\cdot\dist(x_+,\cM)+C\|x-x_+\|\cdot(1+\|w\|)\\
&\leq C(1+\|\nu\|)\dist(x_+,\cM)+C(1+\|\nu\|)^2\alpha\\
&\leq C(1+\|\nu\|)(\dist(x,\cM)+C\|x-x_+\|)+C(1+\|\nu\|)^2\alpha\\
&\leq C(1+\|\nu\|)\dist(x,\cM)+C(1+\|\nu\|)^2\alpha.
\end{align*}
Moreover, clearly we have
$R_2\leq C\dist(x,\cM)$
and 
$R_3\leq C\|x-x_+\|\leq (1+\|\nu\|)\alpha.$ Condition~\ref{assumption:smoothcompatibility} follows immediately.

 Now suppose that $x_+ \notin B_{\epsilon}(\bar x)$, and therefore $\|x_+ - x\| \geq \epsilon/2$. Then, we may write
\begin{align*}
\|P_{\tangentM{P_{\cM}(x)}}(G_{\alpha}(x, \perturb)- \perturb - \nabla f_{\cM}(P_{\cM}(x)))\| &\leq \|G_{\alpha}(x, \perturb)\| +  \|\perturb\| +  \|\nabla f_{\cM}(P_{\cM}(x))\|\\
&\leq \frac{2}{\epsilon}(\|G_{\alpha}(x, \perturb)\| +  \|\perturb\| +  \|\nabla f_{\cM}(P_{\cM}(x))\|)\|x - x_+\| \\
&\leq C(1+\|\perturb\|)^2\alpha,
\end{align*}
as desired.

\noindent\underline{Assumption~\ref{assumption:aiming}:} We begin with the decomposition 
\begin{align*}
\langle G_{\alpha}(x,\nu)-\nu, x-y\rangle=&\underbrace{\langle A(x_+)+w,x_+-P_{\cM}(x_+)\rangle}_{R_1}\\
&+\underbrace{\langle A(x)-A(x_+),x-y\rangle}_{R_2}+\underbrace{\langle A(x_+)+w, (x-P_{\cM}(x))-(x_+-P_{\cM}(x_+))\rangle}_{R_3}.
\end{align*}
We now bound the two terms on the right in the case $x_+\in B_{\epsilon}(\bar x)$. Using \eqref{eqn:angle}, we estimate
\begin{align*}
R_1&\geq \tfrac{\mu}{2}\cdot  \dist(x_+, \cM) - (1+\|v\|)o(\dist(x_+, \cM))\\
&\geq \tfrac{\mu}{2}\cdot  (\dist(x, \cM)-\|x-x_+\|) - (1+\|v\|)(o(\dist(x, \cM)) +\|x-x_+\|)\\
&\geq \tfrac{\mu}{2}  \cdot\dist(x, \cM)-(1+\|\nu\|)^2(o(\dist(x, \cM))+C\alpha).
\end{align*}
Next, we compute
$$|R_2|\leq \|A(x)-A(x_+)\|\cdot\dist(x,\cM)\leq 2L\epsilon\cdot \dist(x,\cM)\leq \frac{\mu}{4} \dist(x,\cM).$$
Next using Lipschitz continuity of the map $I-P_{\cM}$ on $B_{\epsilon}(\bar x)$,  we deduce
$$|R_3|\leq C\|A(x_+)+w\|\cdot\|x-x_+\|\leq  C(1+\|\nu\|)^2\alpha.$$
The claimed proximal aiming condition follows immediately with $\mu$ replaced by $\mu/4$.

Let us look now at the case $x_+ \notin B_{\epsilon}(\bar x)$, and therefore $\dist(x,\cM)\leq \frac{\epsilon}{2}\leq  \|x-x^+\|$.
Then we compute \begin{align*}
\dotp{G_{\alpha}(x, \perturb) - \perturb, x - P_{\cM}(x)} &\geq -\dist(x,\cM)\cdot \|G_{\alpha}(x, \perturb) - \perturb\|\\
&= \dist(x, \cM) - \dist(x, \cM)(1 + \|G_{\alpha}(x, \perturb) - \perturb\|) \\
&\geq\dist(x, \cM) - C\|x - x_+\| (1 + C(1+\|\perturb\|)) \\
&\geq \dist(x, \cM) -C(1+\|\perturb\|)^2 \alpha,
\end{align*}
as desired. The proof is complete.

\section{Proof of Theorem~\ref{thm: asymptotic normality proposed}}\label{sec:proofsofstochastic}
We now outline common notation and conventions used in the proof. We let $\mathcal{U}$ be the neighborhood $\bar x$ where all the standing assumptions hold.
Shrinking $\mathcal{U}$, we may assume that the projection map $P_{\cM}$ is $C^2$ in $\mathcal{U}$, and in particular $P_{\cM}$ is Lipschitz with Lipschitz Jacobian. Throughout, the proof we shrink $\mathcal{U}$ several times, when needed. 

Now, denote stopping time~\eqref{def:stoppingtime} by $\tau := \tau_{k_0, \delta}$ and the noise bound by $Q := \sup_{x \in B_{\delta}(\bar x)} q(x)$. Observe that by Proposition~\ref{prop:shadow}, the shadow sequence $y_k $ satisfies $y_k \in B_{4\delta}(x_k) \cap \cM \subseteq B_{\epsilon}(\bar x) \cap \cM$ and recursion~\eqref{eqn:shadow:eq:iteration} holds. In addition, we let $C$ denote a constant depending on $k_0$ and $\delta$, which may change from line to line.

\subsection{Improved rates near strong local minimizers}\label{thm:localminimizersproofs}
As the first step, we obtain an improved rate of convergence under the growth condition \eqref{eqn:strong_growth}. To this end, we first need the following Lemma ensuring that $F_{\cM}$ has sufficient curvature in $B_{2\delta}(\bar x)$.
\begin{lem}[Curvature]\label{thm:localminimizers:lemma:strongconvexity}
The estimate
$$
\dotp{F_\cM (y), y- \bar x} \ge \frac{\mu}{2}\norm{y-\bar x}^2, 
$$
holds for all $x\in \cM$ sufficiently close to $\bar x$.
\end{lem}
\begin{proof} 
Let $\Phi$ be a smooth extension of $F_{\cM}$ to a neighborhood of $U\subset\R^d$ of $\bar x$. Consider an arbitrary  sequence $x_i\in \cM$ converging to $\bar x$. Passing to a subsequence, we may assume that the unit vectors $\frac{x_i-\bar x}{\|x_i-\bar x\|}$ converge to some unit vector $w\in T_{\cM}(\bar x)$. Let $H_i:=\int_{0}^1 \nabla \Phi(\bar x+\tau(x_i-\bar x))\, d\tau$ denote the average Jacobian between $\bar x$ and $x_i$. Note that $H_i$ clearly tends to $\nabla \Phi(\bar x)$ as $i$ tends to infinity.
The fundamental theorem of calculus yields
\begin{align*}
\frac{\langle \Phi(x_i)-\Phi(\bar x),x_i-\bar x\rangle}{\|x_i-\bar x\|^2}&=\left\langle H_i\left(\frac{x_i-\bar x}{\|x_i-\bar x\|}\right),\frac{x_i-\bar x}{\|x_i-\bar x\|}\right\rangle\mathrel{\mathop{\rightarrow}^{}_{i\to \infty}}\langle\nabla \Phi(\bar x) w,w\rangle\geq \mu.
\end{align*}
Since $x_i\in \cM$ was an arbitrary sequence converging to $\bar x$, the result follows.
\end{proof}

Next, we obtain a familiar one-step improvement guarantee.
\begin{lem}[One-step improvement]\label{lem: contractionofykinexp}
For all sufficiently small $\delta$, there exists a constant $C$ such that for any $k \ge k_0$, we have
 	\begin{align}\label{eqn: conditional contractiony general}
 	\EE [\norm{y_{k+1} -\bar x}^21_{\tau >k}] &\le \left(1-\frac{\alpha_k\mu}{2}\right)\EE[\norm{y_k - \bar x}^21_{\tau >k}] + C\alpha_k^2.
 \end{align}
\end{lem}
\begin{proof}
Expanding $\norm{y_{k+1}-\bar x}^2$, we obtain
	\begin{align}\label{lem:contractionofykinexp:eq1}
&\norm{y_{k+1} -\bar x}^21_{\tau >k} \notag\\
	&=\norm{y_k- \alpha_k F_\cM(y_k) - \alpha_k P_{\tangentM{y_k}}(\perturb_k) + \alpha_k E_k-\bar x}^21_{\tau >k} \notag \\
		&=\norm{y_k- \alpha_k F_\cM(y_k) + \alpha_k E_k-\bar x}^21_{\tau >k} + \alpha_k^2 \norm{P_{\tangentM{y_k}}(\perturb_k) }^21_{\tau >k} \notag\\
		&- 2\alpha_k \dotp{y_k- \alpha_k F_\cM(y_k) + \alpha_k E_k-\bar x,P_{\tangentM{y_k}}(\perturb_k) }1_{\tau >k} \notag\\
		&= \underbrace{\norm{y_k- \alpha_k F_\cM(y_k) -\bar x}^21_{\tau >k}}_{P_1} +\alpha_k^2\norm{E_k}^21_{\tau >k} + 2\alpha_k\underbrace{\dotp{y_k- \alpha_k F_\cM(y_k) -\bar x, E_k}1_{\tau > k}}_{P_2} \notag\\
		&+ \alpha_k^2 \norm{P_{\tangentM{y_k}}(\perturb_k) }^21_{\tau >k} -  2\alpha_k \underbrace{\dotp{y_k- \alpha_k F_\cM(y_k) + \alpha_k E_k-\bar x,P_{\tangentM{y_k}}(\perturb_k) }1_{\tau >k}}_{P_3}.
	\end{align}
Using Lemma~\ref{thm:localminimizers:lemma:strongconvexity}, we may bound $P_1$ as	\begin{align*}
		P_1 &=\left(\norm{y_k - \bar x}^2 -2\alpha_k\dotp{ F_\cM(y_k), y_k -\bar x} +\alpha_k^2\norm{\nabla f_\cM(y_k)}^2\right)1_{\tau >k}\\
		&\le \left((1-\alpha_k\mu)\norm{y_k - \bar x}^2+C\alpha_k^2\right)1_{\tau >k}
	\end{align*}
Next, using Proposition~\ref{prop:shadow} (\ref{prop:shadow:part:error:part:upperbound:2}) and Assumption~\ref{assumption:zero}, we see that $\EE_k[\norm{E_k}^21_{\tau >k}]$ and $\EE_k[\|P_{\tangentM{y_k}}(\perturb_k)\|^2 1_{\tau >k}]$ are bounded by a numerical constant. It remains to bound $P_2$ and $P_3$. Beginning with the former, using Young's inequality, we compute
\begin{align*}
P_2&\leq \frac{\mu \|y_k- \alpha_k F_\cM(y_k) -\bar x\|^2 1_{\tau >k}}{8}+\frac{2\|E_k\|^2 1_{\tau >k}}{\mu}=\frac{\mu P_1}{8}+\frac{2\|E_k\|^2 1_{\tau >k}}{\mu}.
\end{align*}
Next, again using Young's inequality, we bound the conditional expectation of $P_3$ as follows:
\begin{align*}
 \EE_k[P_3] 
&= \alpha_k\EE_k[\dotp{E_k, P_{\tangentM{y_k}}(\perturb_k)}1_{\tau > k}] \leq \frac{\alpha_k\EE_k\|E_k\|^21_{\tau > k}}{2} + C\alpha_k.
\end{align*}
Thus, returning to Lemma~\ref{lem:contractionofykinexp:eq1} and using Proposition~\ref{prop:shadow}(\ref{prop:shadow:part:error:part:upperbound:3}), we arrive at the estimate:
\begin{align*}
\EE [\norm{y_{k+1} -\bar x}^21_{\tau >k}] &\le (1-\alpha_k\mu/2)\EE[\norm{y_k - \bar x}^21_{\tau >k}] +C\alpha_k^2.\end{align*}
This completes the proof.
\end{proof}

Next, we can iterate the recursion to ensure a fast rate of convergence of $\|y_k-\bar x\|$. As a byproduct, we also obtain estimates on the size of the errors $E_k$. To simplify notation, we write $\tau_{k_0} := \tau_{k_0, \delta}$, since we will consider several values of $k_0$. Under these conventions, we have the following Proposition, which will be useful in ensuring summability of certain sequences.
\begin{lem}\label{lem: almostsureyk}
There exists $C> 0$ such that
\begin{enumerate}
	\item \label{lem: almostsureyk:1}    $\EE[\norm{y_k -\bar x}^2 1_{\tau_{k_0}>k}] \le C/k^\gamma$ for all $k \geq 1$.
	\item \label{lem: almostsureyk:2} 	$\sum_{k=1}^{\infty} \frac{1}{\sqrt{k}}\norm{y_k -\bar x}^2 <\infty$ almost surely.
	\item\label{lem: almostsureyk:3}    $\frac{1}{\sqrt{n}}\sum_{k=1}^{n}\norm{y_k -\bar x}^2 \rightarrow 0$ almost surely.
\item\label{lem: almostsureyk:4}  $\sum_{k=1}^\infty \frac{1}{\sqrt{k}}\|E_k\|  < +\infty$ almost surely.
\item\label{lem: almostsureyk:5}  $\frac{1}{\sqrt{n}}\sum_{k=1}^n \|E_k\|  < +\infty$ almost surely.
	\end{enumerate}
\end{lem}
\begin{proof}
Part~\ref{lem: almostsureyk:1} follows immediately from Lemmas~\ref{lem: contractionofykinexp} and \ref{lem:sequencelemmasquared} by setting $s_k = \EE[\norm{y_k -\bar x}^21_{\tau_{k_0}>k}]$.
We now prove Part~\ref{lem: almostsureyk:2}. By Part~\ref{lem: almostsureyk:1}, we have 
		$$
		\expect{\sum_{k=1}^{\infty}\frac{1}{\sqrt{k}}\norm{y_k -\bar x}^2 1_{ \tau_{k_0} >k}} \le \sum_{k=1}^{\infty}\frac{C}{k^{\gamma+\frac{1}{2}}} <\infty.
		$$
		Therefore, 
		$\sum_{k=1}^{\infty}\frac{1}{\sqrt{k}}\norm{y_k -\bar x}^2 1_{\tau_{k_0}>k}$ is finite almost surely. Taking into account that $x_k \rightarrow \bar x$ almost surely, the sum $\sum_{k=1}^{\infty}\frac{1}{\sqrt{k}}\norm{y_k -\bar x}^2$ must be finite almost surely. Part~\ref{lem: almostsureyk:3} now follows immediately follows from Kronecker lemma~\ref{lem:kronecker}

Next, we prove Part~\ref{lem: almostsureyk:4}. By Proposition~\ref{prop:shadow}(\ref{prop:shadow_missing} ), we know that the error sequence $E_k$ almost surely satisfies
$$
\sum_{k=1}^\infty \frac{1}{\sqrt{k}}\|E_k\|1_{\tau_{k_0} > k}  < +\infty.
$$
Since $x_k \rightarrow \bar x$ almost surely, we deduce that almost surely we have
$
\sum_{k=1}^\infty \frac{1}{\sqrt{k}}\|E_k\|  < +\infty,
$
as desired. Part~\ref{lem: almostsureyk:5} follows from Kronecker lemma~\ref{lem:kronecker}.
\end{proof}

\subsection{Completing the proof of Theorem~\ref{thm: asymptotic normality proposed}}
We now turn to the proof of Theorem~\ref{thm: asymptotic normality proposed}. To this end, we introduce an additional sequence 
\begin{align}
z_k := \proj_{\bar x + \tangentM{\bar x}}(y_k).
\end{align}
Evidently, for all $\delta$ sufficiently small, $z_k$ closely approximates $y_k$. Indeed, due to the smoothness of $\cM$, there exists $C > 0$ such that 
\begin{align}\label{eq:zkykapprox}
\|y_k - z_k \|1_{\tau_{k_0} > k} \leq C\|y_k - \bar x\|^21_{\tau_{k_0} > k}.
\end{align}
The next result states that it suffices to study the distribution of $\frac{1}{\sqrt{n}}\sum_{k=1}^{n}(z_k - \bar x)$.

\begin{lem}[Reduction to an auxiliary seqeunce]\label{lem: convergetogether}
	If $\frac{1}{\sqrt{n}}\sum_{k=1}^{n}(z_k - \bar x)$ converges in distribution to some $\mathcal{D}$, then 
	$\frac{1}{\sqrt{n}}\sum_{k=1}^{n}(x_k - \bar x)$ converges in distribution to $\mathcal{D}$.
\end{lem}
\begin{proof}
	Note that 
	$$\frac{1}{\sqrt{n}}\sum_{k=1}^{n}(x_k-\bar x) = \frac{1}{\sqrt{n}}\sum_{k=1}^{n}(z_k - \bar x) + \frac{1}{\sqrt{n}}\sum_{k=1}^{n}(x_k - y_k) + \frac{1}{\sqrt{n}}\sum_{k=1}^{n}(y_k - z_k).
	$$
	By Lemma~\ref{lem:basic_facts_conv} (\ref{basic_facts_conv4}), the result will follow once we show that $$ \frac{1}{\sqrt{n}}\sum_{k=1}^{n}(x_k - y_k) \rightarrow 0 \qquad \text{ and } \qquad  \frac{1}{\sqrt{n}}\sum_{k=1}^{n}(y_k - z_k)\rightarrow 0,$$ almost surely. To that end, we recall that Proposition~\ref{prop:gettingclosertothemanifold}(\ref{eq:prop:gettingclosertothemanifoldbound1:sum}) guarantees that almost surely we have
	$$
	\sum_{k=1}^\infty \frac{1}{\sqrt{k}}\norm{x_k-y_k}1_{ \tau_{k_0} > k}   < +\infty   
	$$
	Since $x_k \rightarrow \bar x$ almost surely, for almost every sample path, we can find a $k_0$ such that $\tau_{k_0}=\infty$. Therefore, almost surely we have
	$ \sum_{k=1}^{\infty} \frac{\norm{x_k - y_k}}{\sqrt{k}} <\infty.$
	Applying Kronecker lemma~\ref{lem:kronecker}, almost surely we have
	 $$ \frac{1}{\sqrt{n}}\sum_{k=1}^{n}\norm{x_k - y_k} \rightarrow 0,$$
	 which implies  $ \frac{1}{\sqrt{n}}\sum_{k=1}^{n}(x_k - y_k) \rightarrow 0$.
	 On the other hand, we have by Lemma~\ref{lem: almostsureyk} and inequality~\eqref{eq:zkykapprox}, that 
\begin{align*}
\sum_{k=1}^{\infty} \frac{1}{\sqrt{k}}\norm{y_k -z_k}1_{\tau_{k_0} > k} \leq\sum_{k=1}^{\infty} \frac{C}{\sqrt{k}}\norm{y_k -\bar x}^21_{\tau_{k_0} > k}  < +\infty
\end{align*}
Again since for almost every sample path we may find $k_0$ such that $\tau_{k_0} = \infty$, we have that $\sum_{k=1}^{\infty} \frac{1}{\sqrt{k}}\norm{y_k -z_k} < + \infty$, as desired.
\end{proof}

In light of Lemma~\ref{lem: convergetogether}, it suffices now to compute the limiting distribution of $\frac{1}{\sqrt{n}}\sum_{k=1}^{n}(z_k - \bar x)$. This is the content of following lemma. Notice that in the lemmas, we state the asymptotic covariance matrix in a different equivalent form to that appearing in Theorem~\ref{thm: asymptotic normality proposed}, and which is more convenient for computation.
\begin{lem}
	The sequence $\frac{1}{\sqrt{n}}\sum_{k=1}^{n}(z_k - \bar x)$ converges in distribution to $$N\left(0,  U(U^\top \nabla F_{\cM}(\bar x) U)^{-1}\cdot \Sigma\cdot (U^\top  \nabla F_{\cM}(\bar x) U)^{-1}U^\top\right),$$ where  $U$ is a matrix whose column vectors form an orthonormal basis of $T_\cM(\bar x)$.
\end{lem}
\begin{proof}
Recall that $UU^{\top}$ is the orthogonal projection onto $T_{\cM}(\bar x)$.
Therefore, we may write $z_k = \bar x + UU^\top (y_k-\bar x)$. Moreover, subtracting $\bar x$ from both sides of~\eqref{eqn:shadow:eq:iteration} and  multiplying by $U^\top$, we have
\begin{align*}
U^\top(y_{k+1} - \bar x) &= U^\top(y_k - \bar x)-  \alpha_k U^\top F_\cM(y_k) - \alpha_k U^\top P_{T_\cM(y_k)}(\perturb_\discrete) +\alpha_k U^\top E_k\\
&=  U^\top(y_k - \bar x)- {\alpha_k U^\top \nabla_{\cM} F_\cM(\bar x)UU^{\top}(y_k -\bar x)}\\
&\qquad- \alpha_k (U^\top F_\cM(y_k) - {U^\top \nabla_{\cM} F_\cM(\bar x)UU^\top(y_k -\bar x))}\\
&\qquad - \alpha_k U^\top P_{T_\cM(\bar x)}(\perturb_\discrete)- \alpha_k( U^\top P_{T_\cM(y_k)}(\perturb_\discrete) - U^\top P_{T_\cM(\bar x)}(\perturb_\discrete))+\alpha_k U^\top E_k.
\end{align*}
Define $\Delta_k = U^\top(y_k -\bar x)$, $H=U^\top \nabla_{\cM} F_\cM(\bar x) U$, $\zeta_k = U^\top P_{T_\cM(y_k)}(\perturb_\discrete) - U^\top P_{T_\cM(\bar x)}(\perturb_\discrete)$, and 
$$R(y)= U^\top  F_\cM(y) - U^\top \nabla_{\cM} F_\cM(\bar x)UU^{\top}(y -\bar x).$$
By our assumption, for every vector $z$ the matrix $H$ satisfies
$$\langle Hz,z\rangle=\langle \nabla_{\cM} F_\cM(\bar x) U z,Uz\rangle\geq \sigma\|Uz\|^2=\sigma \|z\|^2.$$
Consequently $H$ is a strongly monotone matrix. Note moreover the equality $U^\top P_{T_\cM(\bar x)}(\perturb_\discrete)= U^\top UU^\top \nu_k = U^\top \nu_k$. Thus we can rewrite the update of $\Delta_k$ as 
\begin{align*}
	\Delta_{k+1} = \Delta_k - \alpha_k H \Delta_k -\alpha_k U^\top \perturb_k- \alpha_k \left(R(y_k)+ \zeta_k - U^\top E_k \right). 
\end{align*}

In the remainder of the proof, we prove that $\frac{1}{\sqrt{n}} \sum_{k=1}^{n}\Delta_n$ converges in distribution to 
$ N\left(0,  (U^\top \nabla_{\cM} F_\cM(\bar x) U)^{-1} \Sigma(U^\top \nabla_{\cM} F_\cM(\bar x) U)^{-1}\right).$
This implies that result since by definition, $\frac{1}{\sqrt{n}}\sum_{k=1}^{N}(z_k - \bar x) = \frac{1}{\sqrt{n}}\sum_{k=1}^{n}U\Delta_k$. We note that our proof closely mirrors~\cite[Theorem 2]{asi2019stochastic}.

To prove this claim, define the matrices 
$$
B_k^n = \alpha_k \sum_{i=k}^{n}\prod_{j= k+1}^{i}(I-\alpha_j H)\qquad \text{  and  } \qquad A_k^n = B_k^n - H^{-1}.
$$
Polyak and Juditsky~\cite[Lemma 2]{polyak1992acceleration} show that $\bar \Delta_n = \frac{1}{n}\sum_{k=1}^{n} \Delta_k$ satisfies the equality 
\begin{align*}
	\sqrt{n} \bar \Delta_n &= \frac{1}{\sqrt{n}} \sum_{k=1}^{n}H^{-1}U^\top \nu_k\\
	&\qquad + \frac{1}{\sqrt{n}}\sum_{k=1}^{n}A_k^n U^\top \nu_k + \frac{1}{\sqrt{n}}\sum_{k=1}^{n}B_k^n[R(y_k)+ \zeta_k - U^\top E_k] + O\left(\frac{1}{\sqrt{n}}\right),	\end{align*}
where $\sup_{k,n}\max\{\norm{B_k^n}, \norm{A_k^n}\} <+\infty$ and $\lim_{n \rightarrow \infty} \frac{1}{n}\sum_{k=1}^{n} \norm{A_k^n} =0$. Equivalently, after expanding $\nu_k=\nu_k^{(1)}+\nu_k^{(2)}(x_k)$, we obtain
\begin{align*}
\sqrt{n} \bar \Delta_n &= \frac{1}{\sqrt{n}} \sum_{k=1}^{n}H^{-1}U^\top \nu_k^{(1)}\\
	&\qquad + \frac{1}{\sqrt{n}}\sum_{k=1}^{n}A_k^n U^\top \nu_k^{(1)} + \frac{1}{\sqrt{n}}\sum_{k=1}^{n}B_k^n[R(y_k)+ \zeta_k - U^\top E_k + \nu_k^{(2)}(x_k)] + O\left(\frac{1}{\sqrt{n}}\right).\end{align*}
Assumption~\ref{assumption:martinagle} ensures that the sum $ \frac{1}{\sqrt{n}} \sum_{k=1}^{n}H^{-1}U^\top \nu_k^{(1)}$ converges in distribution to 
$$ N\left(0,(U^\top \nabla_{\cM} F_\cM(\bar x) U)^{-1} \Sigma(U^\top \nabla_{\cM} F_\cM(\bar x) U)^{-1}\right).$$
Thus the claim holds if we can show that the other sums in our expression for $\sqrt{n}\bar \Delta_n$ converge to $0$ almost surely. We do so in the following sequence of claims.

\begin{claim}	
	We have that 
	$$\frac{1}{\sqrt{n}}\sum_{k=1}^{n}A_k^n U^\top \nu_k^{(1)} \xrightarrow{\text{a.s.}} 0.$$ 
\end{claim}
\begin{proof}
Using that $\sup_{k,n}\|A_k^n\|<\infty$ and that $\EE\|\nu_k\|^21_{\tau_{k_0} > k}$ is bounded, we deduce
	\begin{align*}
		\EE\left[\norm{\frac{1}{\sqrt{n}} \sum_{k=1}^{n} A_k^n U^\top \nu_k^{(1)}}^21_{\tau_{k_0} > k}\right] &= \frac{1}{n}\sum_{k=1}^{n}\EE\left[\norm{A_k^n U^\top \perturb_k^{(1)}1_{\tau_{k_0} > k}}^2\right]\\
		&\le \frac{C}{n} \sum_{k=1}^{n} \norm{A_k^n}\\	
		&\rightarrow 0.
	\end{align*}
Thus $\frac{1}{\sqrt{n}}\sum_{k=1}^{n}A_k^n U^\top \nu_k1_{\tau_{k_0} > k} $ is a $L^2$-bounded martingale. By~\cite[Theorem 4.4.6]{durrett2019probability}, we know that  $\frac{1}{\sqrt{n}}\sum_{k=1}^{n}A_k^n U^\top \nu_k1_{\tau_{k_0} > k}\xrightarrow{L^2} 0$. On the other hand, by~\cite[Theorem 4.2.11]{durrett2019probability}, $\frac{1}{\sqrt{n}}\sum_{k=1}^{n}A_k^n U^\top \nu_k 1_{\tau_{k_0} > k}$ converges almost surely. Therefore, since for almost every sample path there exists $k_0$ such that $\tau_{k_0} = \infty$, we have $\frac{1}{\sqrt{n}}\sum_{k=1}^{n}A_k^n U^\top \nu_k \xrightarrow{\text{a.s.}} 0$, as desired.
\end{proof}

\begin{claim}\label{item:asymptoticnormalityhardcase}
 We have that $$\frac{1}{\sqrt{n}}\sum_{k=1}^{n}B_k^n U^\top R(y_k) \xrightarrow{\text{a.s.}} 0.$$ 
\end{claim}
\begin{proof}
Let $\Phi$ be a smooth extension of $F_{\cM}$ to a neighborhood $U\subset \R^d$ of $\bar x$. We then deduce
\begin{align*}
R(y)&= U^\top (\Phi(y)-\Phi(\bar x)-\nabla\Phi(\bar x)UU^{\top}(y-\bar x))\\
&=U^{\top}\nabla \Phi(\bar x)(I-UU^{\top})(y-\bar x)+O(\|y-\bar x\|^2)\\
&=U^{\top}\nabla \Phi(\bar x)P_{N_{\cM}(\bar x)}(y-\bar x)+O(\|y-\bar x\|^2)
\end{align*}
Since $\cM$ is $C^2$-smooth, it follows immediately that $\|P_{N_{\cM}(\bar x)}(y-\bar x)\|\leq O(\|y-\bar x\|^2)$ as $y\in \cM$ tends to $\bar x$. Thus, we have $\norm{R(y)}= O(\|y-\bar x\|^2)$. In addition, by our assumption that $x_k \xrightarrow{\text{a.s.}} \bar x$, we have $y_k \xrightarrow{\text{a.s.}} \bar x$. Consequently, there exists a constant $C$ depending on sample path such that $\norm{R(y_k)} \le C \norm{y_k - \bar x}^2$ almost surely. Uniform boundedness of 
$B_k^n$ and  Lemma~\ref{lem: almostsureyk} therefore implies $\frac{1}{\sqrt{n}}\sum_{k=1}^{n}B_k^n U^\top R(y_k) \xrightarrow{\text{a.s.}} 0$.
\end{proof}

\begin{claim}\label{item:asymptoticnormalityhardcase2}
We have that 
$$\frac{1}{\sqrt{n}}\sum_{k=1}^{n}B_k^n \zeta_k \xrightarrow{\text{a.s.}} 0.$$
\end{claim}
\begin{proof}
For $k \ge 1$, define truncated variables $\zeta_k^{(k_0)} = \zeta_k1_{\tau_{k_0}>k}$. Note that suffices to show that 
$$
\frac{1}{\sqrt{n}}\sum_{k=1}^{n}B_k^n \zetak \xrightarrow{\text{a.s.}} 0, 
$$
since on every sample path there exists a $k_0$ such that $\tau_{k_0} = \infty$. Thus, we will work with these truncated variables throughout. 

Turning to the proof, we first show that $\frac{1}{\sqrt{n}}\sum_{k=1}^{n}\zetak \xrightarrow{P} 0$ and $\frac{1}{\sqrt{n}}\sum_{k=1}^{n}A_k^n \zetak \xrightarrow{P} 0$.  Recall that $\zeta_k = U^\top P_{T_\cM(y_k)}(\perturb_\discrete) - U^\top P_{T_\cM(\bar x)}(\perturb_\discrete)$, so we have 
$$
\expect{\zetak\mid \cF_k}= \expect{\zeta_k \mid \cF_k}1_{\tau_{k_0} >k} = 0.
$$
Since $x \mapsto P_{T_\cM(x)}$ is locally Lipschitz on a neighborhood of $\bar x$ in $\cM$, we have the following bound for some $C > 0$ and all sufficiently small $\delta$:
$$
\norm{\zetak}^2 \le C\norm{y_k - \bar x}^21_{\tau_{k_0}>k},
$$
In particular, it holds that  
$$
\expect{\norm{\zetak}^2 \mid \cF_k} \le C^2\norm{y_k - \bar x}^21_{\tau_{k_0}>k}.
$$
Combining with Lemma~\ref{lem: almostsureyk}(\ref{lem: almostsureyk:1}), we know that $\zetak$ is a martingale difference sequence and almost surely,
$$
\sum_{k= 1}^\infty \frac{1}{k} \expect{\norm{\zetak}^2 \mid \cF_k} \le C^2\sum_{k=1}^{\infty}\frac{1}{k}\norm{y_k - \bar x}^2<\infty.
$$
Therefore, by Lemma~\ref{lem: L2martingaleconvergence}, we have 
$$\frac{1}{\sqrt{n}}\sum_{k=1}^{n}\zetak \xrightarrow{\text{a.s.}} 0.$$
In particular, it holds that $\frac{1}{\sqrt{n}}\sum_{k=1}^{n}\zetak \xrightarrow{P} 0.$

Next we show that for any $k_0<\infty$, we have $n^{-1/2}\sum_{k=1}^n A_k^n \zetak \xrightarrow{P} 0$. To see this, note that by Lemma~\ref{lem: almostsureyk}, there exists $C' > 0$ such that 
\begin{align}\label{eq:truncatedboundzeta}
\expect{\norm{\zetak}^2} \le C\expect{\norm{y_k - \bar x}^2 1_{\tau_{k_0}>k}} \le C'\alpha_k.
\end{align}
Hence, the following limit holds 
\begin{align*}
\expect{\norm{\frac{1}{\sqrt{n}}\sum_{k=1}^{n}A_k^n \zetak}^2} &= \frac{1}{n}\sum_{k=1}^{n}\expect{\norm{A_k^n \zetak}^2} 
\leq \frac{C'\alpha_k}{n}\sum_{k=1}^{n} \norm{A_k^n }^2 
\le \frac{C'\alpha_k\sup_{k,n}\|A_k^n\|}{n}\sum_{k=1}^{\infty}\norm{A_k^n} \rightarrow 0,
\end{align*}
where the first equality follows from the martingale difference property and the second inequality follows from the boundedness of moments of $\zetak$. Consequently, we have shown that $$\frac{1}{\sqrt{n}}\sum_{k=1}^{n}A_k^n \zetak \xrightarrow{L^2} 0,$$ which implies that $\frac{1}{\sqrt{n}}\sum_{k=1}^{n}A_k^n \zetak \xrightarrow{P} 0$.

We have therefore proved that $\frac{1}{\sqrt{n}}\sum_{k=1}^{n}B_k^n \zetak \xrightarrow{P} 0.$ We now show that  $\frac{1}{\sqrt{n}}\sum_{k=1}^{n}B_k^n \zetak$ converges almost surely. Since the almost sure limits and limits in probability agree when both exist, this will complete the proof.

To this end, define the sequence 
$$
Z_{n,k_0} = \sum_{k=1}^{n}B_k^n\zetak.
$$
The result follows if we can prove that for any finite $k_0$, the sequence $n^{-1/2}Z_{n,k_0}$ almost surely converges. To that end, note that $B_k^{n+1} -B_k^{n}= \alpha_k \prod_{i=k+1}^{n}(I-\alpha_i H)$. Thus, defining
$$
W_k^n = \prod_{i=k}^n(I-\alpha_i H), \qquad V_{n,k_0}= \sum_{k=1}^{n}\alpha_kW_{k+1}^{n+1}\zetak,
$$
we deduce that $V_{n,k_0}$ is $\cF_{n+1}$ measurable and $Z_{n, k_0}$ admits the decomposition:
$$
Z_{n, k_0} = Z_{n-1, k_0} + V_{n-1,k_0} + \alpha_n \zeta_n^{(k_0)}= \sum_{k=1}^{n-1}V_{k,k_0} + \sum_{k=1}^{n}\alpha_k\zetak.
$$
Note that the sum $\sum_{k=1}^{n}\alpha_k\zetak$ is a square-integrable martingale with summable squared increments, so it converges almost surely~\cite[Theorem 4.2.11]{durrett2019probability}. As a result, we have the following limit $n^{-1/2}\sum_{k=1}^{n}\alpha_k\zetak \xrightarrow{\text{a.s.}}0$.  It thus suffices to show that $n^{-1/2}\sum_{k=1}^{n-1} V_{k,k_0}$ converges almost surely. 

To that end, let $\lambda$ denote the smallest eigenvalue of $H$. Then we have
\begin{equation}\label{eqn:cray_estimation}
\expect{\norm{V_{n,k_0}}^2}=\sum_{k=1}^{n}\alpha_k^2 \norm{W_{k+1}^{n+1}}^2 \expect{\norm{\zetak}^2} \le C'\sum_{k=1}^{n} \alpha_k^3 \norm{W_{k+1}^{n+1}}^2,
\end{equation}
where the inequality follows from the bound $\expect{\norm{\zetak}^2} \le C'\alpha_k$ (see Equation~\eqref{eq:truncatedboundzeta}). The result \cite[Lemma 1 (part 3)]{} shows that there exist constants $\beta>0$ and $K<\infty$ such that for all $k$ and $t\geq k$, the estimate holds:
$$\norm{W_{k+1}^{n+1}}^2\leq K\exp\left(-\beta\sum_{i=k+1}^n \alpha_i\right).$$
Plugging this estimate into \eqref{eqn:cray_estimation}, exactly the same proof as that of~\cite[Lemma A.7]{asi2019stochastic} with $\rho=3$ shows that there exists some constant $C$ such that 
$$
\expect{\norm{V_{n,k_0}}^2} \le \frac{C\log n}{n^{2\gamma}}. 
$$
Hence, for any $\epsilon>0$, we can find some $C$ such that 
$$
\expect{\norm{V_{n,k_0}}^2} \le \frac{C}{n^{2\gamma-\epsilon}}. 
$$
Now define $T_{n,k_0}= \frac{1}{\sqrt{n}}\sum_{k=1}^{n}V_{k,k_0}.$ We claim that $T_{n,k_0}$ almost surely has finite length. Indeed, for any $\epsilon>0$ there exists $C, C' > 0$ such that 
\begin{align*}
	\expect{\norm{T_{n,k_0} - T_{n+1,k_0}}}&\le \abs{\frac{1}{\sqrt{n+1}} -\frac{1}{\sqrt{n}}} \sum_{k=1}^{n}\expect{\norm{V_{k,k_0}}} + \frac{1}{\sqrt{n+1}}\expect{\norm{V_{n+1,k_0}}}\\
	&\le \frac{C}{n^{\frac{3}{2}}}\sum_{k=1}^{n}\frac{1}{k^{\gamma-\epsilon}} + \frac{1}{\sqrt{n}} \frac{1}{n^{\gamma -\epsilon}}\\
	&\le \frac{C'}{n^{\gamma+ 1/2 - \epsilon}}.
\end{align*} 
Since $\gamma \in (\frac{1}{2}, 1)$, we therefore have $\sum_{n}	\expect{\norm{T_{n,k_0} - T_{n+1,k_0}}} <\infty$. Consequently, the sum is finite almost surely: $\sum_{n}	{\norm{T_{n,k_0} - T_{n+1,k_0}}}< + \infty$. This implies that $T_{n,k_0} = n^{-1/2}\sum_{k=1}^{n}V_{k,k_0}$ converges almost surely. Recalling the definition of $V_{k, k_0}$, we find that  $
n^{-1/2}Z_{n,k_0}
$ almost surely converges, which completes the proof. 
\end{proof}

\begin{claim}
We have that 
$$
\frac{1}{\sqrt{n}}\sum_{k=1}^{n}B_k^n \nu_k^{(2)}(x_k) \xrightarrow{\text{a.s.}} 0.$$
\end{claim}
\begin{proof}
This may be proved by argument that mirrors Claim~\ref{item:asymptoticnormalityhardcase}. Indeed, observe that the sequence $\xi_k = \nu_k^{(2)}(x_k)1_{\tau > k_0}$ is a martingale difference sequence, the bounds hold for some $C > 0$
$$
\EE_k[\|\xi_k\|^2] \leq C\|x_k - \bar x\|^21_{\tau > k_0} \qquad \text{ and } \qquad \EE[\|\xi_k\|^2] \leq C\alpha_k, 
$$
and $\sum_{k=1}^\infty \frac{1}{k}\EE_k[\|\xi_k\|^2] \leq \sum_{k=1}^\infty \frac{1}{k} \|x_k - \bar x\|^21_{\tau > k_0} < + \infty$. Only these facts for $\zeta_k^{(k_0)}$ were used to prove Claim~\ref{item:asymptoticnormalityhardcase}.
\end{proof}
\begin{claim}
We have that 
$$\frac{1}{\sqrt{n}}\sum_{k=1}^{n}B_k^n U^\top E_k \xrightarrow{\text{a.s.}} 0.$$ 
\end{claim}
\begin{proof}
This follows immediately from Lemma~\ref{lem: almostsureyk}(\ref{lem: almostsureyk:5} and the fact that $\sup_{k,n}\norm{B_k^n} <\infty$.

\end{proof}
\noindent Taking these claims into account, the proof is complete.
\end{proof}

\section{Auxiliary facts about sequences of random variables.}
\begin{definition}{\rm
Let $\{X_k\}_{k\geq 1}$ and $X$ be random vectors in $\R^d$ defined on a probability space $(\Omega,\mathcal{F},P)$.  
\begin{enumerate}
\item $X_k$ {\em converges almost surely to} $X$, denoted $X_k\xrightarrow[]{a.s.} X$ if for almost every $\omega\in \Omega$, the vector $X_k(\omega)$ converges to $X(\omega)$. 
\item $X_k$ {\em converges in probability to} $X$, denoted $X_k\xrightarrow[]{p} X$, if for every $\epsilon>0$, we have $\lim_{k\to \infty} P(\|X_k-X\|\leq \epsilon)\to 1$. 
\item $X_k$ {\em converges in distribution to} $X$, denoted $X_k\xrightarrow[]{D} X$ if for every bounded continuous function $G\colon\R^d\to\R$, one has $\lim_{k\to \infty} \EE G(x_k)=\EE G(X)$.
\item $X_k$ is {\em bounded in probability}, denoted $X_k=O_p(1)$, if for every $\epsilon>0$, there exist $M_{\epsilon}$ such that $P(\|X_n\|>M_{\epsilon})<\epsilon$ for all sufficiently large indices $k$.
\end{enumerate}}
\end{definition}

\begin{lem}\label{lem:basic_facts_conv} Let $\{X_k\}_{k\geq 1}$, $\{Y_k\}_{k\geq 1}$, $X$, and $Y$ be random vectors in some Euclidean space and let $a, b$ be deterministic. 
The following statements are true.
\begin{enumerate}
\item\label{basic_facts_conv1} The implications hold:
$$X_k\xrightarrow[]{a.s.}  X \qquad\Longrightarrow\qquad X_k\xrightarrow[]{p} X\qquad\Longrightarrow \qquad X_k\xrightarrow[]{D} X\qquad\Longrightarrow\qquad X_k=O_p(1).$$
\item\label{basic_facts_conv2} If $X_k\xrightarrow[]{p}  X$ and $Y_k\xrightarrow[]{p}  Y$, then $aX_k+bY_k\xrightarrow[]{p}  aX+b Y$ and $X_kY_k\xrightarrow[]{p}  XY$. The analogous statement holds for almost sure convergence.
\item\label{basic_facts_conv3} If $X_n=O_p(1)$ and $Y_k\xrightarrow[]{p}  0$, then $X_kY_k\xrightarrow[]{p}  0$.
\item\label{basic_facts_conv4} (Slutsky I) If $X_k\xrightarrow[]{D}  X$ and $Y_k\xrightarrow[]{p} a$, then $X_k+Y_k\xrightarrow[]{D}  X+a$ and $X_kY_K\xrightarrow[]{D} aX$.
\item\label{basic_facts_conv4b} (Slutsky II) If $X_k\xrightarrow[]{p}  X$ and $Y_k\xrightarrow[]{p} Y$, then $X_k+Y_k\xrightarrow[]{p}  X+Y$ and $X_kY_K\xrightarrow[]{p} XY$.
\item \label{elim} If $X_kY_k\xrightarrow[]{D}  X$ and $Y_k\xrightarrow[]{p} c$, then $X_k\xrightarrow[]{D} X/c$, as long as $c\neq 0$.
\item\label{basic_facts_conv5} (Delta Method) Suppose that $\sqrt{k}(X_k-\mu)\xrightarrow[]{d}  \mathsf{N}\big(0,\Sigma)$ for some $\mu\in \R^d$ and some matrix $\Sigma\in \R^{d\times d}$, then we have $\sqrt{k}(g(X_k)-g(\mu))\xrightarrow[]{d}  \mathsf{N}\big(0,\nabla g(\mu)\Sigma\nabla g(\mu)^{\top})$ for any map $g\colon\R^d\to\R^m$ that is differentiable at $\mu$.
\end{enumerate}
\end{lem}

\begin{lem}[Robbins-Siegmund\cite{robbins1971convergence}]\label{lem: Robbins-Siegmund}
Let $A_k, B_k,C_k,D_k \ge 0$ be non-negative random variables adapted to the filtration  $\{\cF_{k}\}$ and satisfying 
$$
\EE[A_{k+1}\mid \mathcal{F}_k] \le (1+B_k)A_k + C_k - D_k.
$$
Then on the event $\{\sum_{k}B_k <\infty, \sum_{k}C_k < \infty\}$, there is a random variable $A_\infty<\infty$ such that $A_k \xrightarrow{\text{a.s.}} A_\infty$ and $\sum_k D_k <\infty$ almost surely. 
\end{lem}

\begin{lem}[Conditional Borel-Cantelli~{\cite{condborellcantelli}}]\label{lem:condborelcantelli}
Let $\{X_n \colon n \ge 1\}$ be a sequence of nonnegative random variables defined on the probability space $(\Omega, \cF, \mathbb{P})$ and  $\{\cF_n \colon n \ge 0\}$ be a sequence of sub-$\sigma$-algebras of $\cF$. Let $M_n = \expect{X_n \mid \cF_{n-1}}$ for $n\ge 1$. If $\{\cF_n \colon n \ge 0\}$ is nondecreasing, i.e., it is a filtration, then $\sum_{n=1}^{\infty} X_n < \infty$ almost surely on $\{\sum_{n=1}^{\infty}M_n <\infty\}$.
\end{lem}

\begin{lem}[{\cite[Exercise 5.3.35]{dembo2016lecture}}]\label{lem: L2martingaleconvergence}
	Let  $M_k$ be an $L^2$ martingale adapted to a filtration $\{\cF_k\}$ and let $b_k \uparrow \infty$ be a positive deterministic sequence. Then if 
	$$\sum_{k\ge 1}b_k^{-2}\expect{(M_k - M_{k-1})^2 \mid \cF_{k-1}} < +\infty,$$
	we have $b_n^{-1}M_n \xrightarrow{\text{a.s.}} 0$.
\end{lem}

\begin{lem}[Kronecker Lemma]\label{lem:kronecker}
Suppose $\{x_k\}_k$ is an infinite sequence of real number such that the sum
$
\sum_{k=1}^{\infty} x_k 
$ 
exists and is finite. Then for any divergent positive nondecreasing sequence $\{b_k\}$, we have 
$$
\lim_{K \rightarrow \infty}\frac{1}{b_K}\sum_{k=1}^K b_k x_k = 0.
$$
\end{lem}

The proofs of the following three lemmas may be found in \cite[Section A.7.2]{davis2021subgradient}. 
\begin{lem}\label{lem:fastsummability}
Fix $k_0 \in \NN, c > 0$, and $\gamma \in (1/2, 1]$. Suppose that $\{X_k\}, \{Y_k\},$ and $\{Z_k\}$ are nonnegative random variables adapted to a filtration $\{\cF_k\}$. Suppose the relationship holds:
$$
\EE[X_{k+1}\mid\mathcal{F}_k] \leq (1-ck^{-\gamma})X_k - Y_k + Z_k \qquad \text{for all $k \geq k_0$}.
$$
Assume furthermore that $c \geq 6$ if $\gamma = 1$. Define the constants $a_k:=\frac{k^{2\gamma-1}}{\log^2(k+1)}$. Then there exists a random variable $V < \infty$ such that  on the event $\{\sum_{k=1}^\infty a_{k+1}Z_k < + \infty\}$, the following is true:
\begin{enumerate}
\item The limit holds
$$
 a_kX_k\xrightarrow{\text{a.s.}} V.
$$
\item The sum is finite
$$
  \sum_{k=1}^\infty a_{k+1}Y_k <  +\infty.
$$
\end{enumerate}
\end{lem} 

\begin{lem}\label{lem:sequencelemmadistance}
	Fix $k_0 \in \NN$, $c, C > 0$, and $\gamma \in (1/2, 1]$. Suppose that $\{s_k\}_{k}$ is a nonnegative sequence satisfying 
	\begin{align*}
	s_{k} \le  \frac{c}{12\gamma} \qquad \text{ and } \qquad s_{k+1}^2 \leq s_k^2 - c k^{-\gamma} s_k + Ck^{-2\gamma},  \qquad \text{for all $k \geq k_0$,}
	\end{align*}
Then, there exists a constant $C_{\ub}$ depending only on $c, C, \gamma$ and $k_0$ such that  
	$$
	s_{k} \leq C_{\ub}k^{-\gamma}, \qquad \forall k \ge 1.
	$$
\end{lem}

\begin{lem}\label{lem:sequencelemmasquared}
	Fix $k_0 \in \NN$, $c, C > 0$, and $\gamma \in (1/2, 1]$. Suppose that $\{s_k\}_{k}$ is a nonnegative sequence satisfying 
	\begin{align*}
 \qquad s_{k+1} \leq (1-ck^{-\gamma})s_k  + Ck^{-2\gamma},  \qquad \text{for all $k \geq k_0$,}
	\end{align*}
Assume furthermore that $c \ge 16$ if $\gamma=1$. Then, there exists a constant $C_{\ub}$ depending only on $c, C, \gamma$ and $k_0$ such that  
	$$
	s_{k} \leq C_{\ub}k^{-\gamma}, \qquad \forall k \ge 1.
	$$
\end{lem}

\end{document}